\documentclass{biometrika}
\usepackage{graphicx} 
\graphicspath{{./art/}}

\usepackage{amsmath,amsfonts,amssymb}

\usepackage{subcaption}
\usepackage{appendix}
\usepackage{float}
\usepackage{hyperref}
\usepackage[plain,noend]{algorithm2e}
\usepackage{braket}
\usepackage{placeins}

\makeatletter
\renewcommand{\algocf@captiontext}[2]{#1\algocf@typo. \AlCapFnt{}#2} 
\def\@algocf@capt@plain{top}
\renewcommand{\algocf@makecaption}[2]{%
  \addtolength{\hsize}{\algomargin}%
  \sbox\@tempboxa{\algocf@captiontext{#1}{#2}}%
  \ifdim\wd\@tempboxa >\hsize
    \hskip .5\algomargin%
    \parbox[t]{\hsize}{\algocf@captiontext{#1}{#2}}
  \else%
    \global\@minipagefalse%
    \hbox to\hsize{\box\@tempboxa}
  \fi%
  \addtolength{\hsize}{-\algomargin}%
}
\makeatother

\usepackage{tikz}

\DeclareMathOperator{\tkr}{%
\begin{tikzpicture}%
\node [inner sep=0pt,draw,circle,scale=0.5] (0,0) {KR};
\end{tikzpicture}%
}

\newcommand{\RR}{\mathbb{R}}

\newcommand{\EE}{\mathbb{E}}
\newcommand{\PP}{\mathbb{P}}

\newcommand{\calO}{\mathcal{O}}
\newcommand{\calC}{\mathcal{C}}
\newcommand{\calM}{\mathcal{M}}
\newcommand{\calE}{\mathcal{E}}
\newcommand{\calP}{\mathcal{P}}
\newcommand{\calJ}{\mathcal{J}}
\newcommand{\calD}{\mathcal{D}}
\newcommand{\calG}{\mathcal{G}}
\newcommand{\calL}{\mathcal{L}}

\newcommand{\hatU}{\widehat{U}}
\newcommand{\hatZ}{\widehat{Z}}
\newcommand{\hatX}{\widehat{X}}
\newcommand{\hatz}{\hat{z}}
\newcommand{\hatD}{\widehat{D}}
\newcommand{\hatB}{\widehat{B}}
\newcommand{\hatV}{\widehat{V}}

\newcommand{\tilZ}{\widetilde{Z}}
\newcommand{\tilU}{\widetilde{U}}
\newcommand{\tilD}{\widetilde{D}}
\newcommand{\tilB}{\widetilde{B}}
\newcommand{\tilC}{\widetilde{C}}
\newcommand{\tilz}{\tilde{z}}

\newcommand{\nmin}{n_{\min}}

\newcommand{\rank}{\mbox{rank}}
\newcommand{\SVD}{\mbox{SVD}}
\newcommand{\Cov}{\mbox{Cov}}
\newcommand{\diag}{\mbox{diag}}

 \newcommand {\Cksmall}{C_{K, \mathrm{small}}}
 
 \usepackage{newtxtext}
\usepackage[subscriptcorrection]{newtxmath}

\title{KRAFTY: Khatri-Rao Framework for Joint Cluster Recovery}
\author{Siyi Gao}
\affil{Department of Statistics, University of Virginia\\ Charlottesville, VA 22904, USA
\email{sg7nx@virginia.edu}}

\author{Zachary Lubberts}
\affil{Department of Statistics, University of Virginia\\ Charlottesville, VA 22904, USA
\email{zlubberts@virginia.edu}}

\author{Marianna Pensky}
\affil{Department of Mathematics, University of Central Florida\\ Orlando, FL 32816, USA
\email{Marianna.Pensky@ucf.edu}}

\date{}

\begin{document}

\maketitle

\begin{abstract}
When multiple datasets describe complementary information about the same set of entities, for example, brain scans of an individual over time, global trade network across years, or user information across social media platforms, integrating these snapshots allows us to see a more holistic picture. A common way of identifying structure in data is through clustering, but while clustering may be applied to each dataset separately, we learn more in the multi-view setting by identifying joint clusters. We consider a clustering problem where each view conflates some of these joint clusters, only revealing partial information, and seek to recover the true joint cluster structure. We introduce this multi-view clustering model and a method for recovering it: the transposed Khatri–RAo Framework for joinT cluster recoverY (KRAFTY). The model is flexible and can accommodate a variety of data-generating processes, including latent positions in random dot product graphs and Gaussian mixtures. A key advantage of KRAFTY is that it represents joint clusters in a space with sufficient dimension so that each joint cluster occupies an orthogonal subspace in the transposed Khatri-Rao matrix, which results in a sharp drop in the scree plot at the true number of joint clusters, enabling easy model selection.
Our simulations show that when the number of joint clusters exceeds the sum of the numbers of clusters in each individual view, our method outperforms existing methods in both joint clustering accuracy and estimation of the number of joint clusters. 

\end{abstract}


\section{Introduction}
\label{sec:introduction}

\subsection{Background and Motivation}
Community detection and clustering are core tasks in statistics and machine learning \cite{ClusterAnalysis,ClusteringMethods,abbe2023community,Fortunato2022}. Clustering analysis is defined as partitioning objects into groups (clusters) so that within-group similarity is high and between-group similarity is low; similarity may be defined by network connectivity, feature-space distances, or other affinities \cite{ClusterAnalysis,ClusteringMethods, vonluxburg2007}. Cluster analysis has been applied to data in several modern applications, such as genetic profiles \cite{toh2002} and social-network snapshots \cite{handcock2007}, to uncover meaningful latent patterns, and numerous techniques have been developed for single-dataset clustering \cite{ClusterAnalysis,ClusteringMethods,vonluxburg2007, Jin2023}.

Early work in the multi-view setting focused on recovery of a single, consistent clustering across sources \cite{abbe2023community, Fortunato2022, Jin2023}. More recently, multi-view clustering has gained attention \cite{Yang2018, MultilayerSBM} as many applications provide datasets from different sources or views, and joint clustering can yield a more informative partition. There has also been a growing interest in studying differences between subspaces across networks or clusterings \cite{macdonald2021, UASE, arroyo2020, agterberg2025}, leading to numerous multi-view spectral clustering approaches \cite{JingLi2021, Lei02102023}. A common multi-view spectral approach concatenates the per-view spectral embeddings and then performs a second spectral decomposition followed by clustering, called Multiple  Adjacency Spectral Embedding ({\bf MASE}) in  \cite{arroyo2020}. Such approaches have two drawbacks. First, rank deficiency: each view's embedding dimension is typically set to its own number of clusters, so concatenation yields a space of dimension at most the sum of per-view cluster counts, whereas the true number of joint clusters can be as large as the product across views. Second, determining the number of joint clusters: these methods usually lack an explicit estimator for the joint number of clusters, relying instead on an ``elbow'' in the scree plot, which may be ambiguous. To address these challenges, our study introduces the transposed Khatri-Rao product framework, which provides sufficient rank for estimating the number of joint clusters and achieves strong performance in both estimating the number of joint clusters and joint clustering.

\subsection{Notations}
For a positive integer $n$, we denote $[n] = \{1, \dots, n\}$. We denote by $\mathcal{O}_{r,k}$   the set of $r\times k$ matrices with orthonormal columns, where $\mathcal{O}_r=\mathcal{O}_{r,r}$. 
We denote the set of clustering matrices that partition $n$ objects into $K$ clusters by $\mathcal{C}_{n,K}$. Here, $\mathcal{C}_{n,K}$
consists of $n \times K$ binary matrices where each row has one value 1 and the rest are zeros.
We denote the $i$-th row and the  $j$-th column of a matrix $A$ by, respectively, $A(i,:)$ and $A(:,j)$.
We denote the identity matrix of order $m$ by $I_m$. We denote the matrix of $K$ left singular vectors of matrix $X$ by 
$\SVD_K(X)$.
We denote the Frobenius, spectral, and $2\rightarrow\infty$ matrix norms by, respectively,
$\|\cdot\|_F, \|\cdot\|$ and $\|\cdot\|_{2,\infty}$, where $\|A\|_{2,\infty}= \displaystyle \max_i \|A(i,:)\|$. 
We denote the Hadamard and the  Kronecker  products of matrices $A$ and $B$ by, respectively,
$A\circ B$ and $A\otimes B$.
For matrices $A  \in \RR^{n \times m}$ and $B\in \RR^{n\times q}$ with the same number of rows, we define the 
transposed-Khatri-Rao product  as $A\tkr B \in \RR^{n\times mq}$, where for $i\in[n],$
\begin{equation*} 
(A\tkr B)(i,:) = A(i,:) \otimes B(i,:)=[A(i,1)B(i,:)\,|\cdots|\,A(i,m)B(i,:)] .
\end{equation*}
Some properties of this product are provided in Appendix~\ref{sec:KR_properties}. We use $C_{\tau}$ for a generic constant that depends on $\tau$ only and can take different values throughout the paper.

\subsection{The Model} \label{sec:model}
In this paper, we consider the common scenario where data are obtained from $V$ sources (views). Each of these $V$ views is associated with a clustering assignment $z_v:\ [n] \to K_v$
with associated clustering matrix $Z_v \in \{0,1\}^{n \times K_v}$, where $K_v$ is the number of clusters in view $v \in [V]$. 
These clustering assignments may have any joint configuration. 
Our goal is to recover the joint clustering assignment function $z: [n] \to [K]$ or  associated
clustering matrix $Z \in \{0,1\}^{n \times K}$ that describes the joint configuration of clusters. To make this precise, let the total number of clusters $K$ in the  joint assignment be a value
\begin{equation} \label{eq:K_joint}
\max_{v \in [V]}\, K_v \leq K \leq \prod_v K_v,
\end{equation}
and let $Z\in \{0,1\}^{n\times K}$ be the true joint clustering matrix, which has exactly one nonzero entry in each row. The individual clustering matrices $Z_v=ZP_v$, where $P_v\in\{0,1\}^{K\times K_v}$ has at least one nonzero entry in each column, and exactly one nonzero entry in each row. This ensures that all of the individual clusterings are consistent with the joint clustering, but may combine some of the joint clusters together. Typically, all views will have fewer clusters individually than the true number of joint clusters.

In what follows, we assume that $V$  is small and data consist  of either view-specific 
estimated clustering matrices $\hatZ_v \in \calC_{n,K_v}$, $v \in [V]$,
or  estimated matrices of singular vectors $\hatU_v \in \calO_{n,K_v}$, $v \in [V]$, that can be used for clustering.
We make no additional assumptions on the sources of the data, however, a common scenario would be the results of $k$-means applied to the $V$ different views.

While the problem initially seems easy, in reality, it is not. Consider a simple situation where 
$V=2$. Let $K_1=K_2 = 3$, let $n$ be large, and let clusters in both views be generated uniformly at random with probability $1/3$.
Then, the clustering assignments $z_1$ and $z_2$ are independent and the combined clustering consists of $K_1 \times K_2=9$
groups of almost equal sizes.  If one uses a concatenated matrix $\tilZ = [Z^{(1)}|Z^{(2)}] \in \{0,1\}^{n \times 6}$,
where $Z^{(1)}, Z^{(2)}$ are individual clustering matrices,  as a basis for joint clustering, then $\rank(\tilZ) = 5$. 
For example, if $n = 1000$, the singular values of $\tilZ$  are (25.8852, 18.7956, 18.5849, 18.2198,  17.3010, 0). 
Hence, the matrix $\tilZ$ is rank-deficient, and it is highly non-trivial to choose the number of clusters in the joint assignment.

As a remedy, here we propose to replace $\tilZ$ by the matrix $Z^{(1,2)}=Z^{(1)}\tkr Z^{(2)}$.
It is easy to notice that matrix $Z^{(1,2)} \in \calC_{n,K_1 K_2}$, so it is a clustering matrix 
and it has 9 almost identical singular values  (11.1803, 11.0454, 10.9087, 10.7238, 10.6301, 10.5830, 10.3923, 9.8489, 9.4340).
We shall show later that, in general, the rank of matrix $Z^{(1,2)}$ can be used as a proxy for the number of clusters in the 
joint assignment, and its estimated version (as well as its version based on estimated matrices of singular values)
forms a solid foundation for the recovery of the combined clustering assignment.

As such, our paper   introduces the transposed Khatri-RAo Framework for joinT cluster recoverY (KRAFTY).   
A key advantage of KRAFTY is that it represents joint clusters in a space with sufficient dimension, 
which enables more reliable estimation of the number of joint clusters.
Figure~\ref{fig:KR_example} illustrates the key idea behind application of the transposed Khatri-Rao product. Here, $Z^{(1)}$ and $Z^{(2)}$ are clustering matrices for two datasets over the same set of entities.
In this example, $V=2$ and both views have $K_1 = K_2 = 3$ clusters. Cluster  1 is common for both views while 
elements of clusters 2 and 3 in the  view 1 can be scattered among clusters 2 and 3 in the view 2. 
This leads to a joint configuration with $K=5$ different clusters. Matrix $Z^{(1,2)} = Z^{(1)}\tkr Z^{(2)}$
has $3 \times 3 = 9$ columns, where the column indexed by $(k_1, k_2)$,   corresponds to the combination of clusters $k_1$ and $k_2$ in the views 1 and 2, respectively. It is easy to see that  
the only non-zero columns of  $Z^{1,2)}$ correspond to the non-empty combinations of clusters, so that 
it has 4 zero columns.  Again, if $n=1000$, matrix $Z^{1,2)}$ has 5 large singular values 
$(18.7617, 13.1149, 12.8452, 12.4900, 12.4499)$, followed by 4 zeros.
 
\begin{figure}[htb!]
    \centering
    \includegraphics[width=10.5cm,height=6.5cm]{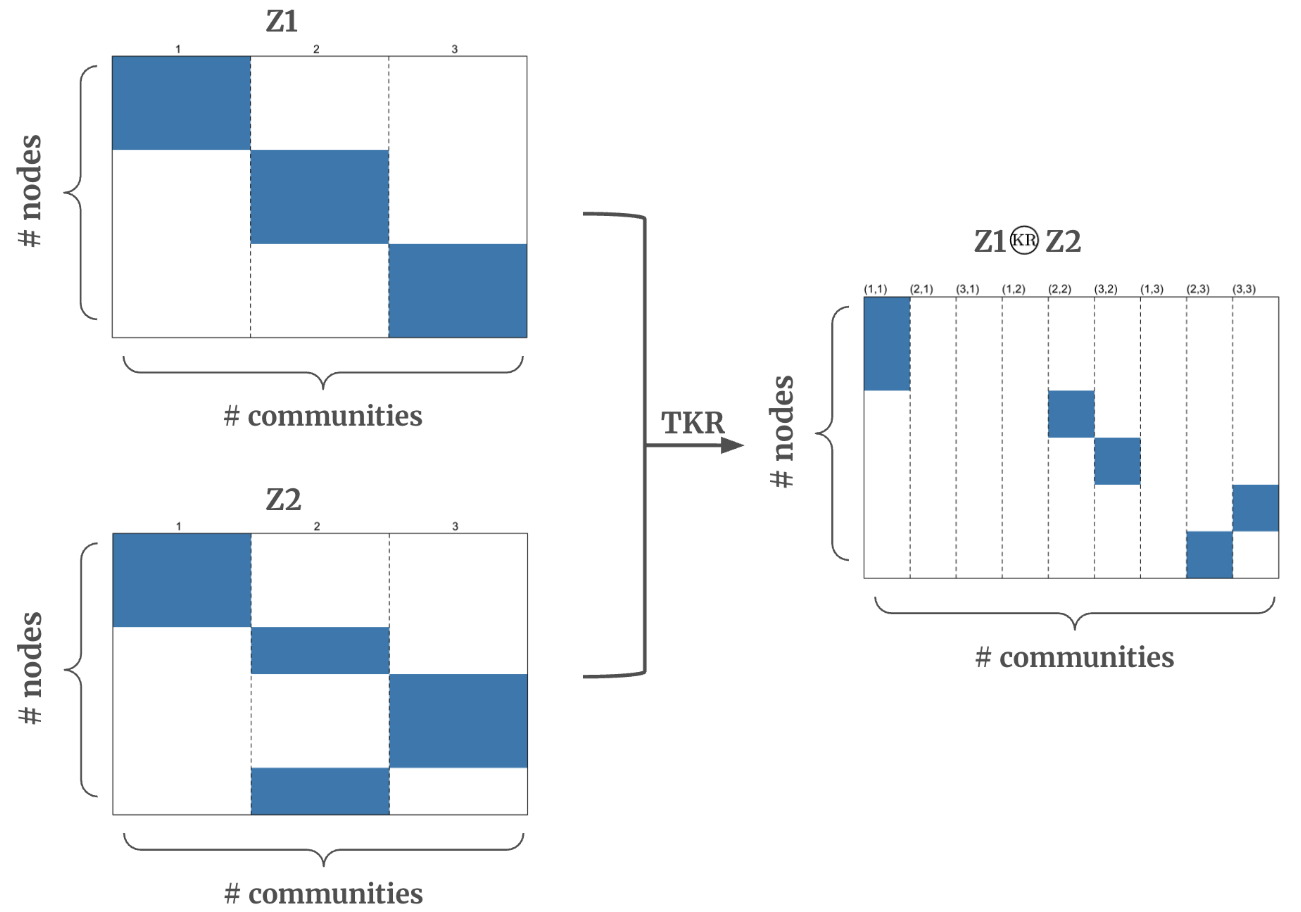}
    \caption{Graphical illustration of the intuition behind the transposed Khatri-Rao product.}
    \label{fig:KR_example}
\end{figure}


\subsection{Related Work}
In recent years, community-based statistical models have gained significant attention because they treat clusters not merely as features of the data, but as the generative mechanism behind it \cite{Fortunato2022}. Prominent examples include Gaussian mixture models (GMMs) and the stochastic block model (SBM) \cite{Holland1983}. These models have led to a rich literature on community detection in single dataset \cite{abbe2023community, Fortunato2022}, and a collection of spectral clustering methods (\cite{Jin_2015, vonluxburg2007, qin2013regularized, Rohe2011, Lei02102023, loffler2020}). The core idea is to use a matrix factorization, such as the eigen-decomposition or singular value decomposition, to extract a lower-dimensional embedding, which is then clustered using methods such as $k$-means or hierarchical clustering \cite{vonluxburg2007}. These methods perform well in practice and have strong theoretical guarantees: under suitable conditions they are consistent, and can even achieve perfect clustering (\cite{Lei2015, lyzinski2015, loffler2020, pensky2024daviskahan, sussman2012}). Our method builds on the nice properties of single-source spectral clustering by exploiting the spectral embeddings from each data set.

Moving to multiview clustering, a number of recent studies have been conducted in the multilayer SBM framework (\cite{MultilayerSBM}). Several clustering methods have been developed, including spectral clustering methods (\cite{JingLi2021, Lei02102023, agterberg2025, DeDomenico2014, han2015}) and matrix factorization (\cite{paul2018}). Many methods involve obtaining spectral embeddings from each graph and then concatenating them for a subsequent spectral decomposition, as in MASE \cite{arroyo2020} and UASE \cite{UASE}. However, these side-by-side concatenation face limitations when the number of joint clusters exceeds the sum of the per-layer cluster counts: because each layer's embedding dimension is typically set to its own number of clusters, the maximum rank of the concatenated matrix is at most that sum, so the representation can become rank-deficient and fail to capture the full joint structure. Another important line of work is to estimate the number of clusters in multi-view data. Existing methods often rely on visual inspection of the scree-plot ``elbow'' (\cite{Ledesma_Valero-Mora_Macbeth_2015}) or the automatic dimension selection method by \cite{ZHU2006918}. A common problem is that the eigenvalues may not exhibit a clear drop, making the dimension hard to determine even for automatic procedures. Driven by these challenges, our work introduces the transposed Khatri-Rao product framework, which provides sufficient rank for estimating the number of joint clusters and a clear ``elbow''  exactly at the joint cluster count.
\\

This rest of this paper is organized as follows. Section~\ref{sec:Methodology} introduces our model and presents the proposed transposed Khatri-Rao product method. In Section~\ref{sec: Theory}, we study the theoretical statistical performance of our proposed method. We evaluate the empirical performance of our method in clustering and estimating the number of joint clusters, and compare it with MASE in Section~\ref{sec: Simulation}. Section~\ref{sec: Real-data} demonstrates the application to FAO food and agricultural trade data, showing the ability of our approach to uncover joint clusters of countries in the global trade network.

\section{Methodology}
\label{sec:Methodology}

\subsection{KRAFTY Joint Clustering Methodology}

Since $V$ is a small integer, for simplicity, we consider the case of $V=2$.
Before launching into a description of the methodology, we closely examine  
the relationship between individual clustering matrices $Z^{(1)}$, $Z^{(2)}$, their transposed Khatri-Rao 
product
\begin{equation} \label{eq:TKR_Z}
    Z^{(1,2)} = Z^{(1)} \tkr Z^{(2)} \in \{0,1\}^{n \times K_1 K_2}
\end{equation}
and the true joint clustering matrix $Z \in \calC_{n,K}$.

As it is easy to see from Section~\ref{sec:model} and Figure~\ref{fig:KR_example}, the matrix $Z$ can be obtained
from the matrix $Z^{(1,2)}$  by removing its zero columns. The latter can be accomplished by introducing a matrix 
$H \in \{0,1\}^{K_1 K_2 \times K}$, where $H$ is formed by the columns of the identity matrix of order $K_1 K_2$ corresponding to the nonzero columns of $Z^{(1,2)}$.   
Therefore, with some abuse of notations,  for $k_1 \in [K_1]$, $k_2 \in [K_2]$ and $k \in [K]$, one has 
$H ((k_1,k_2),:) = 0$   iff  $Z^{(1,2)}(:,(k_1,k_2)) = 0$, and each column of matrix $H$ has one element equal to 1.
These definitions imply the following relationships
\begin{equation} \label{eq:TKR}
Z  = Z^{(1,2)} H, \quad Z^{(1,2)} = Z H^\top, \quad H^\top H = I_K, \ \ H \in \calO_{K_1 K_2,K}.
\end{equation}
Here, $K$ satisfies conditions \eqref{eq:K_joint}, so that $K \leq K_1 K_2$. 
In addition, $Z \in \calC_{n,K}$, so it is a clustering matrix.

%
\begin{algorithm} 
\caption{\ Joint Clustering with KRAFTY}
\label{alg:KRAFTY}
\begin{flushleft} 
{\bf Input:} $\hatZ^{(v)} \in \calC_{n,K_v}$ or  $\hatU^{(v)} \in \calO_{n,K_v}$, $v=1,2$;  the final number of clusters $K$;   parameter $\epsilon > 0$ \\
{\bf Output:} Estimated joint  clustering function $\hatz : [n] \to [K]$ 
\\
{\bf Steps:}\\
{\bf 1:} Obtain matrix $\hatU$ using formula \eqref{eq:hatU_from_hatZ} or \eqref{eq:hatU_from_hatU} \\
{\bf 2:}  Cluster  $n$ rows of  $\hatU$ into $K$ clusters using   $(1+\epsilon)$-approximate $k$-means clustering to obtain the 
estimated clustering function $\hatz.$   
\end{flushleft} 
\end{algorithm}
%

If we knew the true matrices $Z^{(1)}$ and  $Z^{(2)}$, then matrix $Z$ in \eqref{eq:TKR} 
would provide the joint clustering assignment. However, we are given either their estimated versions
 $\hatZ^{(v)} \in \calC_{n,K_v}$,   or estimated matrices of 
singular vectors $\hatU^{(v)} \in \calO_{n,K_v}$, $v=1,2$,   
that can be used for finding individual clustering assignments. 
In order to design a unified clustering procedure, consider diagonal matrices 
\begin{equation} \label{eq:diag}
D^{(1,2)} = \left(Z^{(1,2)}\right)^\top Z^{(1,2)}, \quad D = Z^\top Z,
\end{equation}
and observe that, due to \eqref{eq:TKR}, one has  $D = H^\top D^{(1,2)} H$. 
Here, $D(k,k) = n_k$, the number of elements in the joint cluster $k \in [K]$.
Then it follows from  \eqref{eq:TKR_Z} and  \eqref{eq:diag}  that 
\begin{equation}  \label{eq:joint_U}
    U = Z D^{-1/2} =Z^{(1,2)} H D^{-1/2}\in \calO_{n,K}  \quad \Longrightarrow \quad
Z^{(1,2)} = \left(Z^{(1)} \tkr Z^{(2)}\right) = U D^{1/2} H^\top,
\end{equation}
 and the second equality gives the SVD of the matrix 
 $Z^{(1,2)}$.  Here  rows of $U$ are distinct and correspond to $K$ clusters, hence the clustering assignment 
$z: [n] \to [K]$  can be obtained by clustering rows of the matrix $U$. Moreover, $U$ can be obtained as the top $K$ left singular vectors
of $Z^{(1,2)}$, i.e., $U = \SVD_K \left(Z^{(1,2)}\right)$.
When one has  $\hatZ^{(v)}$  instead of $Z^{(v)}$, $v=1,2$, the matrix $U$ can be replaced by 
\begin{equation} \label{eq:hatU_from_hatZ}
\hatU = \SVD_K \left(\hatZ^{(1,2)}\right), \quad   \hatZ^{(1,2)} =  \left(\hatZ^{(1)} \tkr \hatZ^{(2)} \right). 
\end{equation}

Now, consider the case where data consist of matrices  
\begin{equation*} 
U^{(v)} = Z^{(v)} \left(D^{(v)}\right)^{-1/2} \in \mathcal{O}_{n,K_v}, \quad D^{(v)} = (Z^{(v)})^T Z^{(v)},\ \   v=1,2.
\end{equation*}
Then, due to property {\bf P1} of Khatri-Rao product (see Section~\ref{sec:KR_properties}) and properties of the Kronecker product, one has
$$
U^{(1)} \tkr U^{(2)}  = \left(Z^{(1)} (D^{(1)})^{-1/2}\right) \tkr \left(Z^{(2)} (D^{(2)})^{-1/2}\right) = Z^{(1,2)} \left(D^{(1)} \otimes D^{(2)} \right)^{-1/2}
$$
Hence, it follows from \eqref{eq:joint_U} that 
$U$ and $U ^{(1)} \tkr U ^{(2)}$  are related as 
\begin{equation*} 
 U^{(1,2)}:=    U^{(1)} \tkr U^{(2)}   
= U D^{1/2} H^\top (D^{(1)} \otimes D^{(2)})^{-1/2}.
\end{equation*}
The right-hand side of this equation is not an SVD of $U^{(1,2)}$. In order to obtain an SVD of $U^{(1,2)}$
consider a diagonal matrix $\tilD$ which places $n_{k_1}^{(1)}n_{k_2}^{(2)}$ in the $k$th diagonal entry when the $k$th joint cluster corresponds to clusters $k_1$ and $k_2$ in the individual views:
\begin{equation*} 
    \tilD = H^\top \left(D^{(1)} \otimes D^{(2)}\right) H \in \RR^{K \times K}.
\end{equation*}
$\widetilde{D}  \neq D$ due to 
$D^{(1,2)} \neq D^{(1)} \otimes D^{(2)}.$ We have that  
$H^\top \left(D^{(1)} \otimes D^{(2)}\right)^{-1/2} = \tilD^{-1/2} H^\top$,
and therefore
\begin{equation} \label{eq:U12_SVD}
   U^{(1,2)} =  (U^{(1)} \tkr U^{(2)}) = U \, D^{1/2} \widetilde{D}^{-1/2}  H^\top.   
\end{equation}
Since the matrix $D^{1/2} \widetilde{D}^{-1/2}$ is diagonal, 
this equation provides an SVD of 
$U^{(1,2)}$ 
with the matrix  $U$ of the left singular vectors, i.e., $U = \SVD_K \left( U^{(1,2)} \right).$
Again, if matrices $\hatU^{(v)}$ are available instead of $U^{(v)}$, $v=1,2$, 
the matrix $U$ can be replaced by $\hatU$ where 
\begin{equation} \label{eq:hatU_from_hatU}
\hatU = \SVD_K \left(\hatU^{(1,2)}\right), \quad \hatU^{(1,2)} = \left(\hatU^{(1)} \tkr \hatU^{(2)} \right). 
\end{equation}

 A visual comparison of the results of using $\hatU^{(v)}$ or $\hatZ^{(v)}$ as inputs to KRAFTY is provided in Appendix Figure~\ref{fig:KR-algo2}.
 
\subsection{Joint Clustering Algorithms}

To obtain a joint clustering from the transposed Khatri-Rao matrix, we can simply apply a clustering algorithm to the rows of $\hat{U}$ obtained via Equation~\eqref{eq:hatU_from_hatZ} or \eqref{eq:hatU_from_hatU}. This is described in Algorithm~\ref{alg:KRAFTY}.

As mentioned above, another way of obtaining joint clusters is to concatenate the matrices $\hatZ^{(v)}$ or $\hatU^{(v)}$,
$v=1,2$. We refer to this method as MASE since it can be viewed as a version of  MASE in  \cite{arroyo2020} .
The steps are provided in Algorithm~\ref{alg:MASE}.


%
\begin{algorithm}
\caption{\ Joint Clustering with MASE}
\label{alg:MASE}
\begin{flushleft}  
{\bf Input:} $\hatZ^{(v)} \in \calC_{n,K_v}$ or  $\hatU^{(v)} \in \calO_{n,K_v}$, $v=1,2$;  the final number of clusters $K$;   
parameter $\epsilon > 0$ \\
{\bf Output:} Estimated joint  clustering function $\hatz : [n] \to [K]$ 
\\
{\bf Steps:}\\
{\bf 1:} Construct matrix $\tilZ = \left[\hatZ^{(1)}\, | \, \hatZ^{(2)}\right]$ or  
$\tilU = \left[\hatU^{(1)}\, | \, \hatU^{(2)}\right]$ \\
{\bf 2:}  Find $\hatU = \SVD_K \left( \tilZ \right)$ or $\hatU = \SVD_K \left( \tilU \right)$ \\
{\bf 3:}   Cluster  $n$ rows of  $\hatU$ into $K$ clusters using   $(1+\epsilon)$-approximate $k$-means clustering to obtain the
estimated clustering function $\hatz.$   
\end{flushleft} 
\end{algorithm}
%


Section~\ref{sec: Simulation}  will provide a detailed comparison between Algorithms~\ref{alg:KRAFTY}~and \ref{alg:MASE}.
Here, we shall make just a few remarks. While the matrix $Z$ is a clustering matrix, $\tilZ$  is not. 
In addition, while $\rank(Z^{(1,2)}) = \rank (U^{(1,2)})$ is equal to the number of clusters, this is not true for  $\tilZ = [Z^{(1)}| Z^{(2)}]$ or $\tilU = [U^{(1)}| U^{(2)}]$. The matrices $\tilZ$ and  $\tilU$ have ranks at most $K_1 + K_2$, so they can be rank deficient when $K > K_1 + K_2$, as we have shown in Section~\ref{sec:model}. 

The clustering algorithm in Algorithms~\ref{alg:KRAFTY}~and \ref{alg:MASE} is suggested to be $(1+\epsilon)$-approximate $k$-means clustering, but this is not the only possible choice. For example, one can use hierarchical clustering or any other appropriate clustering algorithm. In all simulations and real-data analysis in this paper, we use hierarchical clustering, since we found that it produced more stable results under repeated testing than either $k$-means or Gaussian mixture modeling.



\begin{remark}  \label{rem:Vnot2}
{\bf The case of $V \geq 3$.\ }{\rm 
While this section only presents the case of $V=2$, it is fairly straightforward to extend the KRAFTY
methodology    to the case of $V \geq 3$. Specifically, for Algorithm~\ref{alg:KRAFTY}, 
$\hatZ^{(1,2)}$ in formula \eqref{eq:hatU_from_hatZ}
and $\hatU^{(1,2)}$  in \eqref{eq:hatU_from_hatU} 
should be replaced by, respectively,  
$$
\hatZ^{(1,...,V)} = \left(\hatZ^{(1)} \tkr \ldots \tkr \hatZ^{(V)} \right), \quad 
 \hatU^{(1,...,V)} = \left(\hatU^{(1)} \tkr \ldots \tkr \hatU^{(V)} \right).
$$
If one uses  Algorithm~\ref{alg:MASE}, then in Step~1,  
$$
\tilZ = \left[\hatZ^{(1)}\, | \ldots\, | \, \hatZ^{(V)}\right], \quad
\tilU = \left[\hatU^{(1)}\, | \ldots\,  | \, \hatU^{(2)}\right].
$$
 Since $\rank(\tilZ) \leq K_1 + \cdots + K_V$ is possibly much smaller than $K$, the latter again will result in rank-deficient clustering.}
\end{remark}


\subsection{A Common Data-Driven Example.}
 
 
While in the paper we impose no assumptions on  the data-generating process, one of the most common settings is the 
$k$-means scenario, where  the observed data in view $v$ are given by matrix  
\begin{equation}  \label{eq:k-means-data}
Y^{(v)} = G^{(v)} + \Xi^{(v)}  \in \RR^{n \times d_v}, \quad v \in [V],
\end{equation}
where  matrix $G^{(v)}$ has $K_v$ distinct rows and $\EE(\Xi^{(v)})=0$.
The latter means that there exist clustering functions $z_v:[n]\to[K_v]$ such that 
the $i$-th row of $G^{(v)}$ is $G^{(v)}(i,:) = \mu^{(v)}_k$, 
$i \in [n]$, $k \in [K_v]$.  If $\calM^{(v)} \in \RR^{K_v  \times d_v}$ with rows $\mu^{(v)}_k$, then in matrix form we have
\begin{equation}  \label{eq:k-means-mean}
G^{(v)}  = Z^{(v)}\,  \calM^{(v)},  \quad v=1,2.
\end{equation}
In the above situation, one obtains matrices $\hatU^{(v)} = \SVD_{K_v} (Y^{(v)})$,
and possibly clusters rows of $\hatU^{(v)}$  to obtain matrices  $\hatZ^{(v)}$, $v \in [V]$.
Matrices $Y^{(v)}$ and $G^{(v)}$  in \eqref{eq:k-means-data} and \eqref{eq:k-means-mean} can be of 
any origin. For example, rows $Y^{(v)}(i,:)$  of $Y^{(v)}$  can be independent,
with covariance matrices determined by $k= z_v(i)$, so that 
$\Cov \left(Y^{(v)}(i,:)\right) = \Sigma^{(v)}_{k}$, or they can be adjacency matrices of random graphs.

To further illustrate how the transposed Khatri-Rao product is used in recovering joint clusters, 
in Figure~\ref{fig:KR-algo}, we show a visual representation of the application of our method to two adjacency matrices generated from the stochastic block models with three underlying clusters.  
We assume that the joint clustering structure follows Figure~\ref{fig:KR_example}, so that there are five joint clusters. 
We first perform the SVD on each of them and extract matrices $\hatU^{(v)}, v=1,2$.
Then, we construct $\hatU^{(1,2)}$  using formula \eqref{eq:hatU_from_hatU} and obtain 
$\hatU$ as the matrix of left singular vectors of $\hatU^{(1,2)}$.
The resulting scree plot of singular values of  $\hatU^{(1,2)}$ shows a clear gap 
after the fifth singular value, indicating $K=5$ latent joint clusters. 
Finally, the joint clusters are recovered by applying the $k$-means algorithm with $K=5$ to the 
five left leading singular vectors of the matrix $\hatU^{(1,2)}$.

\begin{figure}[htb!]
    \centering
    \includegraphics[width=\linewidth]{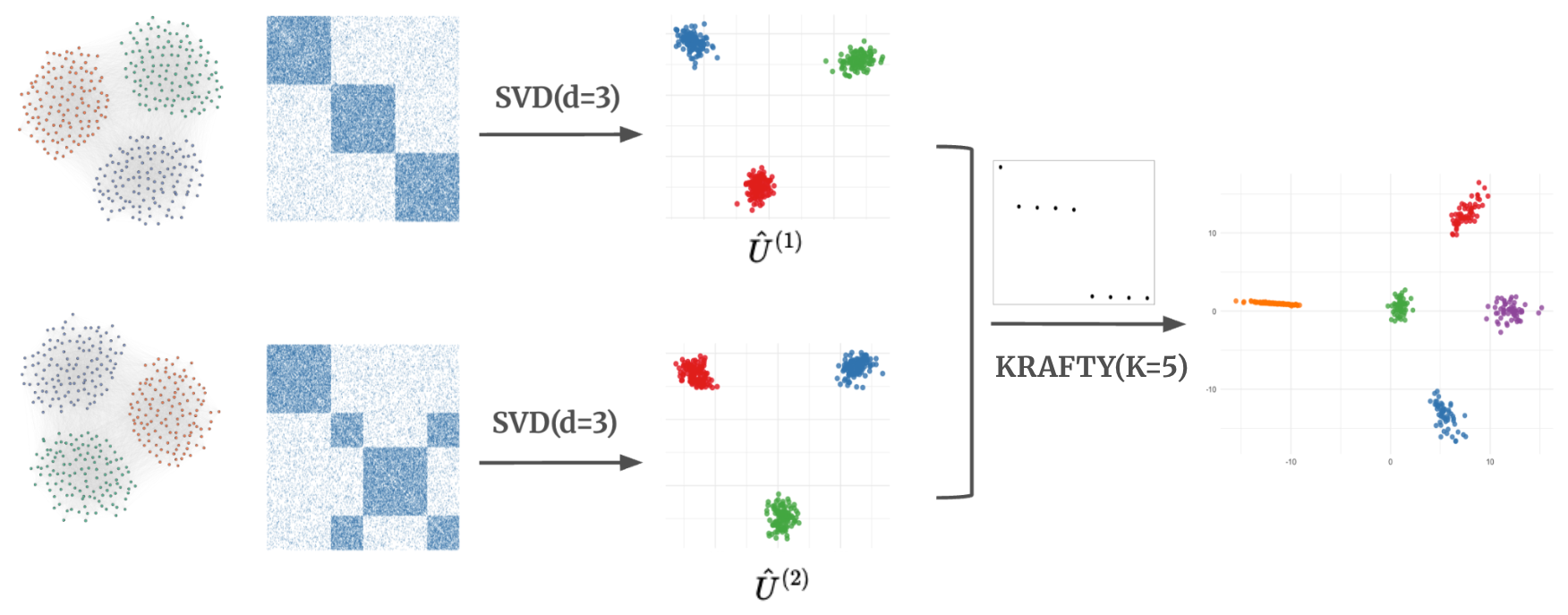}
    \caption{Graphical representation for a common data-driven example.}
    \label{fig:KR-algo}
\end{figure}

\section{Theoretical Results}
\label{sec: Theory}

\subsection{Assumptions}
\label{sec:assump}

In this section, we present guarantees for joint cluster recovery using KRAFTY (Algorithm~\ref{alg:KRAFTY}).
As before, for simplicity, we consider the case of $V=2$, and denote the true clustering matrices
by $\hatZ^{(v)}$, $v=1,2$. We assume that data come in the form of $\hatZ^{(v)}$ or $\hatU^{(v)}$, $v=1,2$. 
In the former case, let  $J_v \subseteq \{1,\dots,n\}$  denote the set of elements  misclassified by 
$\hatZ^{(v)}$ and let $\calE_{v,n}=  |J_v|$  be its cardinality.  Similarly, we denote the true clustering matrix by $Z$ and its 
estimated version by $\hatZ$. Let $J  \subseteq \{1,\dots,n\}$  denote the set of elements  misclassified by 
$\hatZ$ and let $\calE_{n} =|J|$  be the number of clustering errors.

If data are given in the form of  $\hatU^{(v)}$,  $v=1,2$, we shall place upper bounds on the norms of the differences between 
$U^{(v)}$ and $\hatU^{(v)}$. Since matrices $\hatU^{(v)}$ are defined only up to a rotation, we consider matrices 
$W_U^{(v)}$ that deliver the closest match between $U^{(v)}$ and $\hatU^{(v)}$ in Frobenius norm.
In particular, if $(U^{(v)})^\top \hatU^{(v)} = W_1 D_W W_2^T$ is the SVD of  $(U^{(v)})^\top \hatU^{(v)}$, then 
$W_U^{(v)} = W_1 W_2^T \in \calO_{K_v}$.

We say that clustering by $\hatZ$ is {\it consistent} if $n^{-1}\, \calE_{n} \to 0$ as $n \to \infty$.
If  $\calE_{n} \to 0$ as $n \to \infty$, we say that clustering is  {\it perfect}. In what follows, we show that 
under straightforward regularity assumptions, clustering precision by  Algorithm~\ref{alg:KRAFTY} is as good as individual clusterings,
and, in addition, one can correctly identify the number of joint clusters.
We begin by stating some regularity assumptions.

\begin{assumption}[Balanced cluster sizes] 
\label{assump:balancedclusters}
Both the single-view and joint clusterings are assumed to have balanced cluster sizes:
if $n^{(v)}_k$ is the number of elements in cluster $k$ in view $v$, $k \in [K_v]$, and $n_k$ is the number of elements in the joint cluster $k \in [K]$,
then for some positive constants, one has  
\begin{align}
    \left(c_o^{(v)}\right)^2 \,\frac{n}{K_v}
    & \leq \min_k n^{(v)}_k
    \leq \max_k n^{(v)}_k
    \leq\left(C_o^{(v)}\right)^2\, \frac{n}{K_v},\label{eq:balanced_view}\\
      c_o^{2}\,\frac{n}{K}
    & \leq \min_k n_k
    \leq \max_k n_k
    \leq C_o^{2}\,\frac{n}{K}. \label{eq:balanced_joint}
\end{align}
\end{assumption}
%
\begin{assumption} [Finite number of clusters] 
\label{assump:finiteclusternumbers}
The number of clusters in each view and, therefore, the number of joint clusters are finite
\begin{equation} \label{eq:clust_sizes}
    K_1 = O(1), \quad K_2 = O(1), \quad K = O(1), \quad n \to \infty.
\end{equation}  
\end{assumption}
%
\begin{assumption} [Consistent clustering in each view]
\label{assump:consistentZs}
If data are in the form of $\hatZ^{(v)}$, $v=1,2,$ then the number of clustering errors $\calE_{v,n}$ on the basis of $\hatZ^{(v)}$ is such that 
\begin{equation} \label{eq:calEvn}
    n^{-1}\, \calE_{v,n}   \;\to\; 0, 
    \quad \text{as  } n\to\infty.
\end{equation}    
\end{assumption}
%
\begin{assumption} [Consistent subspace estimation] 
\label{assump:consistentUs}
If data are in the form of $\hatU^{(v)}$, $v=1,2,$  then, for some non-random quantities 
$\calE_{on}^{(v)}$ and $\calE_{2,\infty,n}^{(v)}$, $v=1,2$  and an arbitrary  constant $\tau >0$, one has
\begin{align}   \label{eq:Eopvn}   
    & \PP \left\{ \left\| \hatU^{(v)}- U^{(v)} W_U^{(v)}\right\|  \leq C_{\tau}\, \calE_{0,n}^{(v)}\right\} 
       \geq 1 - n^{-\tau},\\
  \label{eq:E2infvn}
    & \PP \left\{ \left\| \hatU^{(v)}- U^{(v)} W_U^{(v)}\right\|_{2,\infty} \leq C_{\tau}\, \calE_{2,\infty,n}^{(v)}\right\} 
       \geq 1 - n^{-\tau},    
\end{align} 
where the constant $C_{\tau}$  depends on $\tau$ only.
\end{assumption}
%
\begin{assumption} [Incoherence of $\hatU^{(v)}$] 
\label{assump:incoherentUhats}
For $v=1$ or $v=2$ and a constant $C_{\tau}$ that depends on $\tau$ only, one has
\begin{equation}
    \PP \left\{ \left\| \hatU^{(v)}\right\|_{2,\infty} \leq 
C_{\tau}\, \frac{\sqrt{K_v}}{\sqrt{n}}\right\} 
       \geq 1 - n^{-\tau}.
    \label{eq:hatUv_2inf}  
\end{equation}
\end{assumption}
%

The assumptions above are quite common. Assumption~\ref{assump:balancedclusters} allows cluster sizes to vary up to constant factors.
If the clusters are highly unbalanced, then it is very hard to distinguish between an erroneous cluster and 
a genuine very small cluster. Assumption~\ref{assump:finiteclusternumbers} states that we are interested in the setting where the number of 
clusters in each view is fixed and does not grow with $n$. Assumption~\ref{assump:consistentZs} addresses the situation where data 
come in the form of estimated clustering matrices, such  that estimated individual clustering assignments are consistent with high probability. Assumption~\ref{assump:consistentUs} handles the case where the data appear as 
estimated matrices of singular vectors. It postulates that, with high probability, the spectral and the two-to-infinity norms of the differences are bounded above by some non-random quantities. The only condition that appears somewhat unusual is Assumption~\ref{assump:incoherentUhats}. However, this assumption 
can be satisfied in a variety of ways. It is easy to see that \eqref{eq:hatUv_2inf} holds if 
$\calE_{2,\infty,n}^{(v)} \leq C\, K_v^{1/2}\, n^{-1/2}$. The latter is true, for 
example, if one is dealing with the $k$-means model 
\eqref{eq:k-means-data}, where the error matrix $\Xi^{(v)}$ has independent sub-gaussian entries.
For instance, the following statement follows directly from Theorem~3.1 of \cite{Xie-AOS2025}.

\begin{lemma}[Incoherence of matrix $\hatU^{(v)}$]  
Consider model \eqref{eq:k-means-data} with $G^{(v)}$ given by \eqref{eq:k-means-mean}. Let $d_v \geq K_v$ and the ratio between the smallest and the largest nonzero eigenvalues of $\calM^{(v)}$ is bounded below by a constant, i.e.
$$
\sigma_{K_v} (G^{(v)}) \Big/ \sigma_1 (G^{(v)})  \geq C_{\sigma}.
$$
Suppose the elements of the matrix $\Xi^{(v)}$ in \eqref{eq:k-means-data} 
are independent sub-gaussian or sub-exponential random variables, so that, for any $t>0$ and some absolute constant $C>0$, one has
$$
\PP \left\{ \sigma^{-1} |\Xi(i,j)| >t \right\} \leq C\, \exp(-t^{a}),
$$
where $a=2$ for the sub-gaussian and $a=1$ for the  sub-exponential 
case. Then, condition \eqref{eq:hatUv_2inf}  in Assumption~\ref{assump:incoherentUhats} is satisfied provided
$$
 \sigma \sqrt{n} (\log n)^{1 + \xi} \Big/ \sigma_{K_v} (G^{(v)}) = O(1)
$$
for some $\xi >0$.
\label{lem:incoherence}
\end{lemma}

Another way to ensure validity of  Assumption~\ref{assump:incoherentUhats} is to modify Algorithm~\ref{alg:KRAFTY}
by adding a regularization step. Define a regularization operator $\calP_{\delta} (U, K)$ as
\begin{equation*} 
\calP_{\delta} (U) = \SVD_K (U_{*})\quad \mbox{where} \quad
U_{*} (i,:):= U (i,:)\,  \cdot \min \left\{ \delta, \|U (i,:)\| \right\} \Big/ \|U (i,:)\|, \quad 
i \in [n]. 
\end{equation*}  
Before applying Algorithm~\ref{alg:KRAFTY},
replace one of the matrices $\hatU^{(v)}$ by $\hatU_{\delta}^{(v)}$, 
\begin{equation*} 
\hatU_{\delta}^{(v)} = \calP_{\delta_v} \left(\hatU^{(v)}, K_v \right), \quad \delta_v = C_v\, K_v^{1/2}\, n^{-1/2}, \quad v=1,2,
\end{equation*}
 where $C_v$ is a sufficiently large  constant. 
Subsequently, obtain matrix $\hatU$ in Step~1 of Algorithm~\ref{alg:KRAFTY}  using formula  
\begin{equation*} 
\hatU = \SVD_K \left(\hatU^{(1,2)}\right), \quad \hatU^{(1,2)} = \left(\hatU_{\delta_1}^{(1)} \tkr \hatU^{(2)} \right) \ \ 
\mbox{or}\  \ \hatU^{(1,2)} = \left(\hatU^{(1)} \tkr \hatU^{(2)}_{\delta_2} \right)
\end{equation*}
instead of  \eqref{eq:hatU_from_hatU}. In this case, $\hatU_{\delta_v}^{(v)}$ will satisfy Assumption~\ref{assump:incoherentUhats}, and
together with Assumption~\ref{assump:balancedclusters}, Equation~\eqref{eq:E2infvn} holds for this view with 
$\calE_{2,\infty,n}^{(v)} = n^{-1/2}\, K_v^{1/2}$.
In addition, by properties of regularization (see, e.g., \cite{kim2018sbmhypergraphs}
or \cite{kivela2014multilayer}), one has 
\begin{equation*} 
\left\| \hatU^{(v)}_{\delta_v} - U^{(v)} W_{U, \delta}^{(v)}\right\| \leq 
 2 \sqrt{2}\, \left\| \hatU^{(v)}- U^{(v)} W_U^{(v)}\right\|,  
\end{equation*}
so Equation~\eqref{eq:Eopvn} holds with the same $\calE_{0,n}^{(v)}$ for the regularized version $\hatU^{(v)}_{\delta}$ of $\hatU^{(v)}$, with a different constant $C_{\tau}$.


\subsection{Accuracy of Joint Cluster Recovery}
\label{sec:accuracy}
 

In this section, we evaluate the precision of Algorithm~\ref{alg:KRAFTY}.
We assume that the number of clusters  $K_v$, $v=1,2$, in each view,  and the number of 
joint clusters $K$ are known. We start with the case when one observes $\hatZ^{(v)}$, $v=1,2$.

\begin{theorem}[Accuracy of  joint cluster recovery on the basis of $\hatZ^{(v)}$]
\label{thm:consistentZhats}
Let Assumptions~\ref{assump:balancedclusters},\ref{assump:finiteclusternumbers},\ref{assump:consistentZs} hold, and suppose that the joint clustering assignment $\hat{z}$ is obtained using  Algorithm~\ref{alg:KRAFTY}
applied to $\hatZ^{(v)}$, $v=1,2$. If $n$ is large enough, the number of clustering errors $\calE_n$ for the clustering $\hat{z}$ is at most the sum of clustering errors in the individual views $\hatZ^{(1)}$ and $\hatZ^{(2)}$, i.e.,
\begin{equation} \label{eq:number_clust_errors}
    \calE_n \leq \calE_{1,n} + \calE_{2,n}.
\end{equation}  
Therefore, if the clusterings $\hatZ^{(v)}$ for $v=1,2$ are consistent, 
then clustering on the basis of $\hatZ$ is consistent. 
If the clusterings $\hatZ^{(v)}$, $v=1,2$, are perfect, then clustering on the basis of $\hatZ$ is also perfect. 
\label{th:number_clust_errors}
\end{theorem}
%


\noindent
Now, consider the case where data come in the form of estimated matrices of singular vectors $\hatU^{(v)}$, $v=1,2$.

%
\begin{theorem}[Consistent recovery of joint clusters on the basis of $\hatU^{(v)}$]
\label{thm:consistentUhats}
Let Assumptions \ref{assump:balancedclusters}, \ref{assump:finiteclusternumbers}, \ref{assump:consistentUs}, and \ref{assump:incoherentUhats} hold and $\calE_{0,n}^{(v)} = o(1)$ as $n \to \infty$ for $v=1,2$. 
Then, clustering using Algorithm~\ref{alg:KRAFTY} applied to $\hatU^{(v)},v=1,2$
is consistent.
\label{th:consistent_U}
\end{theorem}
%
\begin{theorem}[Perfect clustering]
Let the conditions of Theorem \ref{thm:consistentUhats} hold and, in addition, for $v=1,2$
\begin{equation}  \label{eq:perf_cl_cond_th3}
     \sqrt{n/K_v} \, \calE^{(v)}_{2,\infty,n} = o(1), \quad \text{as }n \to \infty.
\end{equation}
Then for large enough $n$, clustering using Algorithm~\ref{alg:KRAFTY} applied to $\hatU^{(v)},v=1,2$
is perfect with high probability.
\label{th:perfect_U}
\end{theorem}
%
\medskip

\noindent
If one uses the $\hatU^{(v)}$-based version of Algorithm~\ref{alg:KRAFTY}, the
condition $\calE_{0,n}^{(v)} = o(1)$ is necessary for consistent 
clustering in each view $v$. In order for clustering to be perfect under Assumptions \ref{assump:balancedclusters}, \ref{assump:finiteclusternumbers}, \ref{assump:consistentUs}, and \ref{assump:incoherentUhats}, one needs 
$\sqrt{n/K_v} \, \calE^{(v)}_{2,\infty,n}$ to be bounded by a small absolute constant, and the theorem remains true under this weaker condition (though the constant is hard to specify).

In our simulations, we found that hierarchical clustering with complete linkage (see Algorithm~\ref{alg:HC}) exhibited more stable performance for joint clustering compared to $k$-means or Gaussian mixture modeling, so we provide a similar result for the version of Algorithm~\ref{alg:KRAFTY} in which we replace $(1+\epsilon)$-approximate $k$-means with this algorithm. 

\begin{theorem}[Perfect clustering with HC] Let the conditions of Theorem~\ref{th:perfect_U} hold, and let $\hatU$ be the result of Algorithm~\ref{alg:KRAFTY} applied to $\hatU^{(v)},v=1,2$. Then when $n$ is large enough, clustering using Algorithm~\ref{alg:HC} applied to the rows of $\hatU$ yields a perfect clustering at step $t=n-K+1$ with high probability.
\label{th:perfect_hc}
\end{theorem}

\begin{algorithm}
\caption{Complete-Linkage Hierarchical Clustering}
\label{alg:HC}
\begin{flushleft}  
{\bf Input:} $\hatU(i,:) \in \RR^K$, $i \in [n]$. \\
{\bf Output:} Clustering partitions $\calG^{(t)}$ and merge heights $h(t)$ for $t = 1, \ldots, n$.
\\
{\bf Steps:}\\
{\bf 1:} Initialize $\calG_{1}^{(0)} = \{\{1\}, \{2\}, \ldots, \{n\}\}$\;
{\bf 2:}  \For{$t = 2$ to $n$}{
    Find the pair of clusters $(G_i^{(t-1)}, G_j^{(t-1)})$ that minimizes the quantity:
    $$d (G_i^{(t-1)},\, G_j^{(t-1)}) = \max \{ \| \hatU(a,:) - \hatU(b,:) \| : a \in G_i^{(t-1)},\, b \in G_j^{(t-1)} \}$$

    Merge $G_{i}^{(t-1)}$ and $G_{j}^{(t-1)}$ to form $G_{new}^{(t)} = G_{i}^{(t-1)} \cup G_{j}^{(t-1)}$\;
    Set $h(t)=\max\{\|\hat{U}(a,:)-\hat{U}(b,:)\|:a,b\in G_{new}^{(t)}\}.$\;
    Update $\calG^{(t)} = \left( \calG^{(t-1)} \setminus \{G_{i}^{(t-1)}, G_{j}^{(t-1)}\} \right) \cup \{G_{new}^{(t)}\}$\;
}
\end{flushleft} 
\end{algorithm}


\subsection{Estimating the Number of Joint Clusters}
\label{sec:num_clust_est}

As we have discussed earlier, one of the advantages of the KRAFTY
algorithm is that it enables a more accurate estimation of the number of joint clusters, especially  when $K > K_1 + K_2$. 
 The following statements confirm that the singular values of matrices 
 $\hatZ^{(1,2)}$ and $\hatU^{(1,2)}$, defined in \eqref{eq:hatU_from_hatZ} and \eqref{eq:hatU_from_hatU} respectively,  exhibit elbows at the $K$-th singular values.


\begin{theorem}[Consistent estimation of the number of joint clusters on the basis of $\hatZ^{(v)}$] Let Assumptions \ref{assump:balancedclusters}, \ref{assump:finiteclusternumbers}, and \ref{assump:consistentZs} hold, and let $\calE_{v,n}$ be as defined in \eqref{eq:calEvn}. Then
\begin{align}  \label{eq:cl_num_Z1}
    \sigma_K(\hatZ^{(1,2)}) - \sigma_{K+1}(\hatZ^{(1,2)}) & \geq \frac{c_o\sqrt{n}}{\sqrt{K}} \left( 1-\frac{2\sqrt{2}}{c_o} 
    \sqrt{\frac{K(\calE_{1,n} + \calE_{2,n}) } {n}}
     \right), \\
    \sigma_{K+j}(\hatZ^{(1,2)}) - \sigma_{K+j+1}(\hatZ^{(1,2)}) & \leq \frac{c_o\sqrt{n}}{\sqrt{K}} \, \frac{2\sqrt{2}}{c_o} \,
    \sqrt{\frac{K(\calE_{1,n} + \calE_{2,n}) } {n}}, \quad j \geq 1.
    \label{eq:cl_num_Z2}
\end{align}
Since $n^{-1}\,(\calE_{1,n} + \calE_{2,n})=o(1)$ as $n \to \infty$, the singular values display  an elbow at $K$ when $n$ is large enough. 
\label{th:cl_number_Z}
\end{theorem}
%

\begin{theorem}[Consistent estimation of the number of joint clusters on the basis of $\hatU^{(v)}$]
 Let Assumptions \ref{assump:balancedclusters}, \ref{assump:finiteclusternumbers}, \ref{assump:consistentUs}, and \ref{assump:incoherentUhats} hold and $\calE_{o,n}^{(v)}$ in \eqref{eq:Eopvn} be such that $\calE_{o,n}^{(v)} = o(1)$ 
 as $n \to \infty$, $v=1,2$. 
 Then 
\begin{align}  \label{eq:cl_num_U1}
    \sigma_K(\hatU^{(1,2)}) - \sigma_{K+1}(\hatU^{(1,2)}) & \geq \frac{c_D\sqrt{K}}{\sqrt{K_1 K_2\, n}}\, \left[ 1 - \frac{2\, C_{\tau} \, K_1 K_2 }{c_D\, \sqrt{K} } (\calE_{o,n}^{(1)}+\calE_{o,n}^{(2)})  \right],\\
    \sigma_{K+j}(\hatU^{(1,2)}) - \sigma_{K+j+1}(\hatU^{(1,2)}) & \leq \frac{c_D\sqrt{K}}{\sqrt{K_1 K_2\, n}}\,  
    \left[ \frac{2\, C_{\tau} \, K_1 K_2 }{c_D\, \sqrt{K} } \,(\calE_{o,n}^{(1)}+\calE_{o,n}^{(2)})  \right], \quad j \geq 1.
    \label{eq:cl_num_U2}
\end{align}
Since $\mathcal{E}_{o,n}^{(1)}+\mathcal{E}_{o,n}^{(2)} = o(1)$ as $n \to \infty$,  the singular values display an elbow at $K$ when $n$ is large enough.    
\label{th:cl_number_U}
\end{theorem}
%

A similar result holds for the case of hierarchical clustering with complete linkage, if we base our estimate of the number of clusters on an elbow in the merge heights $h(t)$.

\begin{theorem}[Consistent estimation of the number of joint clusters with HC]
\label{th:cluster_number_hc}
Let the conditions of Theorem~\ref{th:perfect_hc} hold, and consider clustering using Algorithm~\ref{alg:HC} applied to the rows of $\hatU^{(1,2)}$. When $n$ is large enough, with high probability there is an elbow in the merge heights at step $n-K+2$, in the following sense: For any fixed $\eta\in(0,1/6)$, once $n$ is large enough, with high probability we have
\begin{align*}
    h(n-K+2)-h(n-K+1)&\geq (1-4\eta)\frac{\sqrt{2K_1K_2}}{c_o^{(1)}c_o^{(2)}n},\\
    h(k+1)-h(k)&\leq 2\eta \frac{\sqrt{2K_1 K_2}}{c_o^{(1)}c_o^{(2)}n},\quad k\leq n-K.
\end{align*}
\end{theorem}

Simulation results demonstrating this phenomenon can be found in Appendix Figure~\ref{fig:hc_mh}.

\section{Simulation Results}
\label{sec: Simulation}
\subsection{Simulation Design}
In this section, we evaluate the empirical performance of KRAFTY in joint clustering and cluster number estimation through simulations, comparing it with MASE under two input settings: (1) $\hatU^{(v)}$ matrices from SVD on each view, and (2) $\hatZ^{(v)}$ matrices obtained by applying $k$-means to $\hatU^{(v)}$.

The simulation data are generated from Gaussian mixtures characterized by three parameters: the dimension ($p$), the within-cluster variance ($\sigma^2$), and the number of components ($K$). To obtain the joint cluster labels, we initialize a $K_1 \times K_2$ probability matrix with a number of nonzeroes equal to the true number of joint clusters $K$. After ensuring that there is at least one point within each joint cluster, the remaining points are drawn according to a Multinomial distribution with the specified probability vector. For each dataset $v \in {1, 2}$, we sample $K_v$ distinct cluster centers from a standard multivariate normal distribution $\mathcal{N}(\mathbf{0}, I_p)$. The resulting cluster means are denoted by $\boldsymbol{\mu}_k^{(v)},k\in [K_v]$.

Observations in each dataset are then generated by

$$
\mathbf{Y}_i^{(v)} \sim \mathcal{N}(\boldsymbol{\mu}_{\tau_v(i)}^{(v)}, \sigma^2 I_p),\quad i=1,\dots,n.
$$

To evaluate performance across different conditions, we design three scenarios, each varying one parameter while keeping the others fixed. 

After generating the datasets, we prepare the inputs to KRAFTY and MASE. Both methods require two inputs, data matrices and number of joint clusters, which are obtained as follows. For the matrices input, when using $\hatU^{(v)}$ matrices, we perform SVD on each dataset and extract the top $K_1$ and $K_2$ left singular vectors to form $\hatU^{(v)}$. These two matrices are used as inputs to KRAFTY or MASE. When using $\hatZ^{(v)}$ matrices, we apply $k$-means clustering to $\hatU^{(v)}$ with $k=K_1$ or $k=K_2$, and use the resulting $\hatZ^{(v)}$ matrices as inputs. To emphasize the comparison between KRAFTY and MASE, we take the number of clusters within individual datasets, $K_1$ and $K_2$, as known. In applications, these numbers would also need to be estimated, but one can make use of standard methods for doing this \cite{ZHU2006918}. For the other input, the number of joint clusters, when $K$ is unknown, we use the automatic scree plot selection method of \cite{ZHU2006918}, choosing the second elbow as $\widehat K$ since the first one often corresponds to a single dominant dimension rather than the true cluster structure. The automatic scree plot selection method is used for two main reasons: (1) it offers an objective and reproducible criterion for estimating $K$, and (2) it facilitates fully automated estimation across the large number of experimental replications required in our study. Under known $K$ condition, we use $K$ as the input.

Applying $k$-means clustering to the individual views to obtain $\hatZ^{(v)}$ provides stable results, but is less consistent for clustering the rows of $\hatU$ after applying KRAFTY or MASE, so we use hierarchical clustering with complete linkage to do the final clustering in all experiments. For evaluation metrics, we evaluate joint clustering performance using the Adjusted Rand Index (ARI) \cite{hubert1985ARI}, which compares the clustering results with the ground truth labels. We evaluate the accuracy of estimation for the joint cluster number using absolute error, defined as the difference between the estimated $\widehat K$ and the true $K$. Results are averaged over 100 repetitions. In all figures, points represent mean values, shaded regions indicate empirical 95\% confidence intervals ($\mathrm{mean} \pm 1.96 \mathrm{SE}$), and KRAFTY is abbreviated as “KR” in the legends.

\subsection{Varying Number of Joint Clusters}

In this simulation study, we fix the parameters at $p=20, n=1000, \sigma^2=0.1, K_1=K_2=4$, while varying the true number of joint clusters from 4 to 16, corresponding to the range from $\max(K_1, K_2)$ to $K_1 \times K_2$.

Figure~\ref{fig:K_ZU} compares the joint clustering performance of KRAFTY and MASE using either $\hatU^{(v)}$ or $\hatZ^{(v)}$ matrices under both known and unknown $K$ conditions. In the left panel, where $\hatZ^{(v)}$ matrices are used as inputs, the three curves overlap and lie above that of MASE under unknown $K$ when $K > K_1 + K_2$, showing that KRAFTY’s advantage in this regime comes from correctly estimating $K$. In the right panel, where the $\hatU^{(v)}$ matrices are used as inputs, KRAFTY outperforms MASE when $K > K_1 + K_2$ under both known and unknown $K$ conditions. 

Figure~\ref{fig:K_compareZU_sigma} further compares the joint clustering performance across the four combinations of methods and input types (KRAFTY-$\hatZ^{(v)}$, KRAFTY-$\hatU^{(v)}$, MASE-$\hatZ^{(v)}$, and MASE-$\hatU^{(v)}$) under unknown $K$ setting. The results are shown for three representative values of $K$ (4, 9, and 15), corresponding to cases where $K$ is smaller than, similar to, and greater than $K_1 + K_2$, respectively, and for six noise levels where $\sigma^2 \in \{0.01, 0.05, 0.1, 0.25, 0.5, 1\}$. It shows that across different noise levels and input types, when $K > K_1 + K_2$, KRAFTY exhibits better performance. Moreover, when $K$ is small, KRAFTY with $\hatZ^{(v)}$ matrices performs comparably to MASE with $\hatZ^{(v)}$ and better than the other two combinations, while as $K$ increases, KRAFTY with $\hatU^{(v)}$ matrices achieves better performance. 

In Appendix Figures~\ref{fig:k1k2_ARI_method} and \ref{fig:k1k2_ARI_matrix}, we compare KRAFTY versus MASE in both the known $K$ and unknown $K$ settings, as well as comparing $\hatU^{(v)}$ versus $\hatZ^{(v)}$ matrices as input to these methods, under various growth rates $K=f(K_1,K_2)$.

Next, we examine the performance of estimating the number of joint clusters for both methods in Figure~\ref{fig:K_compareMethod_estK}. The heatmap colors represent the difference in mean absolute error between KRAFTY and MASE, with warm colors indicating that KRAFTY performs better and cool colors indicating that MASE performs better. Both panels (using $\hatU^{(v)}$ or $\hatZ^{(v)}$ as inputs) corroborate that KRAFTY outperforms MASE when $K$ is large. It is also worth noting that when using $\hatZ^{(v)}$ matrices (right panel), even for small $K$, KRAFTY performs comparably to MASE when the noise level is low (here, $\sigma^2 < 0.5$). This finding aligns with our earlier conclusion that when $K$ is small and the noise level is low, KRAFTY with $\hatZ^{(v)}$ matrices is preferable. Appendix Figure~\ref{fig:K_compareZU_estK} further supports this conclusion by directly comparing the use of $\hatZ^{(v)}$ and $\hatU^{(v)}$ matrices. Appendix Figures~\ref{fig:k1k2_K_method}, and ~\ref{fig:k1k2_matrix} extend the analysis to more values of $K_1, K_2$, and growth rates $K=f(K_1,K_2)$, providing further evidence in support of our findings.

\begin{figure}
    \centering
        \begin{minipage}{\linewidth}
            \centering
            \includegraphics[width=0.9\linewidth]{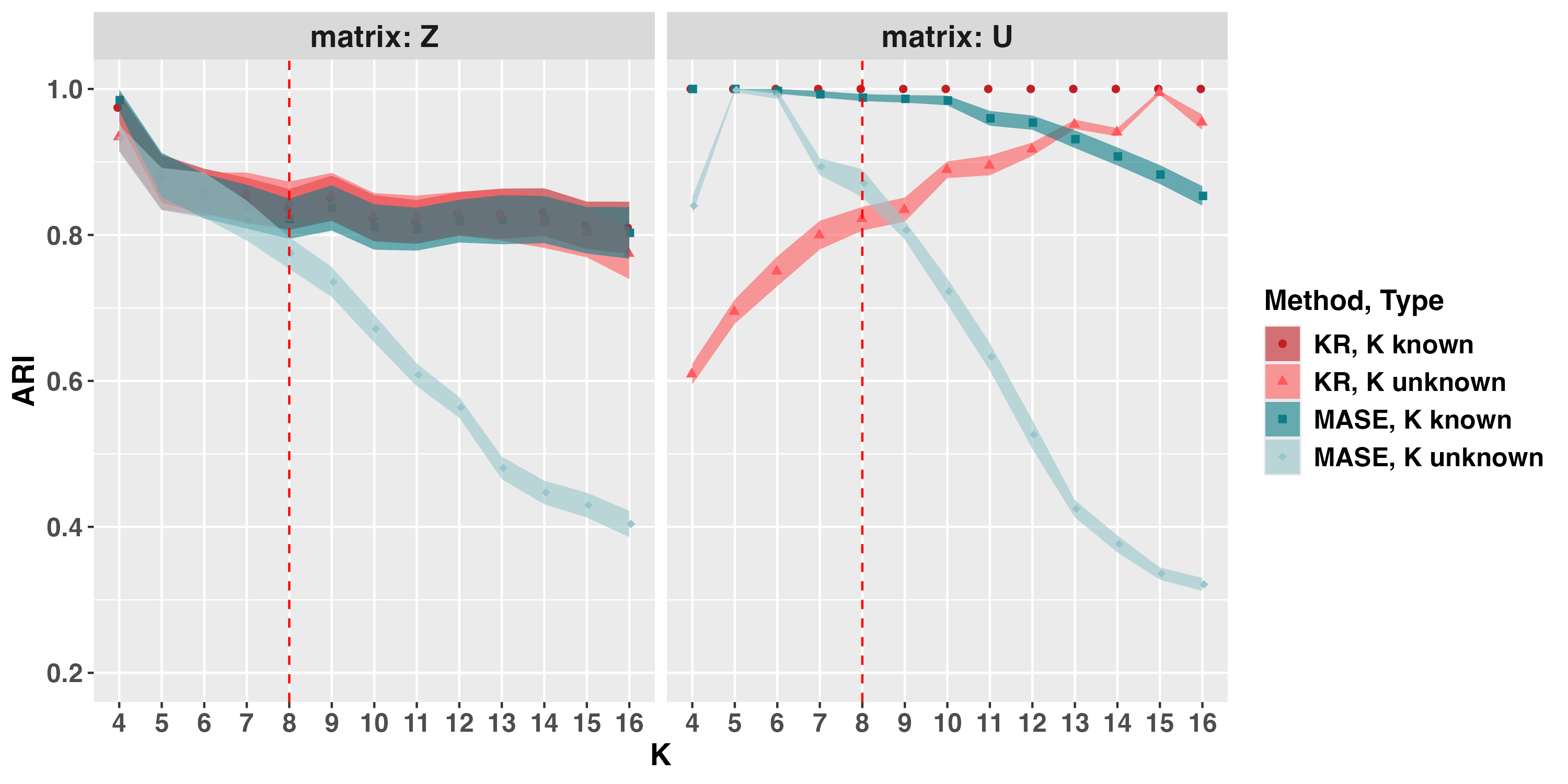}
            \caption{Joint clustering performance of KRAFTY and MASE using $\hatZ^{(v)}$ or $\hatU^{(v)}$ matrices, evaluated under known and unknown $K$. ARI measures agreement between estimated and true joint clusters. $K_1 = K_2 = 4, \sigma^2=0.1, p=20.$}
            \label{fig:K_ZU}
        \end{minipage}  
\end{figure}

\begin{figure}
    \centering
        \begin{minipage}{\linewidth}
            \centering
            \includegraphics[width= .9\linewidth]{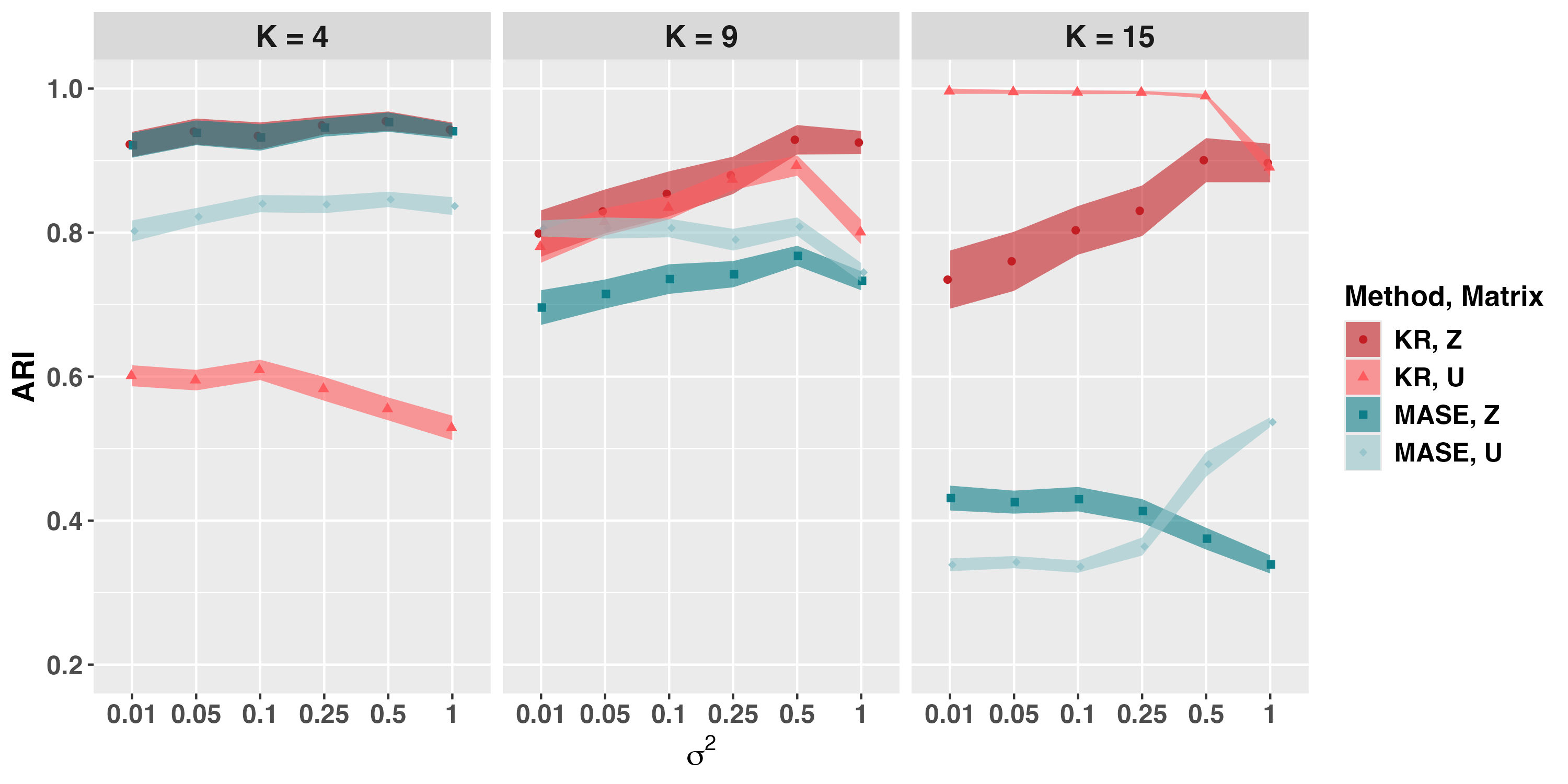}
            \caption{Joint clustering performance of the $\hatZ^{(v)}$ and $\hatU^{(v)}$ matrices using KR and MASE under \textbf{unknown $K$} setting. The true number of joint clusters $K$ is  held constant at 4, 9, and 15 across experiments, with $\sigma^2$ varying. ARI compares estimated joint clusters with ground truth. Here, $K_1 = K_2 = 4, p=20.$}
            \label{fig:K_compareZU_sigma}
        \end{minipage}  
\end{figure}

\begin{figure}
    \centering
        \begin{minipage}{\linewidth}
            \centering
            \includegraphics[width=.9\linewidth]{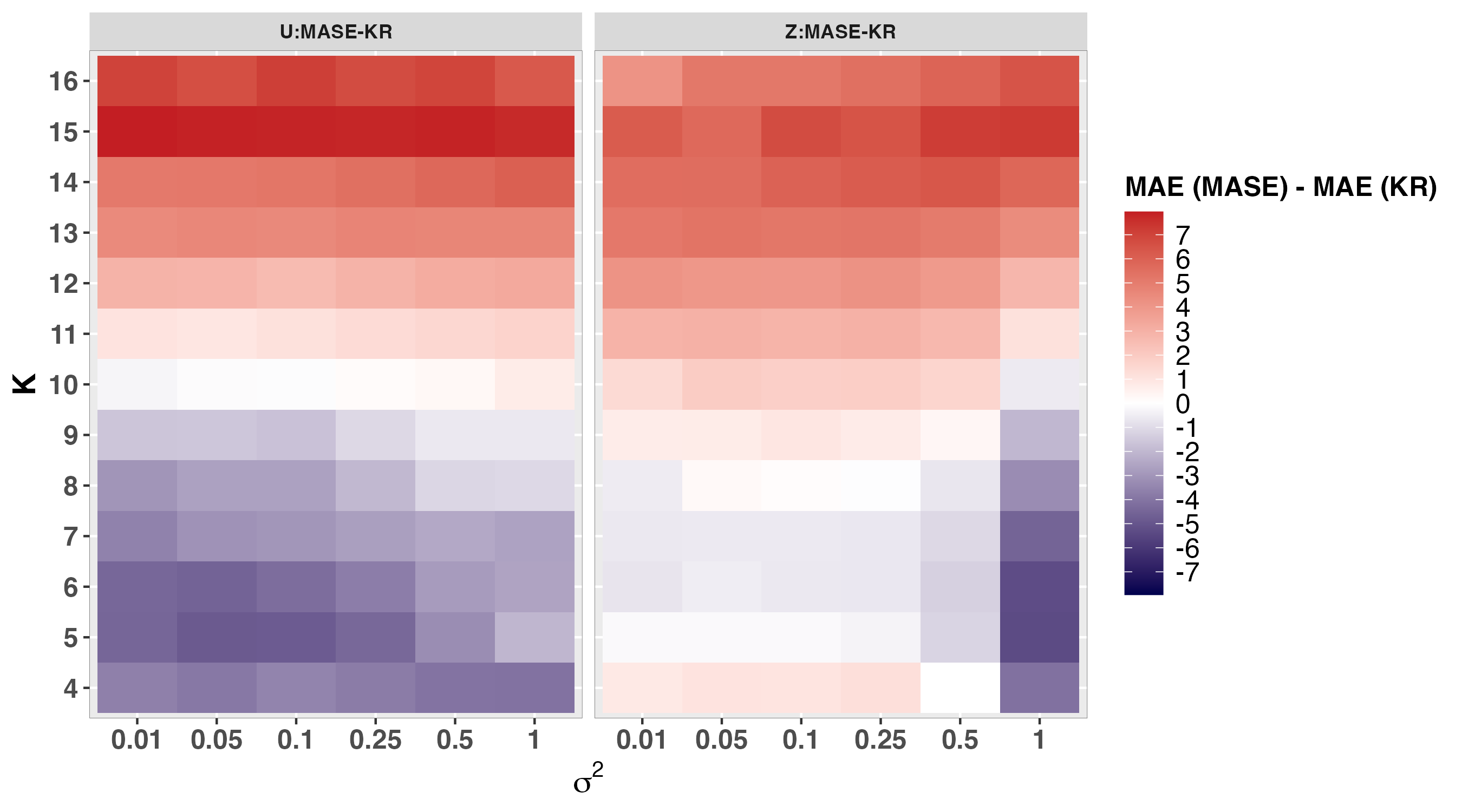}
            \caption{Comparison of joint cluster number estimation performance between KR and MASE across $\hatZ^{(v)}$ and $\hatU^{(v)}$ matrices, when varying the number of true clusters $K$ and noise level $\sigma^2$. Heatmap colors represent the difference in mean absolute error (averaged over 100 repetitions) when estimating $K$ using MASE versus KR. The width of the largest credible interval for the $\hatU^{(v)}$ matrices ($=2(1.96)SE$) is \textcolor{blue}{1.06}, for the $\hatZ^{(v)}$ matrices it is \textcolor{blue}{1.15}. \textbf{Cool colors $\to$ MASE better; warm colors $\to$ KR better.} $K_1 = K_2 = 4, p=20.$}
            \label{fig:K_compareMethod_estK}
        \end{minipage}    
\end{figure}

\subsection{Varying Signal-to-Noise Ratios}
In this simulation, we study how the dimension and noise level affect the performance of KRAFTY. We fix $K_1 = K_2 = 4$, set $K \in \{6, 12\}$, and vary the noise level $\sigma^2 \in \{0.01, 0.25, 1, 5\}$ and the dimension $p \in \{2, 4, 6, 8, 10, 20, 30, 40, 50\}$. Figure~\ref{fig:m_ZU_unknown} presents the results of joint clustering of KRAFTY and MASE when K is unknown. For both methods, performance improves with increasing dimension up to a certain point and then plateaus. This pattern aligns with intuition: as dimensionality increases, clusters become more separable due to the reduced overlap of noisy data, but once they are sufficiently distinct, further increases in dimension provide little additional benefit. 

\begin{figure}
    \centering
        \begin{minipage}{\linewidth}
            \centering
            \includegraphics[width=\linewidth]{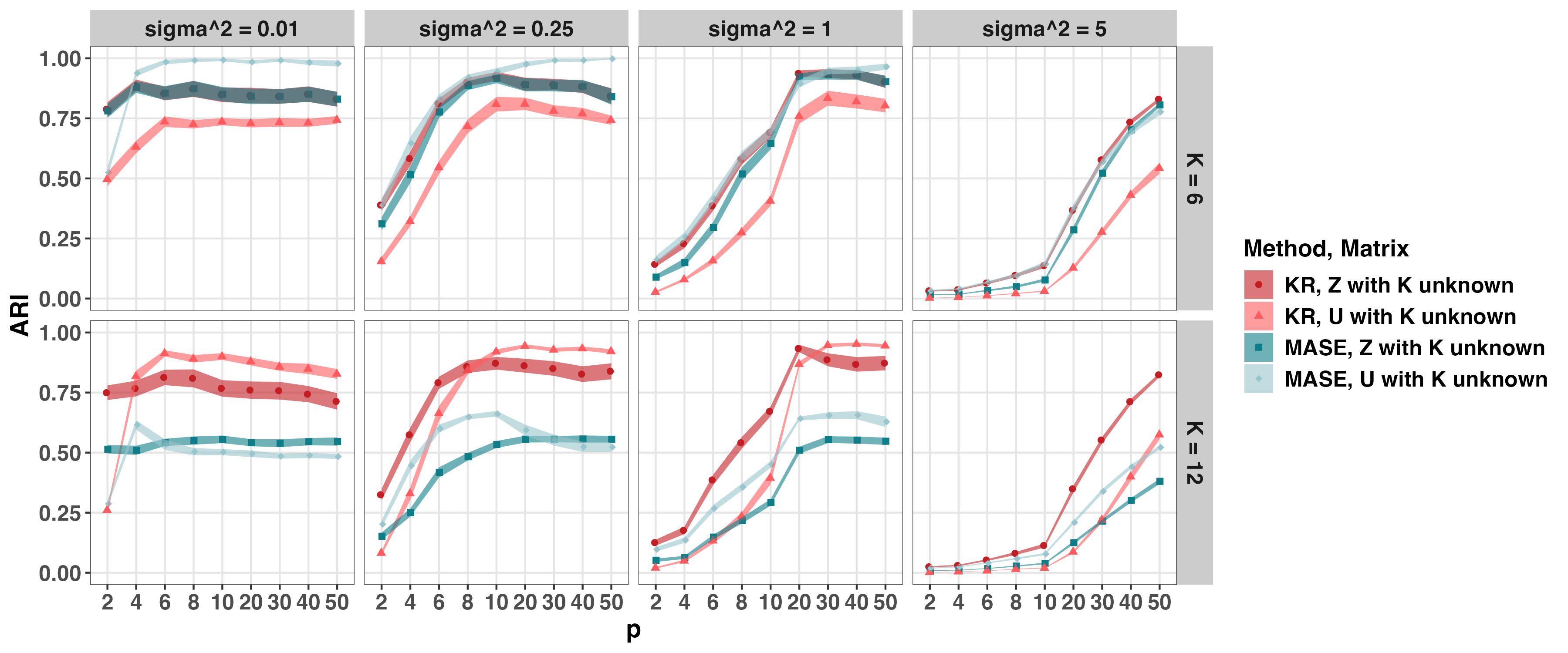}
            \caption{Comparison of joint clustering performance between KR and MASE across the $\hatZ^{(v)}$ and $\hatU^{(v)}$ matrices when the true number of joint clusters ($K$) is \textbf{unknown}, and the dimension $p$ is varied. ARI is computed by comparing the estimated joint clusters with the ground truth. We set $K_1 = K_2 = 4$, the number of clusters in each single graph.}
            \label{fig:m_ZU_unknown}
        \end{minipage}
\end{figure}

In the large-$K$ setting, we see that KRAFTY outperforms MASE across several dimensions and noise levels. In the small-$K$ setting, MASE and KRAFTY with $\hatZ^{(v)}$ perform comparably well, with only very low-noise settings showing a preference for MASE with $\hatU^{(v)}$ matrices as input. For low dimensional data, across values of $K$ and noise level, we see that KRAFTY with $\hatZ^{(v)}$ matrices is an excellent choice, which is perhaps surprising since we expect strong performance from MASE for small values of $K$. A different view of this comparison can be found in Appendix Figure~\ref{fig:m_diffARI_method} (fourth panel in the first row), which further corroborates this result, where the warm-colored lower-right triangle indicates KRAFTY’s better performance under high-noise conditions, in both low and high $K$ settings.

Comparing $\hatZ^{(v)}$ versus $\hatU^{(v)}$ matrices as inputs, we see that under the same noise and dimension settings, KRAFTY with $\hatZ^{(v)}$ matrices consistently outperforms KRAFTY with $\hatU^{(v)}$ when $K$ is small, while for large $K$, KRAFTY with $\hatU^{(v)}$ performs better once the signal-to-noise ratio is sufficiently high, as seen in higher dimensions. This pattern is further supported by Figure~\ref{fig:m_diffARI_ZU_subset}, which directly compares the results of $\hatU^{(v)}$ and $\hatZ^{(v)}$ across all combinations of dimension and noise level. When $K = 6$, the heatmaps are all cool-colored, indicating that KRAFTY with $\hatZ^{(v)}$ performs better. In contrast, when $K = 12$, the upper-left region of the heatmap (corresponding to high dimension relative to noise) is warm-colored, while the lower region is cool-colored, suggesting that KRAFTY with $\hatU^{(v)}$ performs better under high-dimensional, low-noise conditions. A more complete comparison of $\hatU^{(v)}$ versus $\hatZ^{(v)}$ matrices can be found in Appendix Figure~\ref{fig:m_diffARI_ZU}.

\begin{figure}
    \centering
        \begin{minipage}{\linewidth}
            \centering
            \includegraphics[width=0.55\linewidth]{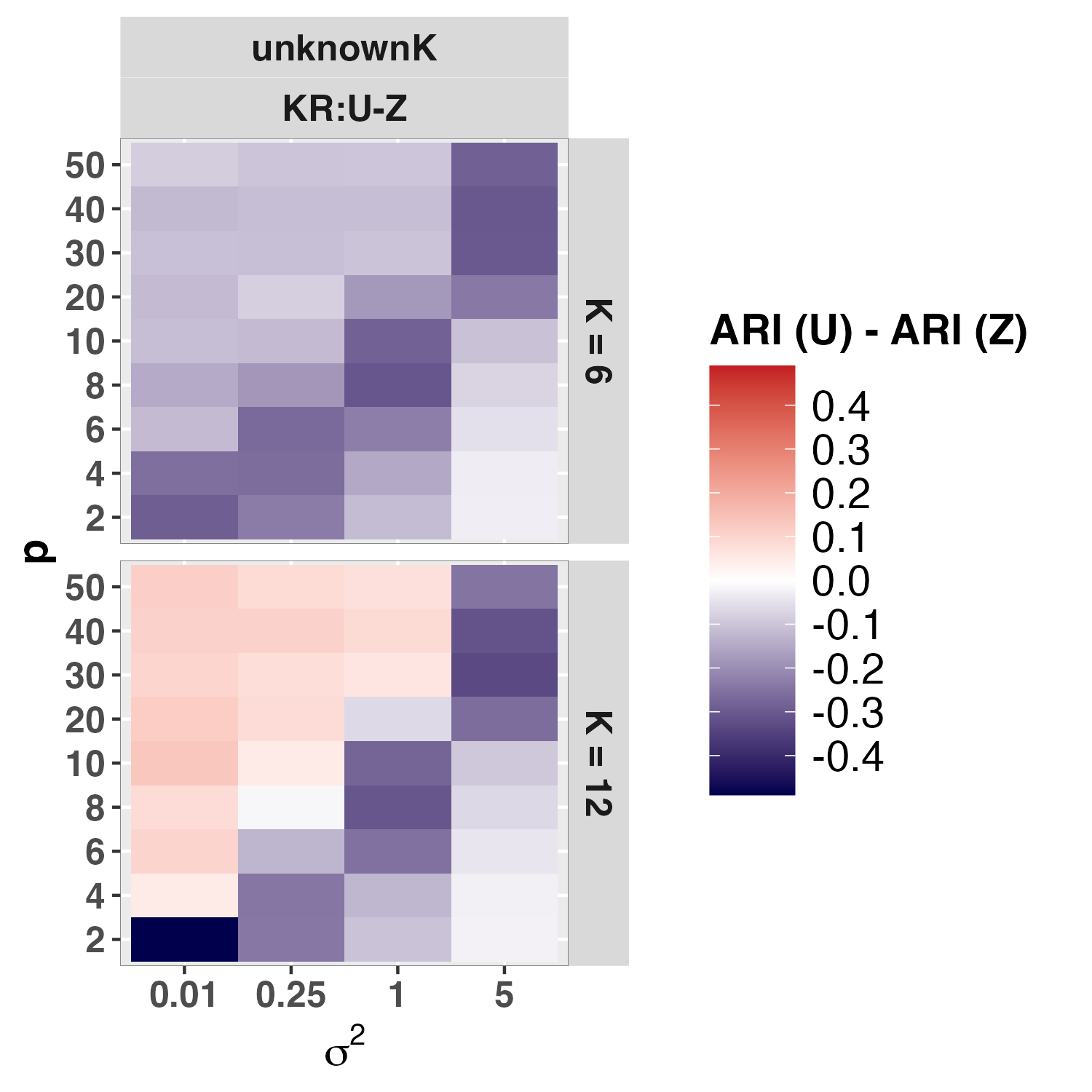}
            \caption{Comparison of joint clustering performance between $\hatZ^{(v)}$ and $\hatU^{(v)}$ for KRAFTY. Heatmap colors represent the difference in mean ARI (averaged over 100 repetitions) when the number of joint clusters $K$ is \textbf{unknown}. \textbf{Cool colors $\to$ Z better; warm colors $\to$ U better.} $K_1=K_2=4$.
            }
            \label{fig:m_diffARI_ZU_subset}
        \end{minipage}
\end{figure}


Figure~\ref{fig:m_K_method} presents the results for estimating the number of joint clusters. The figure is organized such that the left four panels compare $\hatU^{(v)}$ and $\hatZ^{(v)}$ for each method, while the right four panels compare KRAFTY and MASE for each input type. When $K$ is small, MASE performs better, except when using $\hatZ^{(v)}$ under relatively high-dimensional settings, where the difference between KRAFTY and MASE becomes small. In contrast, when $K$ is large, KRAFTY clearly outperforms MASE. Comparing $\hatU^{(v)}$ and $\hatZ^{(v)}$ for KRAFTY, we observe that $\hatZ^{(v)}$ provides more accurate estimates of the number of joint clusters when the dimension is sufficient, while $\hatU^{(v)}$ performs better when the dimension is low and noise level is high. In these settings, $\hatZ^{(v)}$ tends to over-estimate the true number of clusters: this leads to high error on the estimation of the number of joint clusters, but does not lead to a major decrease in clustering performance measured using ARI.

\begin{figure}
    \centering
        \begin{minipage}{\linewidth}
            \centering
            \includegraphics[width=.9\linewidth]{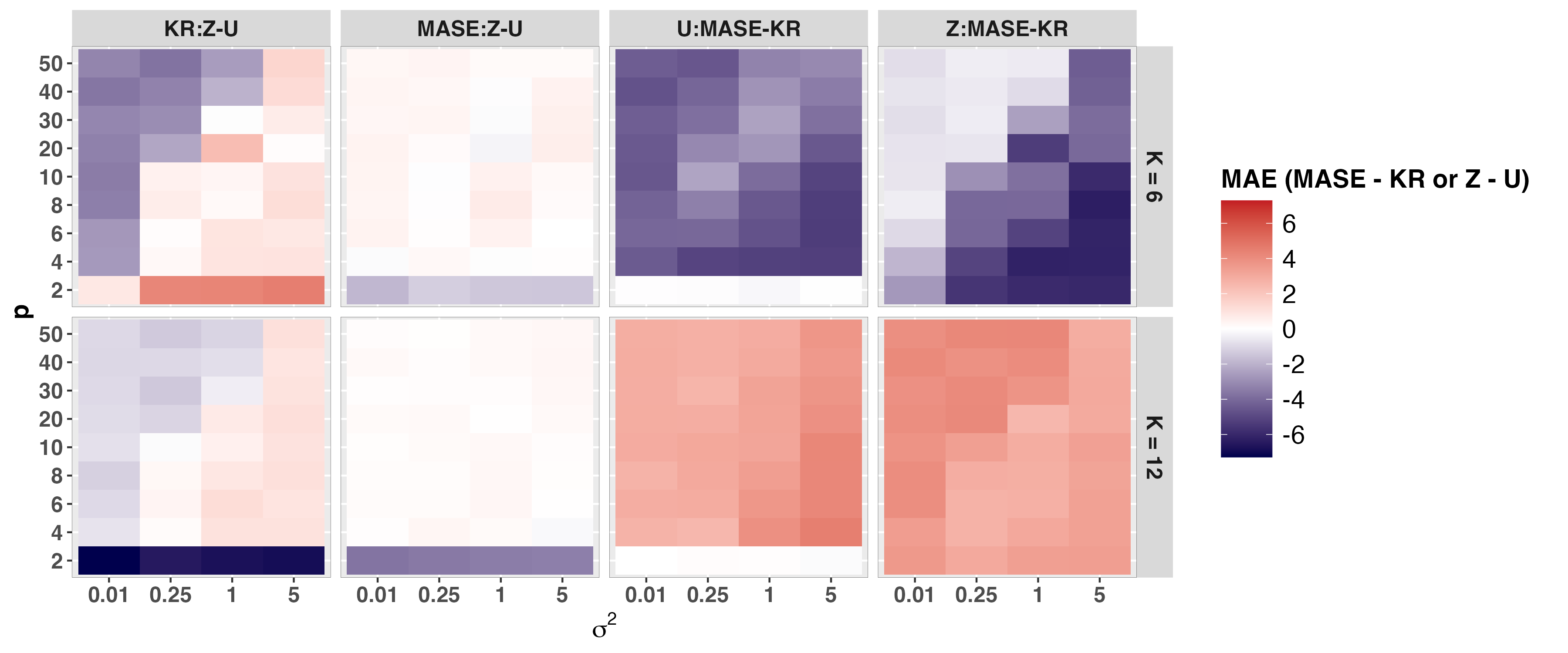}
            \caption{Comparison of joint cluster number estimation performance between KR and MASE across Z and U matrices, as the dimension $p$ and noise level $\sigma^2$ are varied. The left two columns of the heatmap show the difference in mean absolute error (averaged over 100 repetitions) when estimating $K$ using Z versus U. \textbf{Cool colors $\to$ Z better; warm colors $\to$ U better.} The right 2 columns show the difference in mean absolute error between KR and MASE methods. \textbf{Cool colors $\to$ MASE better; warm colors $\to$ KR better.} Here, $K_1 = K_2 = 4$.}
            \label{fig:m_K_method}
        \end{minipage}
\end{figure}

\section{Global Trade Data Analysis}
\label{sec: Real-data}
We use the food and agricultural trade dataset provided by the Food and Agriculture Organization (FAO), which is also used in \cite{JingLi2021}, and is publicly available at http://www.fao.org. The data contains a collection of directed, weighted trade networks, one per product, where nodes are countries and edges point from exporters to importers, weighted by the product’s USD trade value. 

Following previous work \cite{DeDomenico2014, JingLi2021, agterberg2024}, we select data from 2010. We also include data from 2023, the most recent available year, in our analysis. The analysis aims (i) to uncover joint clustering structures that capture countries’ joint memberships across their exporting and importing behaviors in 2010, and (ii) to uncover exporter-specific and importer-specific joint clusterings that generalize across years. We focus on one product (raw chicken meat) and analyze two snapshots of its trade network (2010 and 2023), which share an identical vertex set of 191 countries. For each year $v \in \{2010, 2023\}$, let $A^{(v)}$ denote the observed trade matrix (entries are USD exports of the product from row country to column country). We compute its singular value decomposition $A^{(v)} = U^{(v)} \Sigma^{(v)} {V^{(v)}}^\top$, and retain the top three dimensions for 2010 and the top two for 2023, which are selected by \cite{ZHU2006918}. Then we estimate countries' latent positions as exporters and importers by row-normalizing $U^{(v)} (\Sigma^{(v)})^{1/2}$ and $V^{(v)} (\Sigma^{(v)})^{1/2}$ to unit length. The normalized exporter and importer embeddings are denoted $\hatU^{(v)}$ and $\hat V^{(v)}$, respectively. We then apply $k$-means with $k=3$ and $k=2$ to $\hatU ^ {(v)}$ and $\hat V^{(v)}$ to obtain the $\hatZ^{(v)}: v \in \{(2010,\mathrm{imp}),\ (2010,\mathrm{exp}),\ (2023,\mathrm{imp}),\ (2023,\mathrm{exp})\}$ matrices as our inputs. We detail the findings for the single-year analysis in this section and leave the double-year analysis to the appendix, since they follow the same analysis procedures.

\begin{figure}
    \centering  
    \begin{minipage}[b]{0.48\linewidth} 
        \centering
        \includegraphics[width=\linewidth]{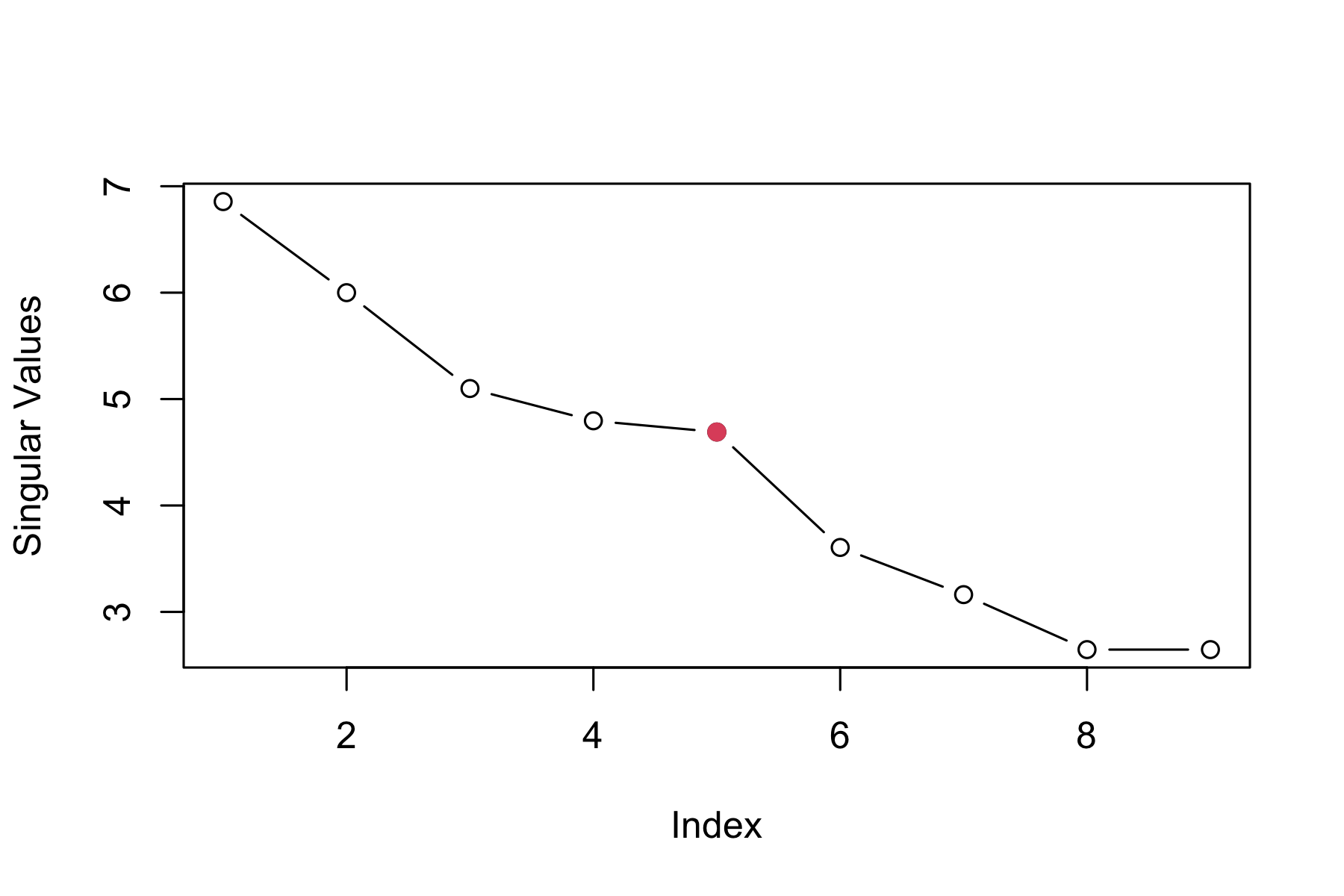}
        \vspace{3pt}
        \centerline{\small (a) KRAFTY} 
        \label{fig:kr_1yr}
    \end{minipage}
    \hfill 
    \begin{minipage}[b]{0.48\linewidth}
        \centering
        \includegraphics[width=\linewidth]{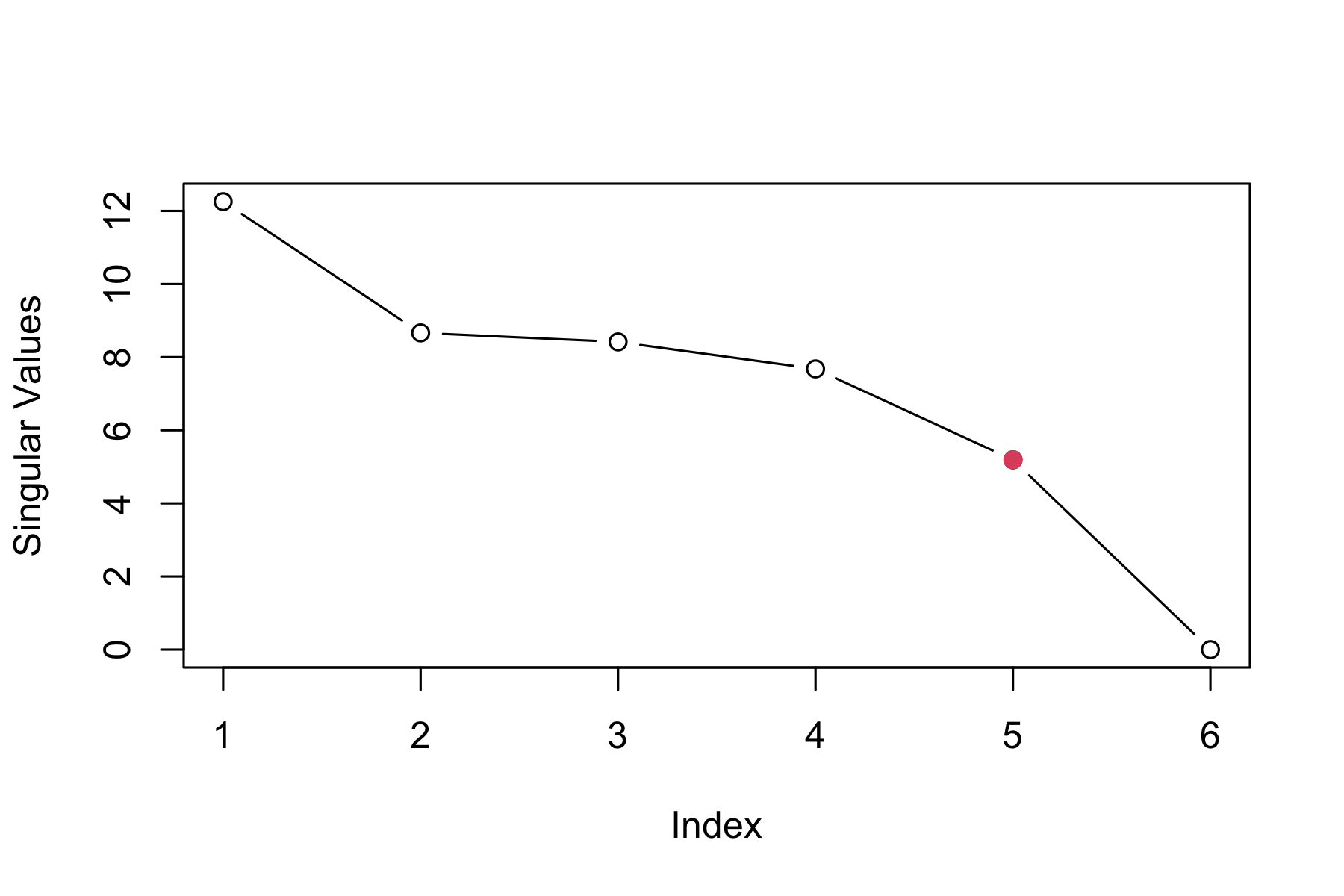}
        \vspace{3pt}
        \centerline{\small (b) MASE} 
        \label{fig:mase_1yr}
    \end{minipage}
    
    \caption{Eigenvalue scree plot for 2010; the red dot indicates the selected dimension.}
    \label{fig:scree_1yr}
\end{figure}

\begin{figure}
    \centering
    \includegraphics[width=\linewidth]{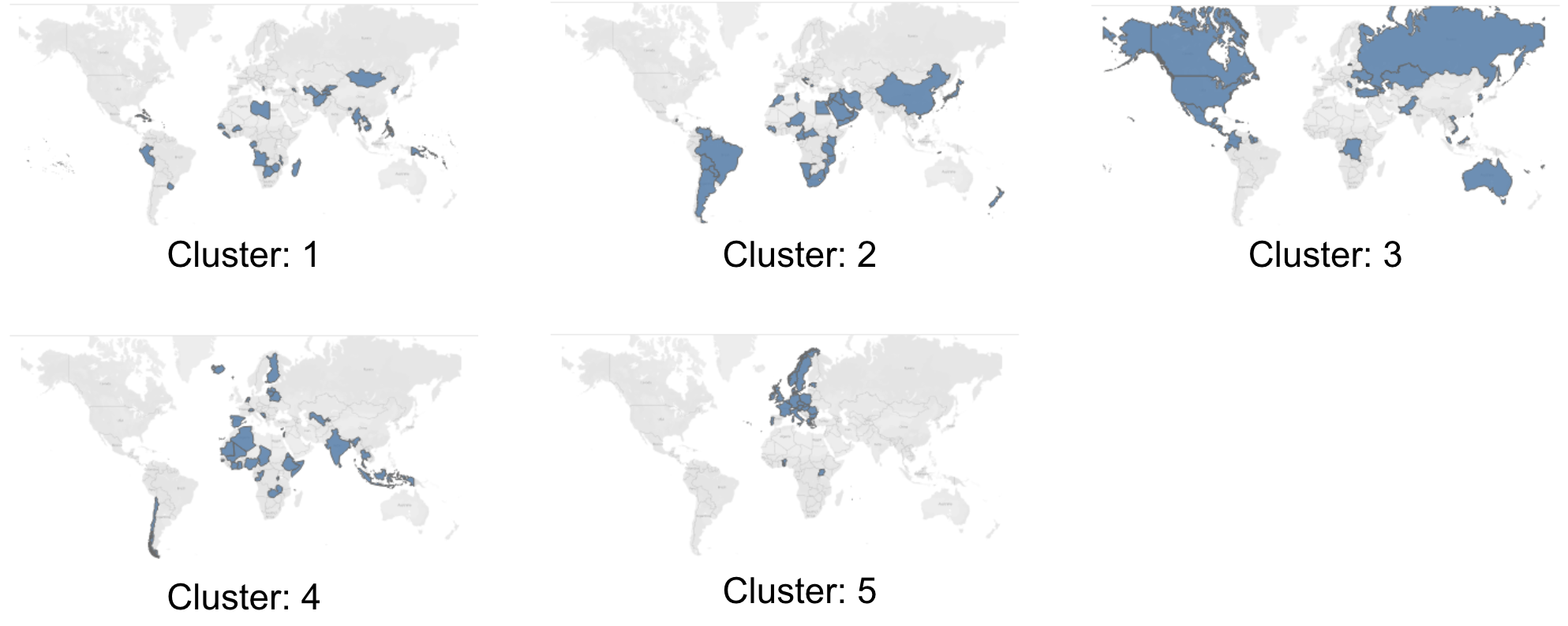}
    \caption{World map colored by KRAFTY joint clusters for 2010; blue color denotes cluster membership.}
    \label{fig:map_1yr}
\end{figure}


\begin{figure}
    \centering   
    \begin{minipage}[b]{0.48\linewidth}
        \centering
        \includegraphics[width=\linewidth]{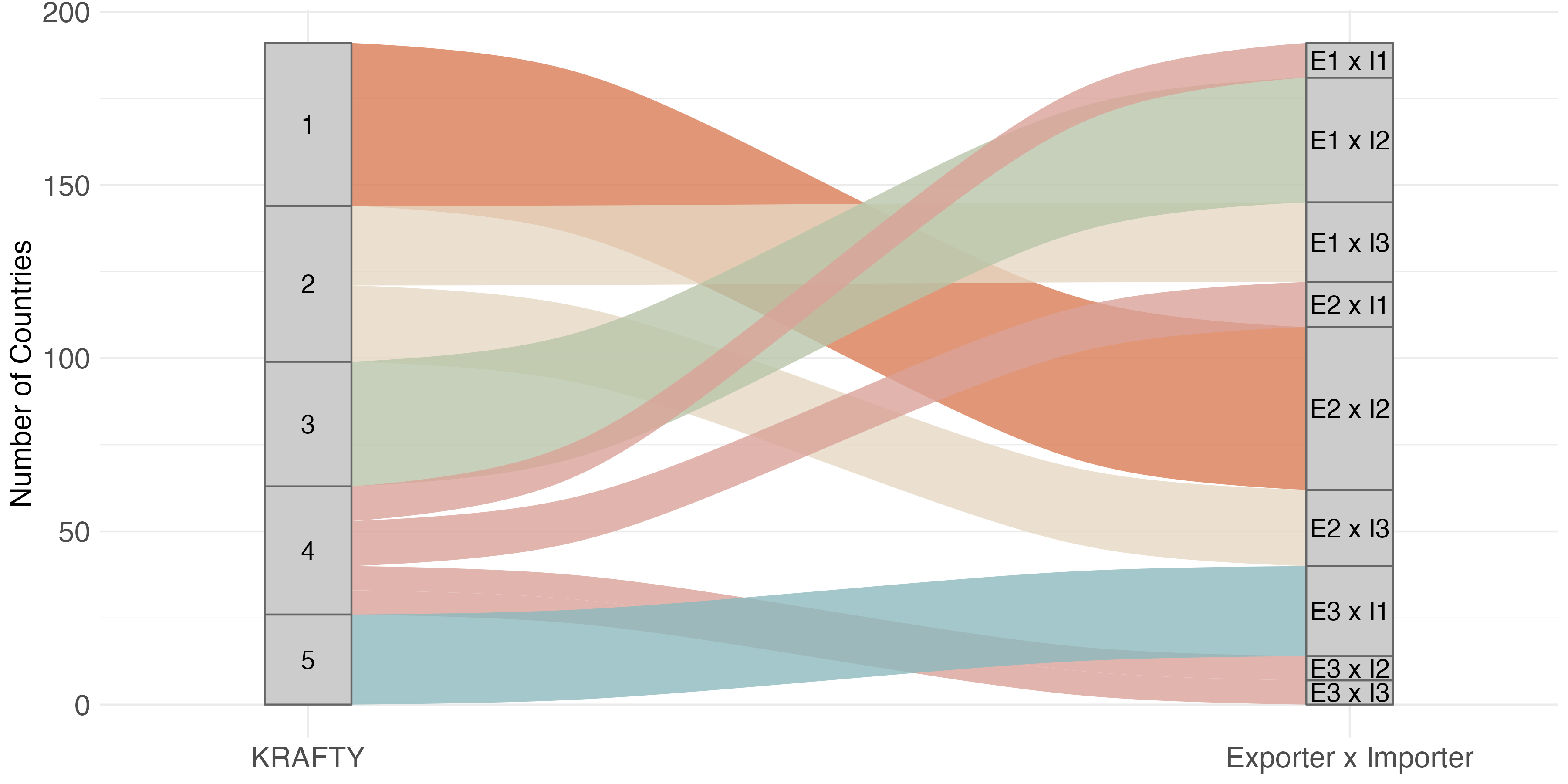}
        \vspace{3pt}
        \centerline{\small (a) Joint (KRAFTY) vs Single-Source}
        \label{fig:alluvial_1yr_kr_exp_imp} 
    \end{minipage}
    \hfill 
    \begin{minipage}[b]{0.48\linewidth}
        \centering
        \includegraphics[width=\linewidth]{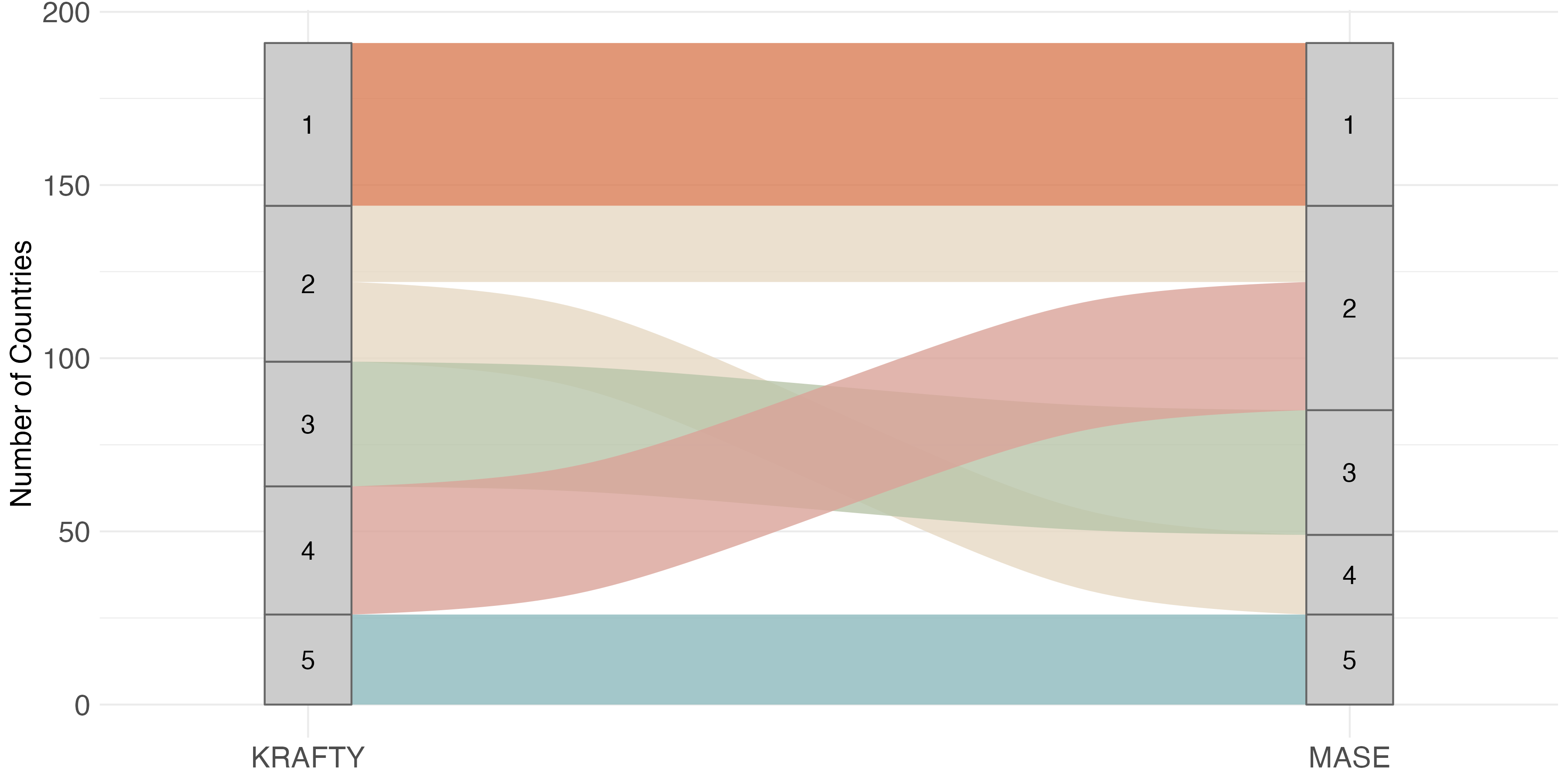}
        \vspace{3pt}
        \centerline{\small (b) KRAFTY vs MASE}
        \label{fig:alluvial_1yr_kr_mase}
    \end{minipage}
    
    \caption{Alluvial plot comparing cluster assignments for 2010. (a) contrasts KRAFTY’s joint clusters with combinations of single-view clusters; identical colors indicate the same set of entities, and block height reflects the number of countries in each cluster. (b) compares KRAFTY’s joint clusters with those from MASE.}
    \label{fig:alluvial_1yr}
\end{figure}

In this study, we apply both KRAFTY and MASE method to the $\hatZ^{(v)}: v \in \{(2010,\mathrm{imp}),\ (2010,\mathrm{exp})\}$ matrices which are the 2010 importer and exporter cluster assignments, to recover joint groups. Figure~\ref{fig:scree_1yr} shows the scree plots, with a red point marking the estimated number of joint clusters for KRAFTY (Fig.~\ref{fig:scree_1yr}(a)) and MASE (Fig.~\ref{fig:scree_1yr}(b)). The scree plot of KRAFTY exhibits a distinct elbow, making selection easier, which is consistent with our earlier findings. The appendix scree plots for the two-year exporter and importer analysis also show the same advantage. We use hierarchical clustering in both methods and compare their joint cluster assignments. The ARI is 0.79, which indicates a strong agreement between the two methods. This aligns with our simulations: when the number of joint clusters is smaller than the sum of the single-source clusters, KRAFTY with $\hatZ^{(v)}$ inputs performs similarly to MASE. 

Next, we color the world map by joint clusters to see the patterns they reveal; the result is shown in Figure~\ref{fig:map_1yr}. From the colored map, the clusters are primarily regional: Europe, North America, and an Asia–Africa–South America group. This likely reflects the high perishability of raw chicken meat, which favors shorter, regional trading. However, there are exceptions: clusters 1 and 4 appear to mix subsets of Asia–Africa and Europe–Asia-Africa, respectively. These exceptions correspond to groups with significantly lower trading volumes compared to the major regional clusters, as detailed in the trading network heatmap in Appendix Figure~\ref{fig:adjacency_1yr}. Hence, the joint clusters capture both regional patterns and countries’ trading-value behavior. Moreover, we use alluvial plots (Figure~\ref{fig:alluvial_1yr}) to visualize how the joint clusters come from the single-source clusters and to compare where KRAFTY’s joint assignments differ slightly from MASE’s.

\FloatBarrier
\section{Discussion}

In this work, we propose KRAFTY, a novel joint clustering framework based on the transposed Khatri-Rao product. KRAFTY requires no assumptions about the sources of single-view data and operates directly on either the clustering assignment matrices ($\hatZ^{(v)}$) or low-dimensional spectral embeddings ($\hatU^{(v)}$) from each view. This design makes KRAFTY a flexible and adaptable tool for recovering the joint clustering structure across multiple views. Another appealing feature of KRAFTY is that it supplies sufficient dimension for joint clusters to occupy orthogonal subspaces, unlike concatenation which cannot have rank greater than the sum of the single-view cluster counts. Our main results show that when the number of data points $n$ is large, the joint clustering error is upper-bounded by the sum of the single-view clustering errors. If each single-view clustering is consistent, the joint clustering is also consistent; if the single-view clustering is perfect, the joint clustering is also perfect. This is a desirable property, as researchers can use existing high-quality single-view clustering methods independently and in parallel to achieve accurate joint clustering results. Moreover, we prove that when $n$ is large, the joint embedding produced by KRAFTY has a clear elbow at the true number of joint clusters $K$, enabling reliable estimation of $K$ using the scree plot or automatic dimension selection methods. In simulations, we systematically compare KRAFTY with the widely used MASE method, as well as evaluate its performance under different input types, $\hatZ^{(v)}$ and $\hatU^{(v)}$. KRAFTY demonstrates strong performance in both clustering assignment accuracy and estimation of the number of joint clusters, particularly when the true number of joint clusters $K$ exceeds the sum of single-view cluster counts. Even when $K$ is smaller than this sum, KRAFTY still achieves comparable performance. In our world trade data analysis, we apply KRAFTY to recover joint clusters of countries based on their roles as exporters and importers. 

Finally, we discuss several directions for future work. Although KRAFTY assumes consistent single-view clustering, real-world scenarios often involve noisy, imprecise, or biased clusterings. It would be interesting to explore whether KRAFTY could be used to denoise or refine such noisy inputs through joint structure. Moreover, the presence of evolving cluster structures in many real-world datasets suggests an interesting application of KRAFTY to dynamic community discovery.


\section{Supplementary Material}
The Supplementary Material includes proofs for Theorems 1-7, additional simulations, a visual example of KRAFTY, and an additional example on the real data application. The code for simulation and real data analysis is available at GitHub: \url{https://github.com/Ciel-Gao/KRAFTY}.


\begin{thebibliography}{36}
\expandafter\ifx\csname natexlab\endcsname\relax\def\natexlab#1{#1}\fi

\bibitem[{Abbe(2023)}]{abbe2023community}
\textsc{Abbe, E.} (2023).
\newblock Community detection and stochastic block models.

\bibitem[{Abbe et~al.(2022)Abbe, Fan \& Wang}]{abbe_fan_AOS2022}
\textsc{Abbe, E.}, \textsc{Fan, J.} \& \textsc{Wang, K.} (2022).
\newblock {An ${\ell _{p}}$ theory of PCA and spectral clustering}.
\newblock \textit{The Annals of Statistics} \textbf{50}, 2359 -- 2385.

\bibitem[{Agterberg et~al.(2025)Agterberg, Lubberts \& Arroyo}]{agterberg2025}
\textsc{Agterberg, J.}, \textsc{Lubberts, Z.} \& \textsc{Arroyo, J.} (2025).
\newblock Joint spectral clustering in multilayer degree-corrected stochastic blockmodels.

\bibitem[{Agterberg \& Zhang(2024)}]{agterberg2024}
\textsc{Agterberg, J.} \& \textsc{Zhang, A.} (2024).
\newblock Estimating higher-order mixed memberships via the $\ell_{2,\infty}$ tensor perturbation bound.

\bibitem[{Arroyo et~al.(2020)Arroyo, Athreya, Cape, Chen, Priebe \& Vogelstein}]{arroyo2020}
\textsc{Arroyo, J.}, \textsc{Athreya, A.}, \textsc{Cape, J.}, \textsc{Chen, G.}, \textsc{Priebe, C.~E.} \& \textsc{Vogelstein, J.~T.} (2020).
\newblock Inference for multiple heterogeneous networks with a common invariant subspace.

\bibitem[{Domenico et~al.(2014)Domenico, Nicosia, Arenas \& Latora}]{DeDomenico2014}
\textsc{Domenico, M.~D.}, \textsc{Nicosia, V.}, \textsc{Arenas, A.} \& \textsc{Latora, V.} (2014).
\newblock Structural reducibility of multilayer networks.
\newblock \textit{Nature Communications} \textbf{6}.

\bibitem[{Fortunato \& Newman(2022)}]{Fortunato2022}
\textsc{Fortunato, S.} \& \textsc{Newman, M. E.~J.} (2022).
\newblock 20 years of network community detection.
\newblock \textit{Nature Physics} \textbf{18}, 848–850.

\bibitem[{Gallagher et~al.(2021)Gallagher, Jones \& Rubin-Delanchy}]{UASE}
\textsc{Gallagher, I.}, \textsc{Jones, A.} \& \textsc{Rubin-Delanchy, P.} (2021).
\newblock Spectral embedding for dynamic networks with stability guarantees.
\newblock In \textit{Advances in Neural Information Processing Systems}, M.~Ranzato, A.~Beygelzimer, Y.~Dauphin, P.~Liang \& J.~W. Vaughan, eds., vol.~34. Curran Associates, Inc.

\bibitem[{Han et~al.(2015)Han, Xu \& Airoldi}]{han2015}
\textsc{Han, Q.}, \textsc{Xu, K.~S.} \& \textsc{Airoldi, E.~M.} (2015).
\newblock Consistent estimation of dynamic and multi-layer block models.

\bibitem[{Handcock et~al.(2007)Handcock, Raftery \& Tantrum}]{handcock2007}
\textsc{Handcock, M.~S.}, \textsc{Raftery, A.~E.} \& \textsc{Tantrum, J.~M.} (2007).
\newblock Model-based clustering for social networks.
\newblock \textit{Journal of the Royal Statistical Society Series A: Statistics in Society} \textbf{170}, 301--354.

\bibitem[{Holland et~al.(1983)Holland, Laskey \& Leinhardt}]{Holland1983}
\textsc{Holland, P.}, \textsc{Laskey, K.~B.} \& \textsc{Leinhardt, S.} (1983).
\newblock Stochastic blockmodels: First steps.
\newblock \textit{Social Networks} \textbf{5}, 109--137.

\bibitem[{Hubert \& Arabie(1985)}]{hubert1985ARI}
\textsc{Hubert, L.} \& \textsc{Arabie, P.} (1985).
\newblock Comparing partitions.
\newblock \textit{Journal of classification} \textbf{2}, 193--218.

\bibitem[{Jaeger \& Banks(2023)}]{ClusterAnalysis}
\textsc{Jaeger, A.} \& \textsc{Banks, D.} (2023).
\newblock Cluster analysis: A modern statistical review.
\newblock \textit{WIREs Computational Statistics} \textbf{15}, e1597.

\bibitem[{Jin et~al.(2023)Jin, Yu, Jiao, Pan, He, Wu, Yu \& Zhang}]{Jin2023}
\textsc{Jin, D.}, \textsc{Yu, Z.}, \textsc{Jiao, P.}, \textsc{Pan, S.}, \textsc{He, D.}, \textsc{Wu, J.}, \textsc{Yu, P.~S.} \& \textsc{Zhang, W.} (2023).
\newblock A survey of community detection approaches: From statistical modeling to deep learning.
\newblock \textit{IEEE Transactions on Knowledge and Data Engineering} \textbf{35}, 1149--1170.

\bibitem[{Jin(2015)}]{Jin_2015}
\textsc{Jin, J.} (2015).
\newblock Fast community detection by score.
\newblock \textit{The Annals of Statistics} \textbf{43}.

\bibitem[{Jing et~al.(2021)Jing, Li, Lyu \& Xia}]{JingLi2021}
\textsc{Jing, B.-Y.}, \textsc{Li, T.}, \textsc{Lyu, Z.} \& \textsc{Xia, D.} (2021).
\newblock {Community detection on mixture multilayer networks via regularized tensor decomposition}.
\newblock \textit{The Annals of Statistics} \textbf{49}, 3181 -- 3205.

\bibitem[{Kim et~al.(2018)Kim, Bandeira \& Goemans}]{kim2018sbmhypergraphs}
\textsc{Kim, C.}, \textsc{Bandeira, A.~S.} \& \textsc{Goemans, M.~X.} (2018).
\newblock Stochastic block model for hypergraphs: Statistical limits and a semidefinite programming approach.
\newblock \textit{arXiv preprint} \textbf{arXiv:1807.02884}.

\bibitem[{Kivel{\"a} et~al.(2014)Kivel{\"a}, Arenas, Barthelemy, Gleeson, Moreno \& Porter}]{kivela2014multilayer}
\textsc{Kivel{\"a}, M.}, \textsc{Arenas, A.}, \textsc{Barthelemy, M.}, \textsc{Gleeson, J.~P.}, \textsc{Moreno, Y.} \& \textsc{Porter, M.~A.} (2014).
\newblock Multilayer networks.
\newblock \textit{Journal of Complex Networks} \textbf{2}, 203--271.

\bibitem[{Ledesma et~al.(2015)Ledesma, Valero-Mora \& Macbeth}]{Ledesma_Valero-Mora_Macbeth_2015}
\textsc{Ledesma, R.~D.}, \textsc{Valero-Mora, P.} \& \textsc{Macbeth, G.} (2015).
\newblock The scree test and the number of factors: a dynamic graphics approach.
\newblock \textit{The Spanish Journal of Psychology} \textbf{18}, E11.

\bibitem[{Lei \& Lin(2023)}]{Lei02102023}
\textsc{Lei, J.} \& \textsc{Lin, K.~Z.} (2023).
\newblock Bias-adjusted spectral clustering in multi-layer stochastic block models.
\newblock \textit{Journal of the American Statistical Association} \textbf{118}, 2433--2445.

\bibitem[{Lei \& Rinaldo(2015)}]{Lei2015}
\textsc{Lei, J.} \& \textsc{Rinaldo, A.} (2015).
\newblock Consistency of spectral clustering in stochastic block models.
\newblock \textit{The Annals of Statistics} \textbf{43}.

\bibitem[{Lyzinski et~al.(2015)Lyzinski, Sussman, Tang, Athreya \& Priebe}]{lyzinski2015}
\textsc{Lyzinski, V.}, \textsc{Sussman, D.}, \textsc{Tang, M.}, \textsc{Athreya, A.} \& \textsc{Priebe, C.} (2015).
\newblock Perfect clustering for stochastic blockmodel graphs via adjacency spectral embedding.

\bibitem[{Löffler et~al.(2020)Löffler, Zhang \& Zhou}]{loffler2020}
\textsc{Löffler, M.}, \textsc{Zhang, A.~Y.} \& \textsc{Zhou, H.~H.} (2020).
\newblock Optimality of spectral clustering in the gaussian mixture model.

\bibitem[{MacDonald et~al.(2021)MacDonald, Levina \& Zhu}]{macdonald2021}
\textsc{MacDonald, P.~W.}, \textsc{Levina, E.} \& \textsc{Zhu, J.} (2021).
\newblock Latent space models for multiplex networks with shared structure.

\bibitem[{Milligan \& Cooper(1987)}]{ClusteringMethods}
\textsc{Milligan, G.~W.} \& \textsc{Cooper, M.~C.} (1987).
\newblock Methodology review: Clustering methods.
\newblock \textit{Applied Psychological Measurement} \textbf{11}, 329--354.

\bibitem[{Paul \& Chen(2018)}]{paul2018}
\textsc{Paul, S.} \& \textsc{Chen, Y.} (2018).
\newblock Spectral and matrix factorization methods for consistent community detection in multi-layer networks.

\bibitem[{Pensky(2024)}]{pensky2024daviskahan}
\textsc{Pensky, M.} (2024).
\newblock Davis-kahan theorem in the two-to-infinity norm and its application to perfect clustering.

\bibitem[{Qin \& Rohe(2013)}]{qin2013regularized}
\textsc{Qin, T.} \& \textsc{Rohe, K.} (2013).
\newblock Regularized spectral clustering under the degree-corrected stochastic blockmodel.

\bibitem[{Rohe et~al.(2011)Rohe, Chatterjee \& Yu}]{Rohe2011}
\textsc{Rohe, K.}, \textsc{Chatterjee, S.} \& \textsc{Yu, B.} (2011).
\newblock Spectral clustering and the high-dimensional stochastic blockmodel.
\newblock \textit{The Annals of Statistics} \textbf{39}.

\bibitem[{Sussman et~al.(2012)Sussman, Tang, Fishkind \& Priebe}]{sussman2012}
\textsc{Sussman, D.~L.}, \textsc{Tang, M.}, \textsc{Fishkind, D.~E.} \& \textsc{Priebe, C.~E.} (2012).
\newblock A consistent adjacency spectral embedding for stochastic blockmodel graphs.

\bibitem[{Toh \& Horimoto(2002)}]{toh2002}
\textsc{Toh, H.} \& \textsc{Horimoto, K.} (2002).
\newblock Inference of a genetic network by a combined approach of cluster analysis and graphical gaussian modeling.
\newblock \textit{Bioinformatics} \textbf{18}, 287--297.

\bibitem[{Vall\`es-Catal\`a et~al.(2016)Vall\`es-Catal\`a, Massucci, Guimer\`a \& Sales-Pardo}]{MultilayerSBM}
\textsc{Vall\`es-Catal\`a, T.}, \textsc{Massucci, F.~A.}, \textsc{Guimer\`a, R.} \& \textsc{Sales-Pardo, M.} (2016).
\newblock Multilayer stochastic block models reveal the multilayer structure of complex networks.
\newblock \textit{Phys. Rev. X} \textbf{6}, 011036.

\bibitem[{von Luxburg(2007)}]{vonluxburg2007}
\textsc{von Luxburg, U.} (2007).
\newblock A tutorial on spectral clustering.

\bibitem[{Xie \& Zhang(2025)}]{Xie-AOS2025}
\textsc{Xie, F.} \& \textsc{Zhang, Y.} (2025).
\newblock {Higher-order entrywise eigenvectors analysis of low-rank random matrices: Bias correction, Edgeworth expansion and bootstrap}.
\newblock \textit{The Annals of Statistics} \textbf{53}, 1667 -- 1696.

\bibitem[{Yang \& Wang(2018)}]{Yang2018}
\textsc{Yang, Y.} \& \textsc{Wang, H.} (2018).
\newblock Multi-view clustering: A survey.
\newblock \textit{Big Data Mining and Analytics} \textbf{1}, 83--107.

\bibitem[{Zhu \& Ghodsi(2006)}]{ZHU2006918}
\textsc{Zhu, M.} \& \textsc{Ghodsi, A.} (2006).
\newblock Automatic dimensionality selection from the scree plot via the use of profile likelihood.
\newblock \textit{Computational Statistics \& Data Analysis} \textbf{51}, 918--930.

\end{thebibliography}


\newpage


\appendix
\appendixone

\section{Proofs of the statements in the paper.\ }
\label{sec:proofs}
\setcounter{equation}{0} 

\renewcommand{\theequation}
{A\arabic{equation}}

\subsection{Proof of Theorem~\ref{th:number_clust_errors}.\ }
\label{sec:proof_Th1}
For simplicity, denote
\begin{equation} \label{eq:X-hatX1}
X = Z^{(1)} \tkr Z^{(2)} =   Z\,H^\top, \quad
\hatX =   \hatZ^{(1)} \tkr \hatZ^{(2)}, \quad 
\Xi = \hatX - X, 
\end{equation}
where matrix $H$ is defined in \eqref{eq:TKR}.
Let 
\begin{equation}  \label{eq:X_SVD1}
X = U D H^\top, \quad \hatX
  = \hatU\,\hatD\,\hatV^\top
    + \hatU^\perp\,\hatD^\perp\,(\hatV^\perp)^\top
\end{equation}
be the SVDs of $X$ and $\hatX$. Here, $D=\diag  \left(\sqrt{n_1},\dots,\sqrt{n_K}\right)$.

Recall that $J_v$ is the set of elements mis-clustered by $\hatZ^{(v)}$ and 
$\calE_{v,n}$, is their number, $v=1,2$. Then $J_v^c$ is the set of correctly clustered elements. 
Denote, $J= J_1 \cup J_2$, $J^{c} = J_1^{c} \cap J_2^{c}$ and observe that, if $i \in J^c$, then $\Xi(i,:) = 0$.
Here, due to  $|J_v^{c}| = n - \calE_{v,n}$,  one has 
$|J^{c}| \geq n - (\calE_{1,n}+\calE_{2,n})$.
In order to evaluate the clustering accuracy of Algorithm~\ref{alg:KRAFTY}, observe that,
by Assumption~\ref{assump:balancedclusters} 
\begin{equation} \label{eq:sigK1}
\sigma_K(X) = \min_{k \in [K]}\, \sqrt{n_k}
 \geq \frac{c_o\sqrt{n}}{\sqrt{K}} .
\end{equation}
Since $\left\|\hatX(i,:)-X(i,:)\right\|^2 = 0$ if $i \in J^c$, and
$\left\|\hatX(i,:)-X(i,:)\right\|^2 = 2$ if $i \in J$, we have
\begin{equation} \label{eq:Frob1}
\|\hatX - X\|_F^2 = 2\, |J|   \leq 2(\calE_{1,n} + \calE_{2,n}).
\end{equation}
Therefore, $\|\Xi\| = \|\hatX - X\|  \leq \sqrt{2(\calE_{1,n} + \calE_{2,n})}$. Consistency of 
the individual clusterings (Assumption~\ref{assump:consistentZs}) implies that 
$n^{-1}\calE_{v,n}\rightarrow0$.
Then, for $n$ large enough, one has  
\begin{equation} \label{eq:sig_K_hatX}
    \sigma_K(\hatX) \geq \sigma_K(X) - \|\hatX - X\|
     \geq \frac{c_o\sqrt{n}}{\sqrt{K}} \left(1 - \frac{\sqrt{2K}}{c_o} \frac{\sqrt{\calE_{1n} + \calE_{2n}}}{\sqrt{n}}\right) \geq 
     \frac{c_o\sqrt{n}}{2\,\sqrt{K}}.
\end{equation}
Observe that, by Assumption~\ref{assump:consistentZs}, as $n \to \infty$,
\begin{equation} \label{eq:eps0n}
    \epsilon_{o,n}: = \frac{\sqrt{\calE_{1n} + \calE_{2n}}}{\sqrt{n}} = o(1). 
\end{equation}
Then, by Wedin theorem and Assumptions~\ref{assump:finiteclusternumbers} and \ref{assump:consistentZs}, as $n \to \infty$,
\begin{equation} \label{eq:hatUHbound}
\max \left( \|\hatU - U\, W_U \|^2_F, \|\hatV - H\, W_V \|^2_F \right)
    \leq \frac{2\, \|\Xi\|^2_F}{\sigma_K^2(X)} 
    \leq \frac{4K\, (\calE_{1,n} + \calE_{2,n})}{c_o^2\, n} \asymp \epsilon_{o,n}^2 = o(1).
\end{equation}
Here, $W_U$ and $W_V$ are the rotations which provides the closest match between $\hatU$ and $U$,
and $\hatV$ and $H W_V$, respectively.
Specifically, if $U^\top \hatU = W_1 D_W W_2^\top$
is the SVD of $U^\top \hatU$, then $W_U = W_1 W_2^\top$.

In order to assess the number of clustering errors on the basis of $\hatX$, we  
apply Lemma~D1 of \cite{abbe_fan_AOS2022}), which we provide here for readers' convenience.


\begin{lemma}\label{lem:abbe_fan}   
{\bf (Lemma~D1 of \cite{abbe_fan_AOS2022})}. \  Let matrix $B \in \RR^{r \times m}$ with rows $B(k,:)$, $k \in [K]$, 
be the matrix of true means and $z:  [n] \to [K]$ be the true clustering function. 
For a data matrix $\hatX \in \RR^{n \times m}$, any matrix $\tilB \in \RR^{K\times m}$ and any clustering 
function $\tilz: [n] \to  [K]$, define 
\begin{equation} \label{eq:L_function}
L \left( \tilB, \tilz \right) = \sum_{i=1}^n \Big\|\hatX(i,:) -  \tilB(\tilz(i),:) \Big\|^2.
\end{equation}
Let $\hatB \in \RR^{K \times m}$ and $\hatz [n] \to  [K]$  be solutions to the 
$(1+\epsilon)-$approximate k-means problem, i.e.
$$
L \left( \hatB, \hatz \right) \leq (1+\epsilon)\,  \min_{\tilB, \tilz}\    L \left( \tilB, \tilz \right).
$$  
Let $s = \displaystyle \min_{i  \neq j} \|B(i,:) - B(j,:)\|$ and 
$\nmin$ be the minimum cluster size.  
 If for some $\delta \in (0, s/2)$ one has
\begin{equation} \label{eq:abbe_del_cond}
L \left( B, z \right) = \sum_{i=1}^n \Big\|\hatX(i,:) -  B(z(i),:) \Big\|^2 \leq \frac{\delta^2 \, \nmin}{K\, (1 + \sqrt{1+\epsilon} )^2},
\end{equation} 
then there exists a permutation $\phi: [K] \to [K]$ such that
\begin{align}  \label{eq:cl_err1}
& \left\{ i:\,  \hatz(i) \neq \phi(z(i)) \right\} \subseteq \Big \{ i:\, \|\hatX(i,:) - B(z(i),:)\| \geq s/2 - \delta \Big\}, \\
\label{eq:cl_err2}
& \# \left\{ i:\,  \hatz(i) \neq \phi(z(i)) \right\}  \leq (s/2 - \delta)^{-2}\, L \left( B, z \right). 
\end{align}
\end{lemma}

\medskip  \medskip

\noindent
We apply Lemma~\ref{lem:abbe_fan} with 
$s = \displaystyle  \min_{i \neq j} \|U(i,:) -  U(j,:)\|\geq \displaystyle (2/ \max (n_k))^{1/2}$, $X = U\, W_U$, $B(z(i),:) = X(i,:)$   and  
$ L \left( B, z \right) \leq \left\|\hatU - U\,W_U\right\|^2_F$. Therefore, 
$$
\delta^2  \leq \frac{  \|\hatU - U\, W_U \|^2_F \, K\, (1 + \sqrt{1+\epsilon} )^2}{\min (n_k)},
$$
so that, by \eqref{eq:hatUHbound}, we obtain
\begin{equation} \label{eq:s_del_Z}
\frac{s}{2} \geq \frac{\sqrt{K}}{C_o \, \sqrt{2n}}, \quad
\delta \leq \frac{2 K \, (1 + \sqrt{1+\epsilon})}{c_o^{3/2}}\, \frac{\sqrt{\calE_{1,n} + \calE_{2,n}}\, \sqrt{K}}{n}. 
\end{equation} 
By Lemma~\ref{lem:abbe_fan}, row $i$ is clustered correctly, if $ \| e_i^\top (\hatU - UW_U)   \| < s/2 - \delta$,
where
\begin{equation*}
\frac{s}{2}-\delta \geq \frac{\sqrt{K}}{C_o \, \sqrt{2n}} \, \left( 1 - \frac{2\sqrt{2}\, C_o K \, (1 + \sqrt{1+\epsilon}) }
{c_o^{3/2}}\, \frac{\sqrt{\calE_{1,n} + \calE_{2,n}}}{\sqrt{n}} \right), 
\end{equation*}
and the second term in the parenthesis is $o(1)$ as $n \to \infty$ due to Assumption~\ref{assump:consistentZs}. Hence, 
correct clustering of row $i$ is guaranteed by 
\begin{equation}  \label{eq:clust_row_i_byZ}
     \| e_i^\top (\hatU - UW_U)   \| = o \left(   \sqrt{K}/\sqrt{n}  \right).   
\end{equation}
Now, we consider rows with $i \in J^{c}$ and show that all these rows are clustered correctly. 
For  $i \in J^{c}$, $e_i^\top \Xi = \Xi(i,:)=0$.

From \eqref{eq:X-hatX1} derive that
\begin{equation}  \label{eq:R_1234}
R_i:= \| e_i^\top (\hatU - U W_U)   \| \leq R_{i1} + R_{i2} + R_{i3} + R_{i4} + R_{i5}
\end{equation}
with $\Xi = \hatX - X$ and
\begin{align}  \label{eq:R1_R2}
& R_{i1} = \| e_i \Xi H W_V \hatD^{-1}   \| = 0, \quad
 R_{i2} =  \| e_i^\top \Xi (\hatV - H W_V) \hatD^{-1}  \|=0, \\
& R_{i3} = \| e_i^\top U U^\top \Xi H W_V \hatD^{-1} \|, \quad
 R_{i4} =  \| e_i^\top U U^\top \Xi (\hatV - H W_V) \hatD^{-1}  \|, \nonumber\\
& R_{i5} =  \| e_i^\top U (U^\top \hatU -W_U) \|. \nonumber
\end{align}
Observe that, for $n$ large enough, by \eqref{eq:sig_K_hatX}, one has 
\begin{equation*}   
\| \hatD^{-1}\| = \sigma_{\min}^{-1}(\hatX) \leq   C\, \sqrt{\frac{K}{n}}.
\end{equation*}
Therefore, using \eqref{eq:Frob1}, \eqref{eq:eps0n} and  \eqref{eq:hatUHbound}, as $n \to \infty$, obtain
\begin{align}  \label{eq:R3}
R_{i3} & \leq \|U\|_{2, \infty} \, \|\Xi\| \, \| \hatD^{-1}\|  \leq C \,  (K/\sqrt{n}) \, \epsilon_{o,n} = o (\sqrt{K}/\sqrt{n}), \\
\label{eq:R4}
R_{i4} & \leq  \|U\|_{2, \infty} \, \|\Xi\| \, \| \hatD^{-1}\| \, \|\hatV - H\, W_V \| \leq 
C \,  (K/\sqrt{n})  \, \epsilon_{o,n}^2 = o (\sqrt{K}/\sqrt{n}), \\
\label{eq:R5}
 R_{i5} & \leq   \| e_i^\top U (U^\top \hatU -W_U) \| \leq C\,  (\sqrt{K}/\sqrt{n}) \, \epsilon_{o,n}^2 = o (\sqrt{K}/\sqrt{n}).
\end{align} 
Combination of \eqref{eq:R_1234}--\eqref{eq:R5} implies validity of \eqref{eq:clust_row_i_byZ} and,
therefore, guarantees that, for $n$ large enough, all elements in $J^{c}$ are clustered correctly.
Then \eqref{eq:number_clust_errors} is true and Theorem~\ref{thm:consistentZhats} holds.

\medskip \medskip


\subsection{Proof of Theorems~\ref{th:consistent_U}~and~\ref{th:perfect_U} .\ }
\label{sec:proof_Th2_3}

Without loss of generality, we suppose that Assumption~\ref{assump:incoherentUhats} holds with $v=2$.
In order to assess the number of clustering errors on the basis of $\hatU$, we  apply  Lemma~\ref{lem:abbe_fan} with 
$s = \displaystyle  \min_{z(i) \neq z(j)} \|U(i,:) -  U(j,:)\|$, $X = U\, W_U$, $B(z(i),:) = X(i,:)$   and  
$ L \left( B, z \right) \leq \left\|\hatU - U\,W_U\right\|^2_F$. Therefore, 
\begin{equation} \label{eq:s_lower_Bound}
 s \geq \frac{\sqrt{2}}{\sqrt{\max (n_k)}}   \geq \frac{\sqrt{2K}}{C_o \sqrt{n}}, \quad 
 \delta  \leq \frac{ \|\hatU - U\,W_U \|_F \times \sqrt{K}\, (1+\sqrt{1+\epsilon})}{\sqrt{\min (n_k)}}.  
\end{equation}
Recall that $U = \SVD_K(U^{(1,2)})$ and $\hatU = \SVD_K(\hatU^{(1)} \tkr \hatU^{(2)})$,
where $U^{(1,2)} = U^{(1)} \tkr U^{(2)}$ and  $\hatU^{(1,2)}   = \hatU^{(1)} \tkr \hatU^{(2)}$.
Recall also that, by \eqref{eq:U12_SVD},  
\begin{equation} \label{eq:U12_hatU12}
U^{(1,2)} =    U \, D  V^\top,  \quad V \equiv H, \  \  D = \bigl(D_Z \tilD_Z^{-1}\bigr)^{1/2}.
\end{equation}
Here, $D_Z = Z^\top Z$, so that, for $k_1 \in [K_1]$, $k_2 \in [K_2]$, using double-index, one can write 
\begin{equation*}
 D_Z(k,k) = n_k, \quad 
 \tilD_Z(k,k) = n_{k_1}^{(1)} n_{k_2}^{(2)} I(H((k_1,k_2),k) = 1), \quad k\in[K]. 
 \end{equation*}
Therefore, 
\begin{equation} \label{eq:Dkk}
    D(k,k) = \frac{\sqrt{n_k}}{\sqrt{n_{k_1}^{(1)} n_{k_2}^{(2)}}} I\left(H((k_1,k_2),k) = 1 \right).    
\end{equation}
By Assumption~\ref{assump:balancedclusters},
\begin{equation}  \label{eq:D-balance}
    \frac{ c_D\, \sqrt{K}}{\sqrt{K_1 K_2} \sqrt{n}} \leq D(k,k) \leq \frac{C_D \sqrt{K}}{\sqrt{K_1 K_2}\sqrt{n}},    
\end{equation}
where $c_D =  c_o/(C_o^{(1)} C_o^{(2)})$ and $C_D =  C_o/(c_o^{(1)} c_o^{(2)})$.
The latter, together with \eqref{eq:U12_hatU12} yields
\begin{equation} \label{eq:sig_min_X}
    \sigma_{\min} (U^{(1,2)}) \geq   \frac{ c_D\, \sqrt{K}}{\sqrt{K_1 K_2} \sqrt{n}}.    
\end{equation}

Write the SVD of $\hatU^{(1,2)}$  as
\begin{equation*} \label{eq:hatU12_SVD}
   \hatU^{(1,2)} =  \hatU\,\hatD\,\hatV^\top
    + \hatU^\perp\,\hatD^\perp\,(\hatV^\perp)^\top.
\end{equation*}
Denote $\Xi^{(v)} = \hatU^{(v)} - U^{(v)}W_U^{(v)}$, $v=1,2$, and
$\Xi = \hatU^{(1,2)} -  U^{(1,2)} W_U^{(1,2)}$, where  $W_U^{(v)}$ provide the closest matches in the Frobenius norm, and $W_U^{(1,2)}=W_U^{(1)}\otimes W_U^{(2)}$. Let $W_V$
be a rotation matrix which provides the closest match  in the Frobenius norm
between $\hatV$ and $V W_V$.

Recall that, by Assumptions \ref{assump:finiteclusternumbers} and \ref{assump:consistentUs}, as $n \to \infty$, one has
\begin{equation}  \label{eq:frob_err_v}
     \left\| \Xi^{(v)} \right\|_F  = \left\| \hatU^{(v)} - U^{(v)} W_U^{(v)} \right\|_F    
    \leq C_{\tau} \sqrt{K_v} \calE_{0,n}^{(v)} = o(1),\quad v=1,2.
\end{equation}
Hence, 
\begin{equation}  \label{eq:Xi_decomposition}
\Xi = \Xi^{(1)} \tkr \hatU^{(2)} + (U^{(1)} W_U^{(1)}) \tkr \Xi^{(2)}
\end{equation}
and, by property {\bf P2} of the Khatri-Rao product one has
\begin{equation*}
   \| \Xi  \|_F 
    \leq  \| \Xi^{(1)}  \|_F  \| \hatU^{(2)}  \|_{2,\infty}
    +  \| \Xi^{(2)}  \|_F   \| U^{(1)}  \|_{2,\infty}
\end{equation*}
Then Assumptions \ref{assump:balancedclusters} and \ref{assump:incoherentUhats} and \eqref{eq:frob_err_v} yield that, with probability at least 
$1 - 3 n^{-\tau}$,  
\begin{equation}  \label{eq:Xi_F}
     \| \Xi  \|_F
    \leq \frac{\tilC_{\tau}\, \sqrt{K_1} \sqrt{K_2}}{\sqrt{n}}
 \left( \calE_{0,n}^{(1)}  + \calE_{0,n}^{(2)}  \right),      
\end{equation}
where $\tilC_{\tau}$ is a generic absolute constant that depends on $\tau$ and constants in Assumption \ref{assump:balancedclusters}. 
It is easy to see that the second inequality in  \eqref{eq:s_lower_Bound} and \eqref{eq:Xi_F} immediately yield
\begin{equation}  \label{eq:del_upbound1}
    \delta  \leq \frac{\tilC_{\tau}(1+\sqrt{1+\epsilon})\,K \sqrt{K_1\, K_2}\, (\calE_{0,n}^{(1)} + \calE_{0,n}^{(2)})  }{c_o n}  
\end{equation}
Applying the first inequality in \eqref{eq:s_lower_Bound}, we find that
\begin{equation} \label{eq:half_s-del}
    \frac{s}{2} - \delta 
    \geq \frac{\sqrt{K}}{C_o\sqrt{2n}} \left(1 -  \frac{\tilC_{\tau}\, C_o\, (1+\sqrt{1+\epsilon})\, \sqrt{2K K_1 K_2}\;(\calE_{0,n}^{(1)} + \calE_{0,n}^{(2)})}{c_o\sqrt{n}}    \right). 
\end{equation}
If $n$ is large enough, by the conditions of the theorem, one has
\begin{equation*}
    \frac{s}{2} - \delta \geq \frac{\sqrt{K}}{2\sqrt{2}\, C_o \, \sqrt{n}}.
\end{equation*}
Hence, by Lemma~\ref{lem:abbe_fan}, row $i$ is guaranteed to be clustered correctly if
\begin{equation}  \label{eq:clust_cond_row_i}
     \| e_i^\top (\hatU - UW_U)   \| < 
   \frac{\sqrt{K}}{2\sqrt{2}\, C_o \, \sqrt{n}}.  
\end{equation}
In order to  utilize \eqref{eq:clust_cond_row_i}, we derive an upper bound for $ \| e_i^\top (\hatU - UW_U)   \|$. We expand the quantity $\Xi VW_V\hatD^{-1}$ as
\begin{align*}
\Xi V W_V \hatD^{-1} &= UU^\top \Xi VW_V\hatD^{-1} -(I-UU^\top)\Xi(\hatV-VW_V)\hatD^{-1}+(I-UU^\top)\Xi\hatV\hatD^{-1}\\
&=UU^\top\Xi VW_V\hatD^{-1}-(I-UU^\top)\Xi(\hatV-VW_V)\hatD^{-1}+(I-UU^\top)(\hatU-UW_U),
\end{align*}
making use of the fact that $(I-UU^\top)U^{(1,2)}=0$, and $\hatU^{(1,2)}\hat{V}\hat{D}^{-1}=\hatU$. 
Then for any $i \in [n]$, one has 
\begin{equation} \label{eq:R_i} 
  R_i = \| e_i^\top (\hatU - U W_U) \| \leq R_{i1} + R_{i2} + R_{i3} + R_{i4} + R_{i5}, \quad i \in [n].
\end{equation}
where
\begin{align*}
& R_{i1} = \bigl\| e_i^\top \Xi V W_V \hatD^{-1} \bigr\|,  
& R_{i2} = \bigl\| e_i^\top \Xi (\hatV - V W_V) \hatD^{-1} \bigr\|,\\
& R_{i3} = \bigl\| e_i^\top U U^\top \Xi V W_V \hatD^{-1}\bigr\|,  
& R_{i4} = \bigl\| e_i^\top U U^\top \Xi (\hatV - V W_V) \hatD^{-1} \bigr\|,\\
& R_{i5} = \bigl\| e_i^\top U (U^\top \hatU -W_U) \bigr\|.
\end{align*}
Observe that, by \eqref{eq:D-balance},
\begin{equation}   \label{eq:hatD_kk_lower0}
  \min_k (\hatD (k,k)) \geq \min_k (D (k,k)) - \| \Xi \| 
  \geq   \frac{ c_D\, \sqrt{K}}{\sqrt{K_1 K_2} \sqrt{n}} \left(1 -  
\frac{\tilC_\tau K_1 K_2  ( \calE_{0,n}^{(1)}  + \calE_{0,n}^{(2)})}{c_D\, \sqrt{K}} \right).
\end{equation}
Therefore, if $n$ is large enough, with high probability,
\begin{equation}  \label{eq:hatD_kk_lower}
  \min_k (\hatD (k,k)) \geq \frac{c_D\, \sqrt{K}}{2\, \sqrt{K_1 K_2}\, \sqrt{n}},
\end{equation}
and we have a similar upper bound with high probability
\begin{equation}  \label{eq:hatD_kk_upper}
  \max_k (\hatD (k,k)) \leq \frac{2\, C_D\, \sqrt{K}}{\sqrt{K_1 K_2}\, \sqrt{n}}.
\end{equation}
Since $e_i^\top \Xi = \Xi(i,:)$,  \eqref{eq:hatD_kk_lower} and \eqref{eq:hatD_kk_upper}    lead to
\begin{equation}  \label{eq:e_i_Xi}
 R_{i1} = \| e_i^\top \Xi V W_V \hatD^{-1}  \|  
\leq  \| \Xi (i,:)  \|\, \frac{2 \sqrt{K_1K_2}\sqrt{n}}{c_D\sqrt{K}}.
\end{equation}
By the Davis-Kahan theorem and \eqref{eq:sig_min_X}, we find
\begin{equation}\label{eq:vbound}
     \|\hatV - V W_V \|_F 
    \leq \frac{C  \| \Xi  \|_F}{\sigma_{\min}(U^{(1,2)})}
    \leq \tilC_{\tau} \frac{K_1K_2}{\sqrt{K}} ( \calE_{0,n}^{(1)}  + \calE_{0,n}^{(2)}).
\end{equation}
Hence, by \eqref{eq:e_i_Xi}, with high probability, as $n \to \infty$, 
\begin{equation}  \label{eq:R_i2}
    R_{i2}  \leq \frac{2\tilC_{\tau}}{c_D}  \| \Xi (i,:) \|  \frac{(K_1K_2)^{3/2}\sqrt{n}}{K} \,  ( \calE_{0,n}^{(1)}  + \calE_{0,n}^{(2)}) = o \left( \| \Xi (i,:)  \|\, \frac{ \sqrt{K_1K_2}\sqrt{n}}{\sqrt{K}}  \right).
\end{equation}
Now, Assumption \ref{assump:balancedclusters} and \eqref{eq:hatD_kk_lower}  yield with high probability, if $n$ is large enough, 
\begin{equation*}   
    R_{i3}  \leq  \| U  \|_{2,\infty}\,  \|\Xi \|  \,  \| \hatD^{-1}  \|   
     \leq  2\, (c_D\, c_o)^{-1}\,  \sqrt{K_1  K_2} \,  \|\Xi \|.      
\end{equation*}
Plugging in $\| \Xi  \|_F$ from \eqref{eq:Xi_F}, we obtain
\begin{equation}  \label{eq:R_i3}
 R_{i3}  \leq    
\frac{2\tilC_{\tau}\,  K_1\,  K_2}{c_D c_o\sqrt{n}}
 \left( \calE_{0,n}^{(1)}  + \calE_{0,n}^{(2)}  \right).
\end{equation}
Similarly, with high probability, as $n \to \infty$, 
\begin{equation} \label{eq:R_i4}
    R_{i4}      \leq  \| U  \|_{2,\infty}    \|\Xi \|    \| \hatD^{-1}  \|     \|\hatV - V W_V \| =o(\|U\|_{2,\infty}\|\Xi\|\|\hatD^{-1}\|),
\end{equation}
which is upper-bounded in \eqref{eq:R_i3}. 
Finally,
\begin{equation}  \label{eq:R_i5}
    R_{i5} \leq C  \| U  \|_{2,\infty}  \| \hatU - U W_U \|^2.
\end{equation}
The bound in Equation~\eqref{eq:vbound} applies to $\hat{U}-UW_U$ as well, so we have
\begin{equation} \label{eq:R_i5_1}
    R_{i5} \leq C \frac{\sqrt{K}}{c_o\sqrt{n}} \tilde{C}_\tau^2\frac{(K_1 K_2)^2}{K}(\calE_{0,n}^{(1)}+\calE_{1,n}^{(2)})^2
\end{equation}
Since $\calE_{0,n}^{(v)}=o(1)$ by assumption of the theorem, this upper bound is much smaller than that of Equation~\eqref{eq:R_i3}.
Hence, for $n$ large enough and $i \in [n]$, combination of \eqref{eq:R_i}--\eqref{eq:R_i5_1}   yields
\begin{equation} \label{eq:row_err_i} 
\| e_i^\top (\hatU - U W_U) \| \leq  C_{\tau}\,  \left(   \| \Xi (i,:)  \|\, \frac{\sqrt{K_1K_2}\sqrt{n}}{\sqrt{K}}   +  
\frac{\sqrt{K_1K_2}}{\sqrt{n}} \, (\calE_{0,n}^{(1)}+\calE_{1,n}^{(2)})
\right),
\end{equation}
where the second term in \eqref{eq:row_err_i}  is asymptotically smaller than $\sqrt{K}/\sqrt{n}$.

Thetefore, validity of inequality \eqref{eq:clust_cond_row_i} relies entirely on the behavior of $\|\Xi (i,:)\|$.
Observe that, by \eqref{eq:Xi_decomposition}, one has
\begin{align*}
     \Xi (i,:)     
    & = \left(U^{(1)}(i,:)W_U^{(1)}\right)\tkr \Xi^{(2)}(i,:) + \Xi^{(1)}(i,:)\tkr \hatU^{(2)}(i,:).  
\end{align*}
Hence,  due to Assumptions~\ref{assump:balancedclusters} and \ref{assump:incoherentUhats}, 
\begin{equation}  \label{eq:Xi_row_norm}
     \| \Xi (i,:)  \|     \leq  
\frac{C_{\tau} \,\sqrt{ K_2}}{\sqrt{n}}  \| \Xi^{(1)}(i,:)  \| + \frac{\sqrt{K_1}}{c_o^{(1)}\sqrt{n}}  \| \Xi^{(2)}(i,:)  \|. 
\end{equation} 
Then with high probability for large $n$,
\begin{equation*}
    \|\Xi(i,:)\|\leq C_\tau  \left( 
    \frac{C_{\tau} \,\sqrt{ K_2}}{\sqrt{n}}\,  \calE_{2,\infty,n}^{(1)} + \frac{\sqrt{K_1}}{c_o^{(1)}\sqrt{n}}  \calE_{2,\infty,n}^{(2)} 
    \right).
\end{equation*}
Thus, assumption of Theorem~\ref{th:perfect_U}   implies that,
for $n$ large enough and any absolute constant $\Cksmall$
\begin{equation} \label{eq:xi_i_bound}
\|\Xi(i,:)\| \leq   \frac{\sqrt{K_1 K_2}}{n} \, \Cksmall,
\end{equation}
where $\Cksmall$ is a small positive constant to be chosen later.
%
From Equations~\eqref{eq:row_err_i}  and \eqref{eq:xi_i_bound}, obtain that 
\begin{align}
\|e_i^\top(\hatU-UW_U)\| &\leq  C_{\tau}\,  \left(   \| \Xi (i,:)  \|\, \frac{\sqrt{K_1K_2}\sqrt{n}}{\sqrt{K}}   +  
\frac{\sqrt{K_1K_2}}{\sqrt{n}} \, (\calE_{0,n}^{(1)}+\calE_{1,n}^{(2)})
\right) \notag \\
 &\leq  C_{\tau}\, \frac{\sqrt{K}}{\sqrt{n}} \, \left[( ]
\frac{K_1 K_2}{K}\,  \Cksmall + \frac{\sqrt{K_1 K_2}}{\sqrt{K}}\, \left(\calE_{0,n}^{(1)}+\calE_{0,n}^{(2)}\right) 
 \right] \notag \\
  &\leq \frac{\sqrt{K}}{2\sqrt{2}C_o\sqrt{n}}, \label{eq:u_i_bound}
\end{align}
once $n$ is sufficiently large, and assuming a suitably small choice of $\Cksmall$. This implies that condition \eqref{eq:clust_cond_row_i} holds for any  $i \in [n]$, so Theorem~\ref{th:perfect_U} has been proved.\\

\medskip 

\noindent
Now, we finish the proof of Theorem~\ref{th:consistent_U}. 
Let
\begin{equation} \label{eq:calJ_v_n}
    \calJ_n^{(v)} \equiv \calJ_n^{(v)}  ( C_{\text{small}} ) = \left\{i\in [n] : \|\Xi^{(v)}(i,:)\| > \frac{C_{\text{small}}\sqrt{K_v}}{\sqrt{n}} \right\}, \quad v=1,2,
\end{equation}
and 
\begin{equation}
    \calE^{(v)}_{n} = |\calJ_n^{(v)}|, \quad v=1,2.
\end{equation}
If $i \in (\calJ_n^{(1)})^c \cap (\calJ_n^{(2)})^c$, then by \eqref{eq:Xi_row_norm}, the bound \eqref{eq:xi_i_bound}  holds, and row $i$ is clustered correctly by condition \eqref{eq:clust_cond_row_i} and the inequality \eqref{eq:u_i_bound}. Denote
\begin{equation*}
    \calJ_n = \calJ_n^{(1)} \cup \calJ_n^{(2)}, 
    \quad \calJ_n^c = (\calJ_n^{(1)})^c \cap (\calJ_n^{(2)})^c
\end{equation*}
Then,
\begin{equation}  \label{eq:calE_n}
    \calE_{n}  = |\calJ_n| \leq \calE^{(1)}_{n} + \calE^{(2)}_{n}. 
\end{equation}
By Assumption \ref{assump:consistentUs},
\begin{equation}
    \PP \left\{  \| \Xi^{(v)} \|^2_F \leq C_{\tau} K_v (\calE^{(v)}_{0,n})^2\right\} 
    \geq 1 - n^{-\tau}.
    \label{eq:oct20-62}
\end{equation}
and by the conditions of Theorem~~\ref{th:consistent_U}  and Assumption~\ref{assump:finiteclusternumbers}, one has 
\begin{equation}
    K_v (\calE^{(v)}_{on})^2 = o(1), \quad \text{as } n \to \infty.
    \label{eq:oct20-63}
\end{equation}
Then, \eqref{eq:calJ_v_n} yields
\begin{equation}  \label{eq:Xi_v_Fnorm}
     \| \Xi^{(v)}  \|^2_F      = \sum_{i=1}^n  \| \Xi^{(v)}(i,:)  \|^2   
    \geq \sum_{i\in \calJ_n^{(v)}} \| \Xi^{(v)}(i,:)  \|^2 \\ 
     \geq \frac{K_v}{n} |\calJ_n^{(v)}| C_{\text{small}}^2.
\end{equation}
Hence, \eqref{eq:oct20-62}, \eqref{eq:oct20-63} and \eqref{eq:Xi_v_Fnorm}   imply that 
\begin{equation*}
\PP \left\{ \frac{\calE_n^{(v)}}{n} \leq \frac{C_{\tau}}{ C_{\text{small}}^2} (\calE^{(v)}_{0,n})^2 \right\} 
    \geq 1 - n^{-\tau}.
\end{equation*}
Therefore, due to \eqref{eq:calE_n},
\begin{equation}
    \mathbb{P}\left\{ \frac{\calE_n}{n} \leq  \frac{C_{\tau}}{C_{\text{small}}^2} 
\left[(\calE^{(1)}_{on})^2 + (\calE^{(2)}_{on})^2 \right]\right\} 
    \geq 1 - 2n^{-\tau}.
\end{equation}
Hence, $\calE_n/n \to 0$ as $n\to \infty$ with high probability, and clustering is consistent.

\subsection{Proof of Theorem~\ref{th:cl_number_Z} }
\label{sec:proof_Th4}

Let, as in the proof of Theorem~\ref{th:number_clust_errors},
 $\hatX = \hatZ^{(1,2)} = \hatZ^{(1)} \tkr \hatZ^{(2)}$ and $X = Z H^\top$.
Then, it follows from \eqref{eq:sigK1} and \eqref{eq:Frob1} that
$$
\sigma_K(X)   \geq \frac{c_o\sqrt{n}}{\sqrt{K}}, \quad 
\|\hatX - X\|  \leq \sqrt{2(\calE_{1,n} + \calE_{2,n})}.
$$
Observe that, by \eqref{eq:sig_K_hatX}, one has 
\begin{equation}  \label{eq:sig_K_hatX1}
\sigma_K(\hatX)  \geq \frac{c_o\sqrt{n}}{\sqrt{K}} 
\left(1 - \frac{\sqrt{2K}}{c_o} \frac{\sqrt{\calE_{1n} + \calE_{2n}}}{\sqrt{n}}\right).
\end{equation}
In addition, for $1\leq j \leq K_1K_2 - K$,
\begin{equation} \label{eq:sig_K_hatX2}
    \sigma_{K+j}(\hatX) \leq \sigma_{K+j}(X) +  \|\hatX-X  \|  
    \leq \frac{c_o\sqrt{n}}{\sqrt{K}} \frac{\sqrt{2K}}{c_o}\frac{\sqrt{\calE_{1n}+\calE_{2n}}}{\sqrt{n}}.     
\end{equation}
Then, \eqref{eq:cl_num_Z1} and \eqref{eq:cl_num_Z2} follow directly from
\eqref{eq:sig_K_hatX1} and \eqref{eq:sig_K_hatX2}.
If clustering is consistent and Assumptions \ref{assump:balancedclusters}, \ref{assump:finiteclusternumbers}, \ref{assump:consistentZs} hold, then
$K=O(1)$,  $n^{-1}\, K\, (\calE_{1n}+\calE_{2n})=o(1)$, so the theorem is valid.


\subsection{Proof of Theorem~\ref{th:cl_number_U} }
\label{sec:proof_Th5}

Let, as in the proof of Theorems~\ref{th:consistent_U}~and~\ref{th:perfect_U},
\begin{equation*}
X = U^{(1,2)}W_U^{(1,2)} = (U^{(1)} \tkr U^{(2)} )W_U^{(1,2)}  =  U D H^\top W_U^{(1,2)}, \quad
\hatX = \hatU^{(1,2)} = \hatU^{(1)} \tkr \hatU^{(2)}
\end{equation*}
and $\Xi = \hatX - X$. Recall that by \eqref{eq:sig_min_X}, 
\begin{equation*}
    \sigma_K(X) \geq  \frac{c_D\, \sqrt{K}}{\sqrt{K_1  K_2\, n}}, \quad  \sigma_{K+1}(X)=0.
\end{equation*}
In addition, it follows from \eqref{eq:hatD_kk_lower0} that, under Assumptions \ref{assump:balancedclusters}, \ref{assump:finiteclusternumbers}, \ref{assump:consistentUs}, and \ref{assump:incoherentUhats},
with high probability
\begin{equation*}
    \sigma_K(\hatX)   \geq  \frac{ c_D\, \sqrt{K}}{\sqrt{K_1 K_2} \sqrt{n}} \left(1 -  
\frac{\tilC_\tau K_1 K_2  ( \calE_{0,n}^{(1)}  + \calE_{0,n}^{(2)})}{c_D\, \sqrt{K}} \right).
\end{equation*}
Moreover, by \eqref{eq:Xi_F}, one has
\begin{equation*}
\| \Xi  \|_F \leq \frac{\tilC_{\tau}\, \sqrt{K_1} \sqrt{K_2}}{\sqrt{n}}
 \left( \calE_{0,n}^{(1)}  + \calE_{0,n}^{(2)}  \right).
\end{equation*}
Then, validity of the theorem follows from the fact that 
$\sigma_{K+j}(\hatX) \leq  \| \Xi  \|$ for  $j \geq 1$.

\subsection{Proof of Hierarchical Clustering results}

In this section, we prove that complete-linkage hierarchical clustering applied to the rows of $\hatU = \SVD_K(\hatU^{(1)} \tkr \hatU^{(2)})$ yields perfect clustering under the conditions of Theorem~\ref{th:perfect_U}, and that the merge heights exhibit an elbow at step $n-K+2$. Algorithm~\ref{alg:HC} shows the complete-linkage hierarchical clustering procedure.

\textbf{Proof of Theorem \ref{th:perfect_hc}:} Let us first describe the strategy to the proof. In hierarchical clustering with complete linkage, we begin with $n$ clusters each containing a single point, and at each merge, the number of clusters decreases by 1. Suppose the radius of the set of points belonging to each true cluster is less than the separation between true clusters. Then when points belong to the same true cluster, we expect the merge heights to be small, and only once all of the possible within-cluster merges have occurred will a merge bring together points that belong to distinct true clusters. This will happen at step $n-K+1$, since at that step there will be $n-(n-K+1)=K-1$ clusters, necessitating that one contains points from multiple true clusters, and we will see a jump in the merge heights at this step. Now let us prove these claims.

Let $\calG^{(t)}$ denote the partition obtained after $t$ merges, and let $G_i^{(t)}$ denote the $i$-th cluster in this partition. When merging two clusters, there are two cases. First, if all of the points in the two clusters $G_i^{(t-1)}$, $G_j^{(t-1)}$ belong to the same true cluster, then for any $a \in G_i^{(t-1)}$ and $b \in G_j^{(t-1)}$ we have $z(a)=z(b)$. Hence, $e_a^\top U W_U = e_b^\top U W_U$, so
\begin{equation}
    \|(e_a^\top - e_b^\top)\hatU\| \leq \|e_a^\top (\hatU - UW_U)\| + \|e_b^\top (\hatU - UW_U)\| 
\leq
2\|\hatU - U W_U\|_{2,\infty}
\end{equation}

\begin{equation}
    d(G_i^{(t-1)}, G_j^{(t-1)}) = \max \{ \| \hatU(a,:) - \hatU(b,:) \| : a \in G_i^{(t-1)},\, b \in G_j^{(t-1)} \} \leq 2\|\hatU - U W_U\|_{2,\infty}
\end{equation}

Second, if points in $G_i^{(t-1)}\cup G_j^{(t-1)}$ belong to at least two different true clusters, then there are some indices $a$ and $b$ in this union such that $z(a)\neq z(b)$. We have

\begin{equation}
    \begin{aligned}
    \|(e_a^\top - e_b^\top)\hatU\| &\geq \|(e_a^\top - e_b^\top) U W_U\| - \|(e_a^\top- e_b^\top)(\hatU - U W_U)\| \\
&\geq \|(e_a^\top - e_b^\top) U W_U\| - 2\|\hatU - U W_U\|_{2,\infty}
    \end{aligned}
\end{equation}

Letting $s = \displaystyle \min_{z(i) \neq z(j)} \|U(i,:) - U(j,:)\|$, we have 
\begin{equation}
    \|(e_a^\top - e_b^\top)\hatU\| \geq s - 2\|\hatU - U W_U\|_{2,\infty}
\end{equation}
Equation \eqref{eq:s_lower_Bound} states that $s\geq \frac{\sqrt{2K}}{C_o\sqrt{n}}$. On the other hand, if we fix $\eta\in(0,1/6)$, then choosing an even smaller constant $\Cksmall$ in Equation~\eqref{eq:xi_i_bound} if necessary, we can obtain the following upper bound from Equation~\eqref{eq:u_i_bound}, which holds with high probability for large $n$:
\begin{equation} \label{eq:compareU2inf_to_s}
    \|\hatU-UW_U\|_{2,\infty} \leq \eta\frac{\sqrt{2K}}{C_o\sqrt{n}}\leq \eta s.
\end{equation}

This implies all of the initial merges are within-cluster, with merge heights $h(t) \leq 2\|\hatU - U W_U\|_{2,\infty}$. At the $t$-th stage of the algorithm, there are $n-t+1$ clusters. Therefore, at the $n-K+1$ step, there are exactly $K$ distinct clusters, and 
\begin{equation} \label{eq:initialheights}
h(n-K+1) \leq 2\|\hatU - U W_U\|_{2,\infty}.
\end{equation}
So $G^{(n-K+1)}$ gives the exact cluster labels. This completes the proof of Theorem~\ref{th:perfect_hc}.\\

When we have access to $\hatU$ from Algorithm~\ref{alg:KRAFTY}, we see that for subsequent steps $t \geq n-K+2$, any merge must involve points from different true clusters, yielding
\begin{equation}\label{eq:laterheights}
h(t) \geq s - 2\|\hatU - U W_U\|_{2,\infty}.
\end{equation}
Hence, a gap occurs in the merge heights at step $n-K+2$, which separates within-cluster merges from between-cluster merges and ensures exact recovery. Combining Equations~\eqref{eq:compareU2inf_to_s}, \eqref{eq:initialheights}, and \eqref{eq:laterheights}, we have
\begin{align*}
h(n-K+2)-h(n-K+1)&\geq (1-2\eta)\frac{\sqrt{2K}}{C_o\sqrt{n}}\\
h(k+1)-h(k)&\leq 2\eta \frac{\sqrt{2K}}{C_o\sqrt{n}}
\end{align*}
for $k\leq n-K$. So there is an elbow in the merge heights when we apply Algorithm~\ref{alg:HC} to $\hatU$, but this requires knowledge of $K$ in order to choose the dimension of $\hatU$. To address this issue, we consider applying Algorithm~\ref{alg:HC} to $\hatU^{(1,2)}$ directly. This leads to Theorem~\ref{th:cluster_number_hc}, which we now prove.
\\

\textbf{Proof of Theorem~\ref{th:cluster_number_hc}:} Let $s=\min_{z(i)\neq z(j)} \|U^{(1,2)}(i,:)-U^{(1,2)}(j,:)\|$. By Assumption~\ref{assump:balancedclusters}, we have
\begin{equation} \label{eq:u12_s_lb}
    s= \min_{z(i)\neq z(j)} \left(\frac{1}{n_{z_1(i)}^{(1)} n_{z_2(i)}^{(2)}} + \frac{1}{n_{z_1(j)}^{(1)}n_{z_2(j)}^{(2)}}\right)^{1/2} \geq \frac{\sqrt{2K_1 K_2}}{c_o^{(1)}c_o^{(2)} n}.
\end{equation}
Recall that $\Xi=\hatU^{(1,2)}-U^{(1,2)}W_U^{(1,2)}$, and for large $n$, Equation~\eqref{eq:xi_i_bound} provides the high-probability bound
\begin{equation*}
\|\Xi\|_{2,\infty} \leq   \frac{\sqrt{K_1 K_2}}{n} \, \Cksmall.
\end{equation*}
Let $\eta\in(0,1/6)$. Then choosing a smaller value of $\Cksmall$ if necessary, we can ensure that for large $n$, with high probability, 
\begin{equation*}
    \|\Xi\|_{2,\infty} \leq \eta \frac{\sqrt{2K_1 K_2}}{c_o^{(1)}c_o^{(2)} n}.
\end{equation*}
Putting this together, the first $n-K+1$ merges are within-cluster, with merge heights 
\begin{equation*}
    h(t)\leq 2\|\Xi\|_{2,\infty} \leq 2\eta \frac{\sqrt{2K_1 K_2}}{c_o^{(1)}c_o^{(2)}n},
\end{equation*}
and all subsequent merges involve points from different true clusters, which must have heights
\begin{equation*}
    h(t)\geq s-2\|\Xi\|_{2,\infty}\geq (1-2\eta)\frac{\sqrt{2K_1K_2}}{c_o^{(1)}c_o^{(2)}n}.
\end{equation*}
These inequalities together prove Theorem~\ref{th:cluster_number_hc}.
\\

Figure~\ref{fig:hc_mh} illustrates the idea that a gap occurs in the merge heights at step $n-K+2$. In both subfigures ($n=1000$ and $K=5$), a clear jump is observed between steps 996 and 997.

\begin{figure}[H]
    \centering   
    \begin{minipage}[b]{0.45\linewidth}
        \centering
        \includegraphics[width=\linewidth]{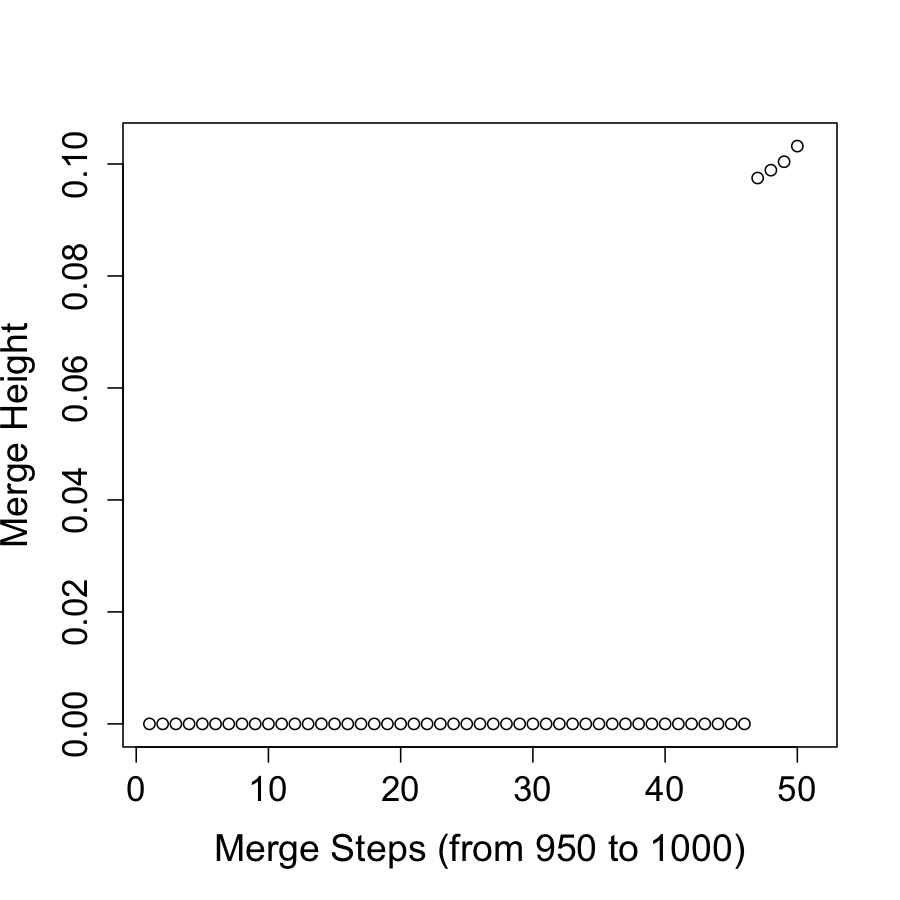}
        \vspace{3pt}
        \centerline{\small (a) Input: $SVD_K (\hatZ^{(1,2)}) $.}
        \label{fig:hc_z}
    \end{minipage}
    \hfill 
    \begin{minipage}[b]{0.45\linewidth}
        \centering
        \includegraphics[width=\linewidth]{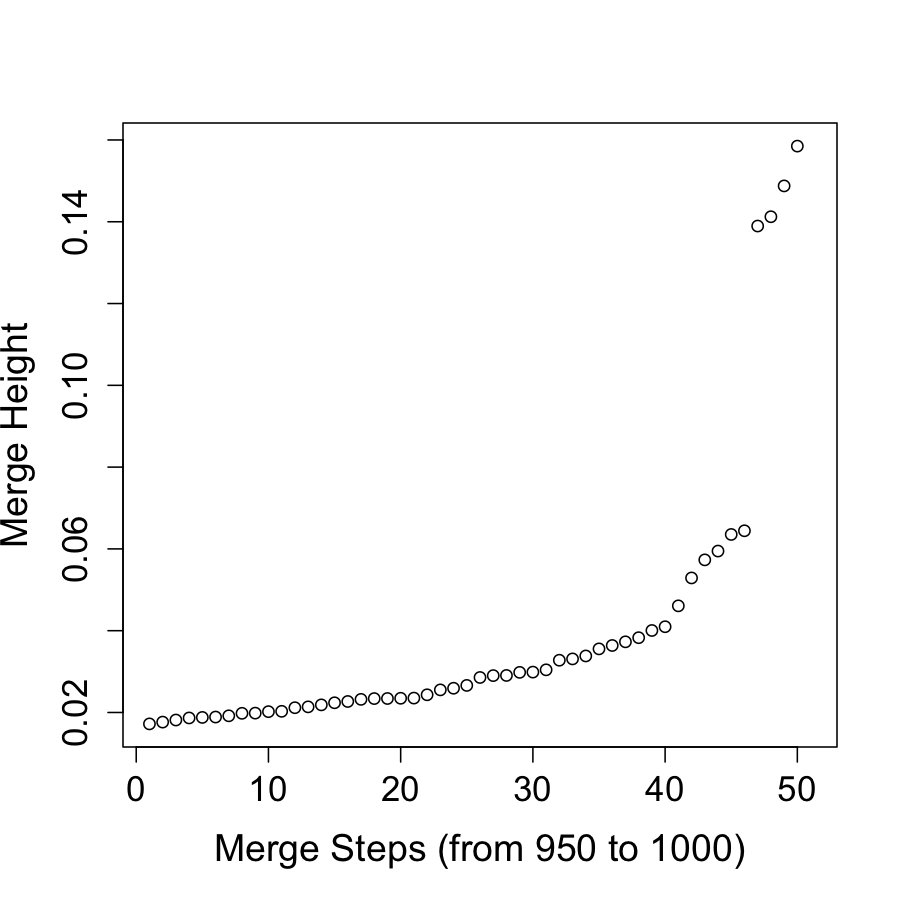}
        \vspace{3pt}
        \centerline{\small (b) Input: $SVD_K (\hatU^{(1,2)}) $ .}
        \label{fig:hc_u}
    \end{minipage}
    
    \caption{Merge height plot for complete-linkage hierarchical clustering with $n=1000$, $K=5$, $K_1=K_2=4$, $m=30$, and $\sigma^2=0.25$.}
    \label{fig:hc_mh}
\end{figure}
\subsection{Properties of the Khatri-Rao product and their proofs}
\label{sec:KR_properties}

\begin{itemize}
    \item[\textbf{P1}.\  ]   $(AC)\tkr(BD) = (A\tkr B)(C \otimes D )$.
 
     
    \item[\textbf{P2}.\  ] $\|A \tkr B\|_F \leq \min\bigl( \|A\|_F \|B\|_{2,\infty}, \|A\|_{2,\infty} \|B\|_F \bigr)$.

    \textbf{Proof:}
      \begin{align*}
      \|A \tkr B\|_F^2 &= \sum_{i=1}^n \|A(i, :) \otimes B(i, :)\|_2^2 \\
      &= \sum_{i=1}^n [A(i, :) \otimes B(i, :)][A(i, :) \otimes B(i, :)]^\top \\
      &= \sum_{i=1}^n \left[ A(i, :) A(i, :)^\top \right] \otimes  \left[ B(i, :) B(i, :)^\top \right] \\
      &\leq \max_j \|A(j, :)\|^2 \sum_{i=1}^n \|B(i, :)\|^2; \text{ or} \\
      &\leq \max_j \|B(j, :)\|^2 \sum_{i=1}^n \|A(i, :)\|^2.
    \end{align*}
 

    \item[\textbf{P3}.\  ] $\|A \tkr B\|_{2,\infty} \leq \|A\|_{2,\infty} \|B\|_{2,\infty}$.

      \textbf{Proof:}
      \begin{align*}
      \|A \tkr B\|_{2,\infty}^2 &= \max_i \|A(i, :) \otimes B(i, :)\|_2^2 \\
      &= \max_i \|A(i, :)\|^2 \cdot \|B(i, :)\|^2 \\
      &\leq \max_i \|A(i, :)\|^2 \cdot \max_v \|B(v, :)\|^2.
      \end{align*}
  

    \item[\textbf{P4}.\  ] $(A \tkr B)^\top = A^\top * B^\top$, where $A^\top * B^\top$ is the Khatri–Rao (columnwise) product.


    \item[\textbf{P5}.\  ] $(A \tkr B)(A^\top * B^\top) = (AA^\top) \circ (BB^\top)$
    where $(X_1 \circ X_2)$ is the Hadamard product of $X_1$ and $X_2$.
 

    \item[\textbf{P6}.\  ] $\|A \tkr B\| \leq \|A \otimes B\| = \|A\| \|B\|$.
 

    \item[\textbf{P7}.\  ] \label{itm:P7} If $Z^{(v)} \in \{0,1\}^{n \times K_v}$, $v = 1, 2$, are clustering matrices, then  
      $Z^{(1)} \tkr Z^{(2)} \in \{0,1\}^{n \times K_1 K_2}$ and       
      $\left( Z^{(1)} \tkr Z^{(2)} \right)^\top \bigl( Z^{(1)} \tkr Z^{(2)} \bigr) = \mathcal{D}_Z^{(1,2)}$
      is a diagonal matrix  with elements      
      $$
       \calD_{Z,(k_1,k_2)}^{(1,2)} = \sum_{i=1}^n I \left( i \in \calG_{k_1}^{(1)} \right) \cdot I \left( i \in \calG_{k_2}^{(2)} \right) = 
         n_{(k_1,k_2)}^{(1,2)},
      $$      
      where $\calG_{k_v}^{(v)}$, $v=1,2$, is the group $k_v$ according to clustering $v \in \{1,2\}$, $k_v \in [K_v]$.
    
      \textbf{Proof:}
      Consider $\calL = \left( Z^{(1)} \tkr Z^{(2)} \right)^\top \left( Z^{(1)} \tkr Z^{(2)} \right) \in \RR^{K_1 K_2 \times K_1 K_2}$.
      Then, using double-indexing, one can write
      \begin{align*}
      \calL ((k_1, k_2), (\tilde{k}_1, \tilde{k}_2))  
      & = \sum_{i=1}^n \left( Z^{(1)} \tkr Z^{(2)} \right) (i,(k_1,k_2)) \, \left( Z^{(1)} \tkr Z^{(2)} \right) (i,(\tilde{k}_1,\tilde{k}_2)) \\
      & = \sum_{i=1}^n I\bigl( i \in \calG_{k_1}^{(1)} \bigr)\, I\bigl( i \in\calG_{k_2}^{(2)} \bigr)\, I\bigl( i \in\calG_{\tilde{k}_1}^{(1)} \bigr)
      \, I\bigl( i \in\calG_{\tilde{k}_2}^{(2)} \bigr) \\
      & = \sum_{i=1}^n I\bigl( i \in\calG_{k_1}^{(1)} \bigr)\, I\bigl( l \in\calG_{k_2}^{(2)} \bigr) \, I(k_1 = \tilde{k}_1,\, k_2 = \tilde{k}_2) \\
      & =\ \  n_{(k_1,k_2)}^{(1,2)} 
      \end{align*}

\end{itemize}

\section{Visual Comparison of $\hatU^{(v)}$ and $\hatZ^{(v)}$ Inputs}
\label{KRAFTY-visualization}

\begin{figure}[H]
    \centering
    \begin{minipage}{\linewidth} 
    
        \begin{subfigure}[b]{0.25\linewidth}
            \centering
            \includegraphics[width=\linewidth]{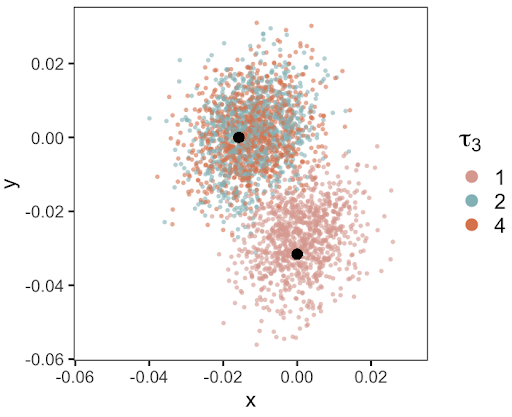} 
            \caption{$\hatU^{(1)}$}
            \label{fig:panel1}
        \end{subfigure}
        \hfill 
        \begin{subfigure}[b]{0.25\linewidth}
            \centering
            \includegraphics[width=\linewidth]{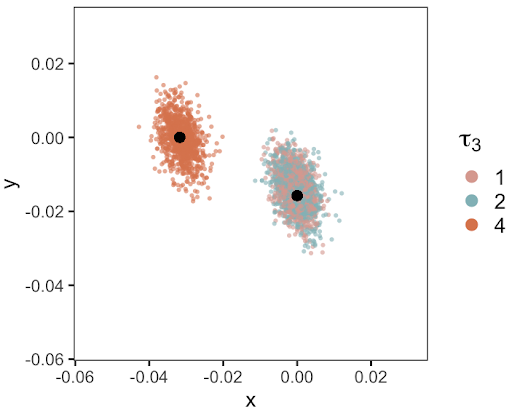}
            \caption{$\hatU^{(2)}$}
            \label{fig:panel2}
        \end{subfigure}
        \hfill
        \begin{subfigure}[b]{0.48\linewidth}
            \centering
            \includegraphics[width=\linewidth]{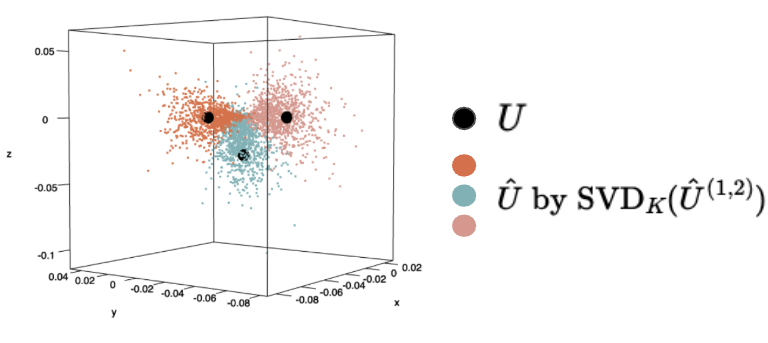}
            \caption{$\hatU$ obtained by $SVD
            _K (\hatU^{(1,2)}) $}
            \label{fig:panel3}
        \end{subfigure}
        
        \vspace{0.25cm} 
        
        \begin{subfigure}[b]{0.25\linewidth}
            \centering
            \includegraphics[width=\linewidth]{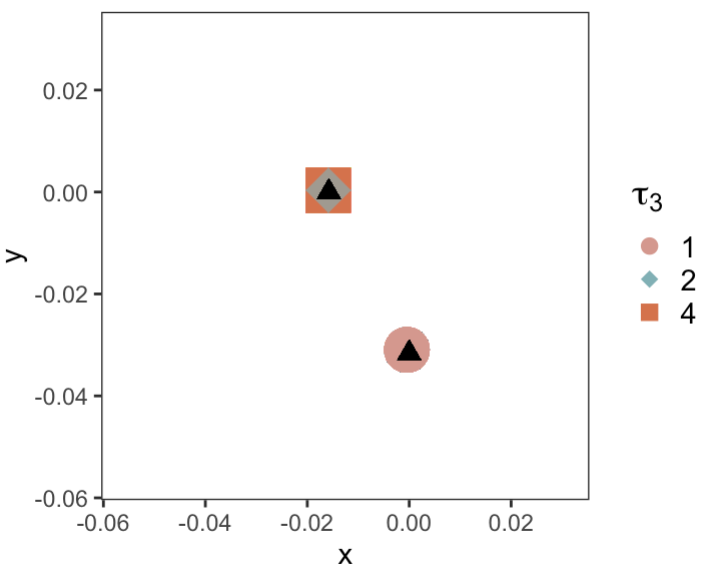}
            \caption{$\hatZ^{(1)}$}
            \label{fig:panel4}
        \end{subfigure}
        \hfill
        \begin{subfigure}[b]{0.25\linewidth}
            \centering
            \includegraphics[width=\linewidth]{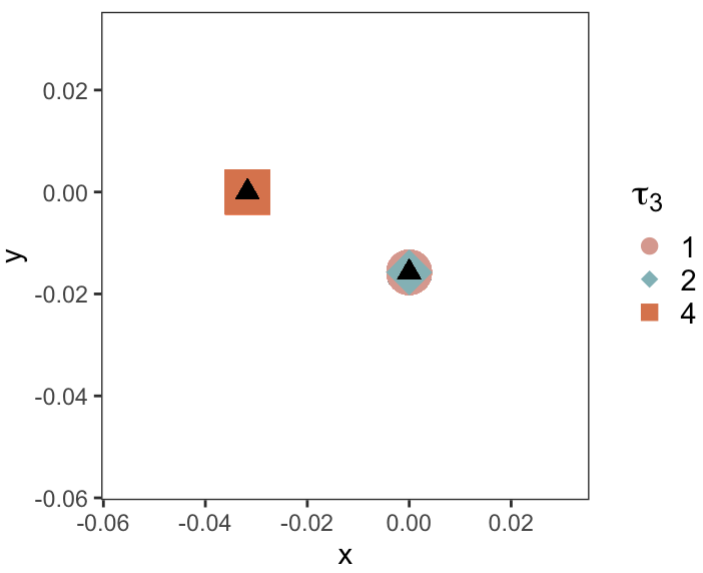}
            \caption{$\hatZ^{(2)}$}
            \label{fig:panel5}
        \end{subfigure}
        \hfill
        \begin{subfigure}[b]{0.48\linewidth}
            \centering
            \includegraphics[width=\linewidth]{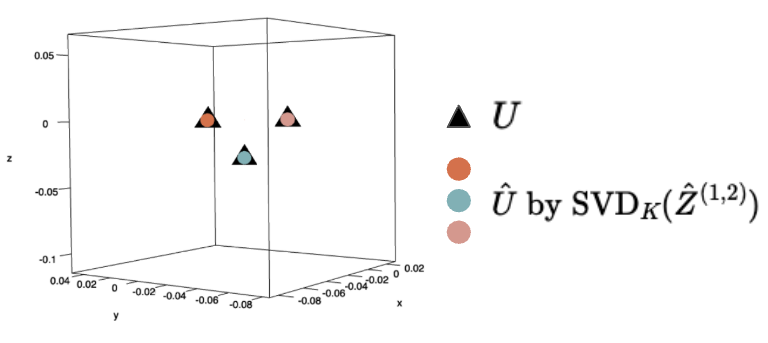}
            \caption{$\hatU$ obtained by $SVD
            _K (\hatZ^{(1,2)}) $}
            \label{fig:panel6}
        \end{subfigure}
        
        \caption{Illustration of KRAFTY using simulated Gaussian mixtures with dimension $p=3$ and $K=3$ joint clusters. (a) $\hatU^{(1)}$ from View 1; (b) $\hatU^{(2)}$ from View 2; (c) left singular vectors from the SVD of $\hatU^{(1)} \tkr \hatU^{(2)}$. Bottom row: (d) $\hatZ^{(1)}$ from View 1, scaled by the square root of singular values, with clustering obtained by $k$-means on $\hatU^{(1)}$; (e) analogous results for $\hatZ^{(2)}$ from View 2 using $\hatU^{(2)}$; (f) left singular vectors from the SVD of $\hatZ^{(1)} \tkr \hatZ^{(2)}$. Points are colored by ground truth membership; black markers indicate the true cluster centers. We can observe that while three joint clusters exist, only two are distinguishable in each single view due to overlap. The figure conveys the idea that both $\hatU^{(v)}$ and $\hatZ^{(v)}$ target the same true centers; however, $\hatZ^{(v)}$ collapses points toward centroids, while $\hatU^{(v)}$ preserves the variance 'cloud.' The resulting $\hatU$ from either $ \hatU^{(v)}$ inputs or $ \hatZ^{(v)}$ inputs estimates the same $U$ for the joint view. The joint space separates all clusters into orthogonal dimensions.}
        \label{fig:KR-algo2}
        
    \end{minipage}
\end{figure}


\section{Additional Simulation Results}
This section presents extended simulation results across three settings: varying the number of joint clusters, signal-to-noise ratios, and single-view cluster numbers. 

\subsection{Varying Number of Joint Clusters}
Figure~\ref{fig:K_ZU_sigma} compares MASE and KR for three joint cluster sizes $K \in {4, 9, 15}$ across varying noise levels $\sigma^2$ (x-axis). The number of clusters in each view are both 4 ($K_1=K_2=4$), the dimension of the observations is $p=20$, and this value is unchanged in the whole subsection. The top row uses $\hatZ^{(v)}$ as input; the bottom row uses $\hatU^{(v)}$. First, focusing on the results with $\hatZ^{(v)}$ as input, when $K=4 < K_1+K_2$, MASE and KR perform similarly whether $K$ is known or unknown. When $K > K_1 + K_2$ (i.e., $K = 9, 15$), MASE with known $K$ and KR with known or unknown $K$ achieve similar performance, all outperforming MASE with unknown $K$. This supports the finding that KR’s advantage comes from correctly estimating $K$. 

For the $U$ input (bottom row), when $K < K_1 + K_2$, MASE outperforms KR when $K$ is unknown, while the two methods perform similarly when $K$ is known. When $K > K_1 + K_2$, MASE and KR perform similarly when $K$ is known, but if $K$ is unknown, there is a large performance gap between KR and MASE.
\begin{figure}
    \centering
        \begin{minipage}{\linewidth}
            \centering
            \includegraphics[width=\linewidth]{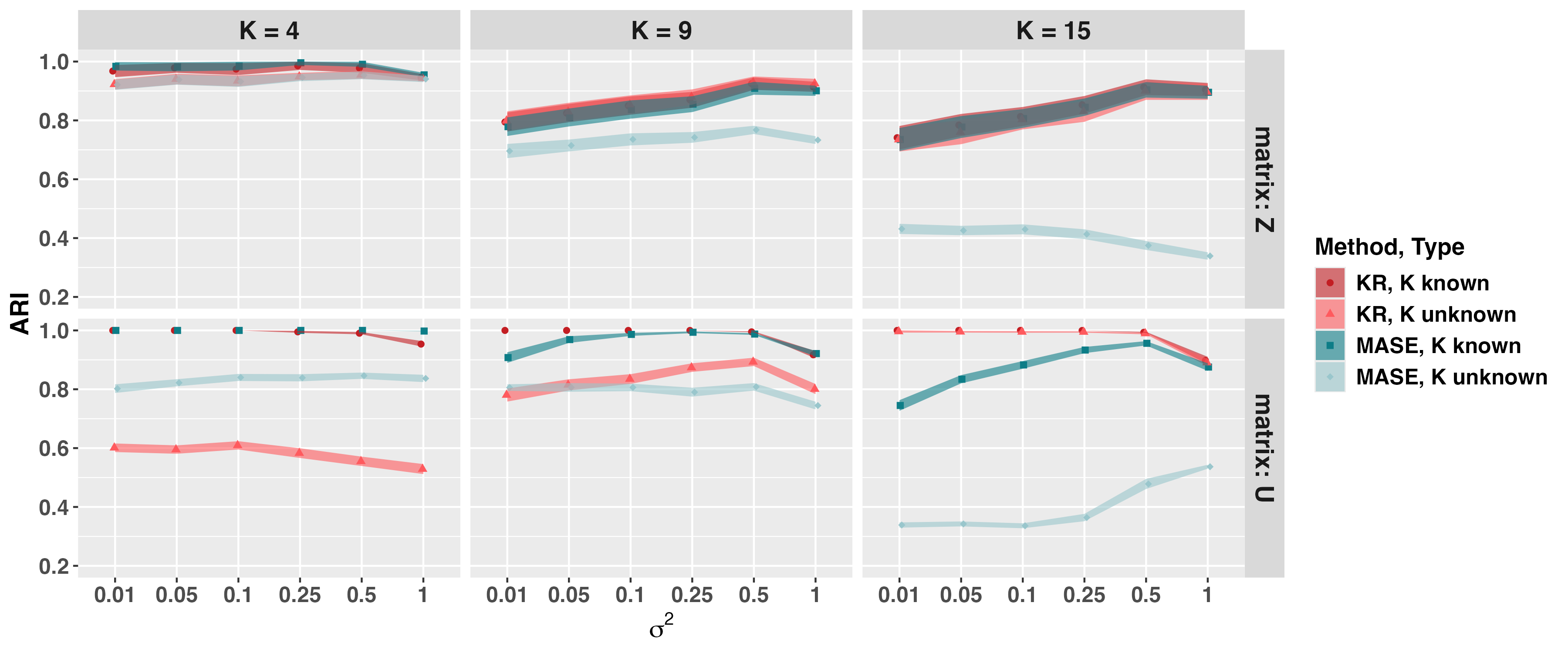}
            \caption{Joint clustering performance of KR and MASE using $\hatZ^{(v)}$ or $\hatU^{(v)}$ matrices, evaluated under known and unknown $K$. The true number of joint clusters $K$ is fixed at 4, 9, and 15, while the noise level $\sigma^2$ varies. ARI compares estimated joint clusters with ground truth. $K_1 = K_2 = 4, p=20.$}
            \label{fig:K_ZU_sigma}
        \end{minipage}   
\end{figure}

Figures~\ref{fig:K_compareZU_K} and~\ref{fig:K_compareZU_estK} investigate how different inputs ($\hatZ^{(v)}$ versus $U^{(v)}$) affect the performance of MASE and KR when the number of joint clusters $K$ is unknown. In Figure~\ref{fig:K_compareZU_K}, joint clustering performance (measured by ARI) is evaluated with noise levels fixed at either 0.01 or 1, while $K$ varies from 4 to 16. We observe that when $K > K_1+K_2$, MASE generally exhibits poor performance. However, under low noise conditions, $\hatZ^{(v)}$ outperforms $\hatU^{(v)}$ as $K$ increases, whereas under high noise conditions, the opposite occurs ($\hatU^{(v)}$ outperforms $\hatZ^{(v)}$). KR's performance with different inputs also varies on different combinations of noise and number of joint clusters. $\hatU^{(v)}$ outperforms $\hatZ^{(v)}$ under low noise conditions, whereas $\hatZ^{(v)}$ outperforms $\hatU^{(v)}$ under high noise conditions when $K > K_1+K_2$. Figure~\ref{fig:K_compareZU_estK} presents a heatmap illustrating performance in estimating $K$ when varying $K$ and $\sigma^2$ simultaneously. The color represents the difference in Mean Absolute Error (MAE) for estimating $K$ between the two inputs (calculated as $\text{MAE}_Z - \text{MAE}_U$). Consequently, cool colors indicate that $\hatZ^{(v)}$ yields better performance (lower error), while warm colors indicate that $\hatU^{(v)}$ is better. The results show that for MASE, $\hatU^{(v)}$ and $\hatZ^{(v)}$ has negligible difference. For KR, $\hatZ^{(v)}$ yields better results at low levels of $K$ and $\sigma^2$; however, this trend reverses as both parameters increase, with $\hatU^{(v)}$ eventually outperforming $\hatZ^{(v)}$.

\begin{figure}
    \centering
        \begin{minipage}{\linewidth}
            \centering
            \includegraphics[width=\linewidth]{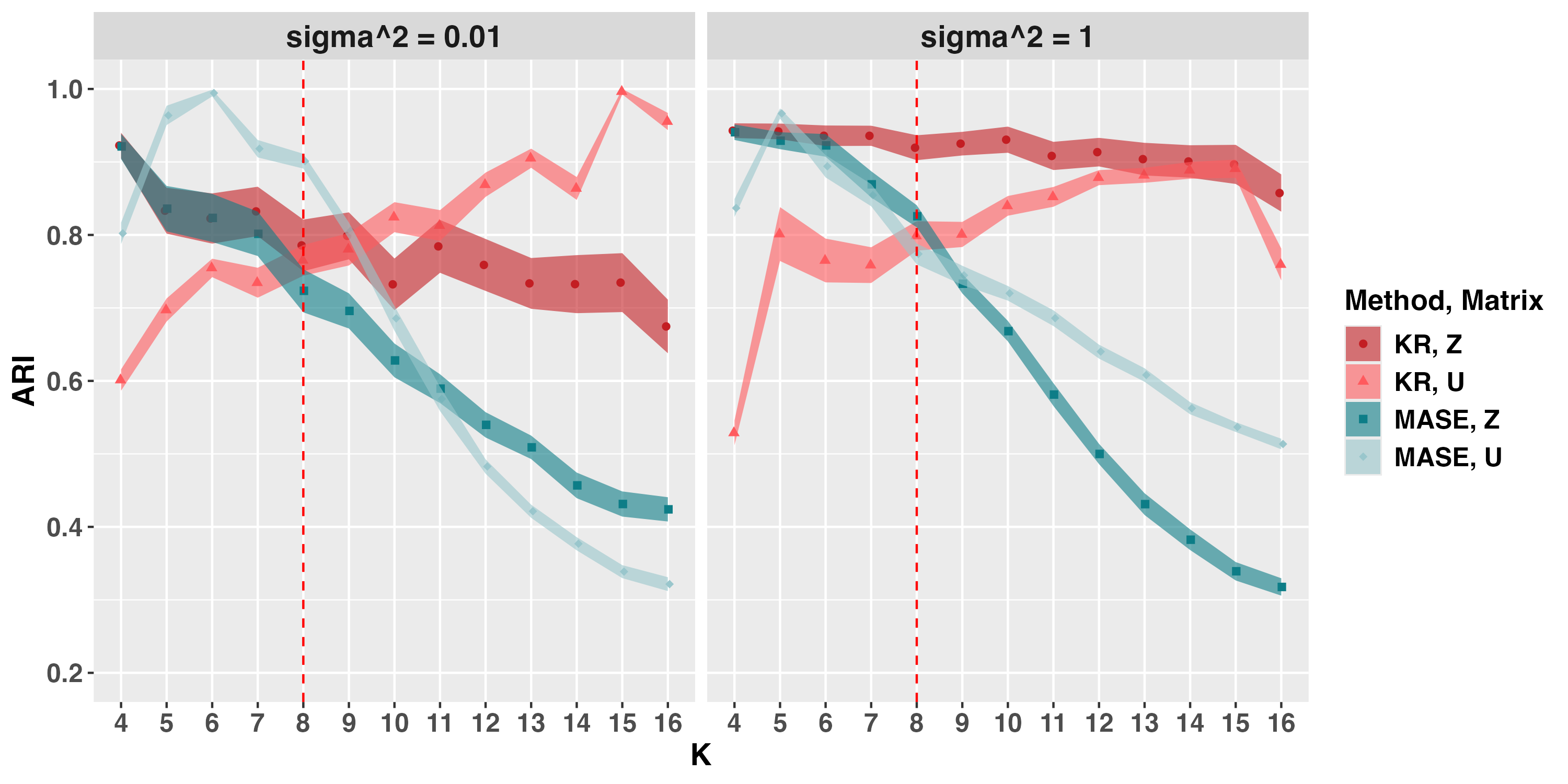}
            \caption{Joint clustering performance across four combinations: KRAFTY with $\hatZ^{(v)}$, KRAFTY with $\hatU^{(v)}$, MASE with $\hatZ^{(v)}$, and MASE with $\hatU^{(v)}$, evaluated under the unknown $K$ setting. The noise level $\sigma^2$ is set to either 0.01 or 1. ARI compares estimated joint clusters with ground truth. $K_1 = K_2 = 4, p=20.$}
            \label{fig:K_compareZU_K}
        \end{minipage}   
\end{figure}


\begin{figure}
    \centering
        \begin{minipage}{\linewidth}
            \centering
            \includegraphics[width=\linewidth]{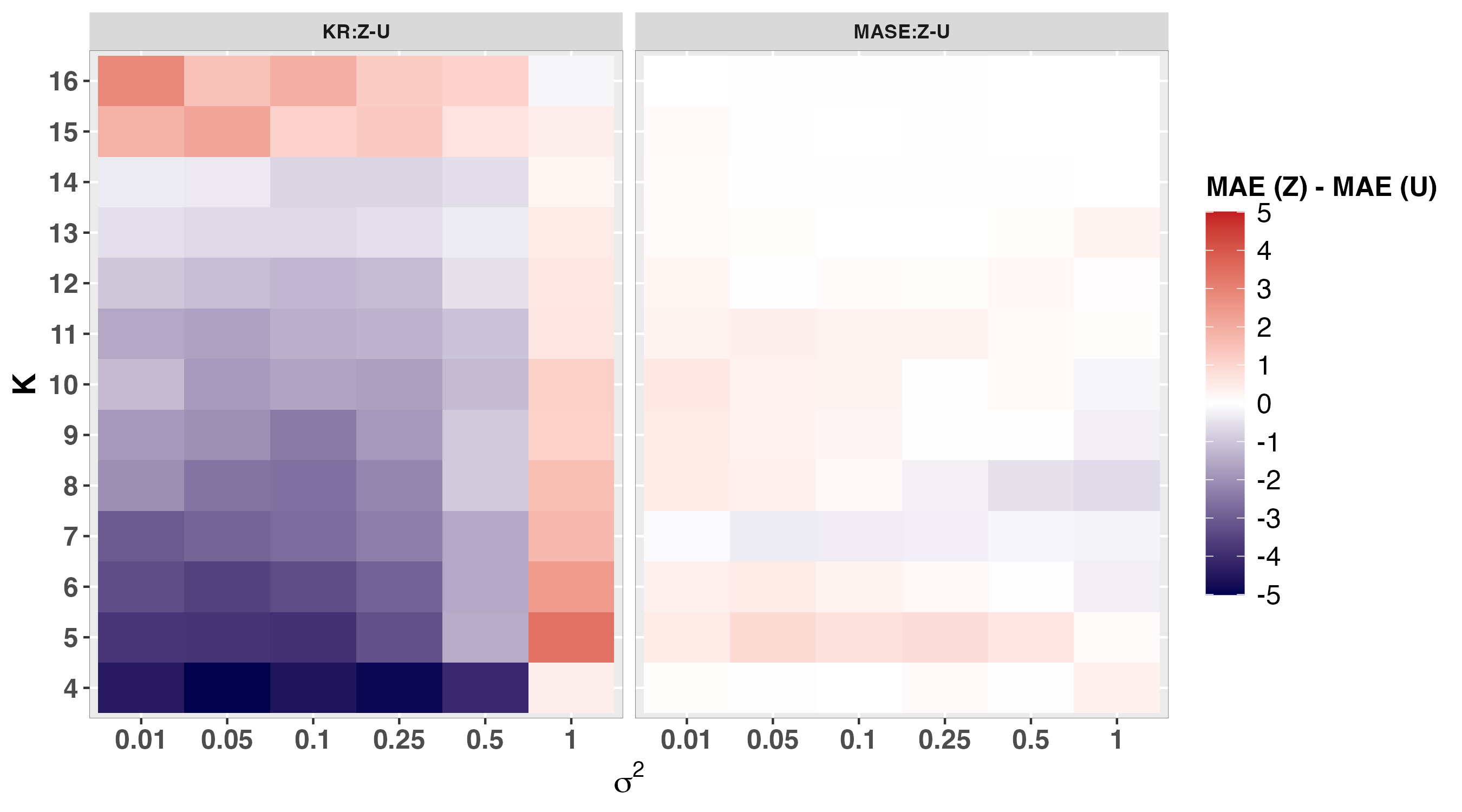}
            \caption{Comparison of joint cluster number estimation performance between Z and U across KR and MASE methods. Heatmap colors represent the difference in mean absolute error (averaged over 100 repetitions) when estimating $K$ using Z versus U. \textbf{Cool colors $\to$ Z better; warm colors $\to$ U better.} $K_1 = K_2 = 4, p=20.$}
            \label{fig:K_compareZU_estK}
        \end{minipage}  
\end{figure}


\subsection{Varying Signal-to-Noise Ratios}
Figure~\ref{fig:m_diffARI_method} presents heatmaps illustrating how the performance of KR and MASE changes as dimensions and noise levels vary. The color scale represents the difference in ARI ($\text{ARI}_{\text{KR}} - \text{ARI}_{\text{MASE}}$); thus, cool colors indicate that MASE performs better, while warm colors indicate that KR performs better. In the known $K$ setting (first two columns), inputs $\hatZ^{(v)}$ yield similar results for both methods regardless the change of dimensions and noise levels. For inputs $\hatU^{(v)}$, MASE is preferred when $K < K_1+K_2$, while KR performs better when $K > K_1+K_2$ under high signal-to-noise ratios. In the unknown $K$ setting, for $\hatU^{(v)}$, MASE is preferred when $K$ is small. When $K=12$ and signal-to-noise ratios are high, KR is better. Notably, when using $\hatZ^{(v)}$, KR consistently outperforms MASE across all signal-to-noise ratios for both $K$ is small or large. We conduct a similar comparison between the $\hatU$ and $\hatZ$ methods for KR and MASE in Figure~\ref{fig:m_diffARI_ZU}.

\begin{figure}
    \centering
        \begin{minipage}{\linewidth}
            \centering
            \includegraphics[width=\linewidth]{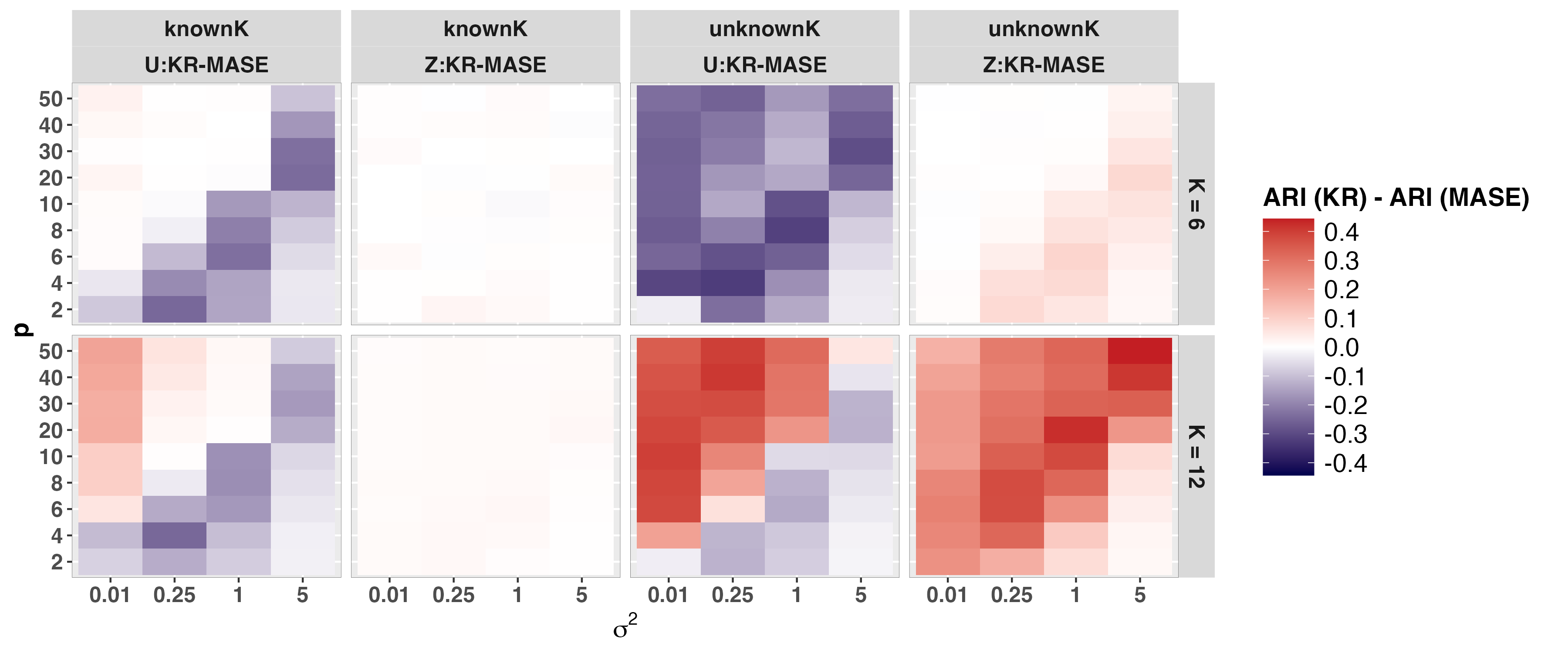}
            \caption{Comparison of joint clustering performance between KR and MASE across Z and U matrices. The left 2 columns of the heatmap show the difference in mean ARI (averaged over 100 repetitions) between KR and MASE when the number of joint clusters $K$ is \textbf{known}, using either the Z or U matrix. The right 2 columns show the difference in mean ARI between KR and MASE when $K$ is \textbf{unknown}. \textbf{Cool colors $\to$ MASE better; warm colors $\to$ KR better.} $K_1 = K_2 = 4$.
}
            \label{fig:m_diffARI_method}
        \end{minipage}
\end{figure}

\begin{figure}
    \centering
        \begin{minipage}{\linewidth}
            \centering
            \includegraphics[width=\linewidth]{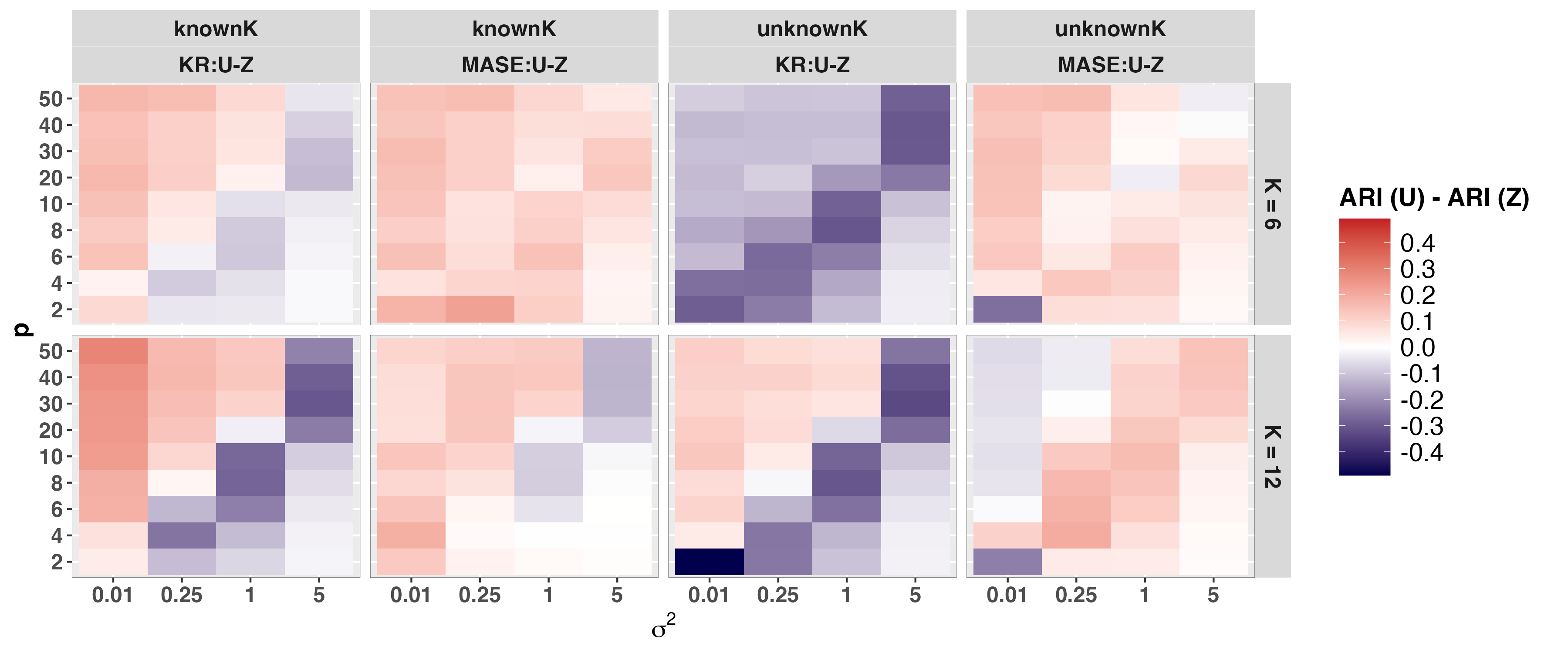}
            \caption{Comparison of joint clustering performance between Z and U across KR and MASE methods. The left 2 columns of the heatmap show the difference in mean ARI (averaged over 100 repetitions) between Z and U when the number of joint clusters $K$ is \textbf{known}, using either KR or MASE method. The right 2 columns show the difference in mean ARI between Z and U when $K$ is \textbf{unknown}. \textbf{Cool colors $\to$ Z better; warm colors $\to$ U better.} $K_1 = K_2 = 4$.
}
            \label{fig:m_diffARI_ZU}
        \end{minipage}
\end{figure}



\subsection{Varying Number of Clusters in Single View}

To obtain more robust conclusions, we conduct additional simulations where $K_1$ and $K_2$ vary from 3 to 10, and the true number of joint clusters $K$ takes four representative values: two within the range of $K \leq K_1 + K_2$, specifically, $\max(K_1, K_2)$ and $K_1 + K_2$, two beyond that range, $\lfloor K_1 K_2 / 2 \rfloor$ and $K_1 K_2$. The noise level is fixed at $\sigma^2 = 0.1$ and the dimension at $p = 101$. The results are shown as heatmaps in Figures~\ref{fig:k1k2_ARI_method} and~\ref{fig:k1k2_ARI_matrix}. 

Figure~\ref{fig:k1k2_ARI_method} further corroborates that KRAFTY’s advantage comes from its ability to correctly estimate $K$. In the left two columns (where $K$ is known), the heatmaps are nearly white or light-colored, indicating that KRAFTY and MASE achieve similar joint clustering performance across various regimes governing the size of $K$. In contrast, the right two columns (where $K$ is unknown) display more distinct warm colors as $K$ increases, reflecting KRAFTY’s improved performance.

\begin{figure}
    \centering
        \begin{minipage}{\linewidth}
            \centering
            \includegraphics[width=\linewidth]{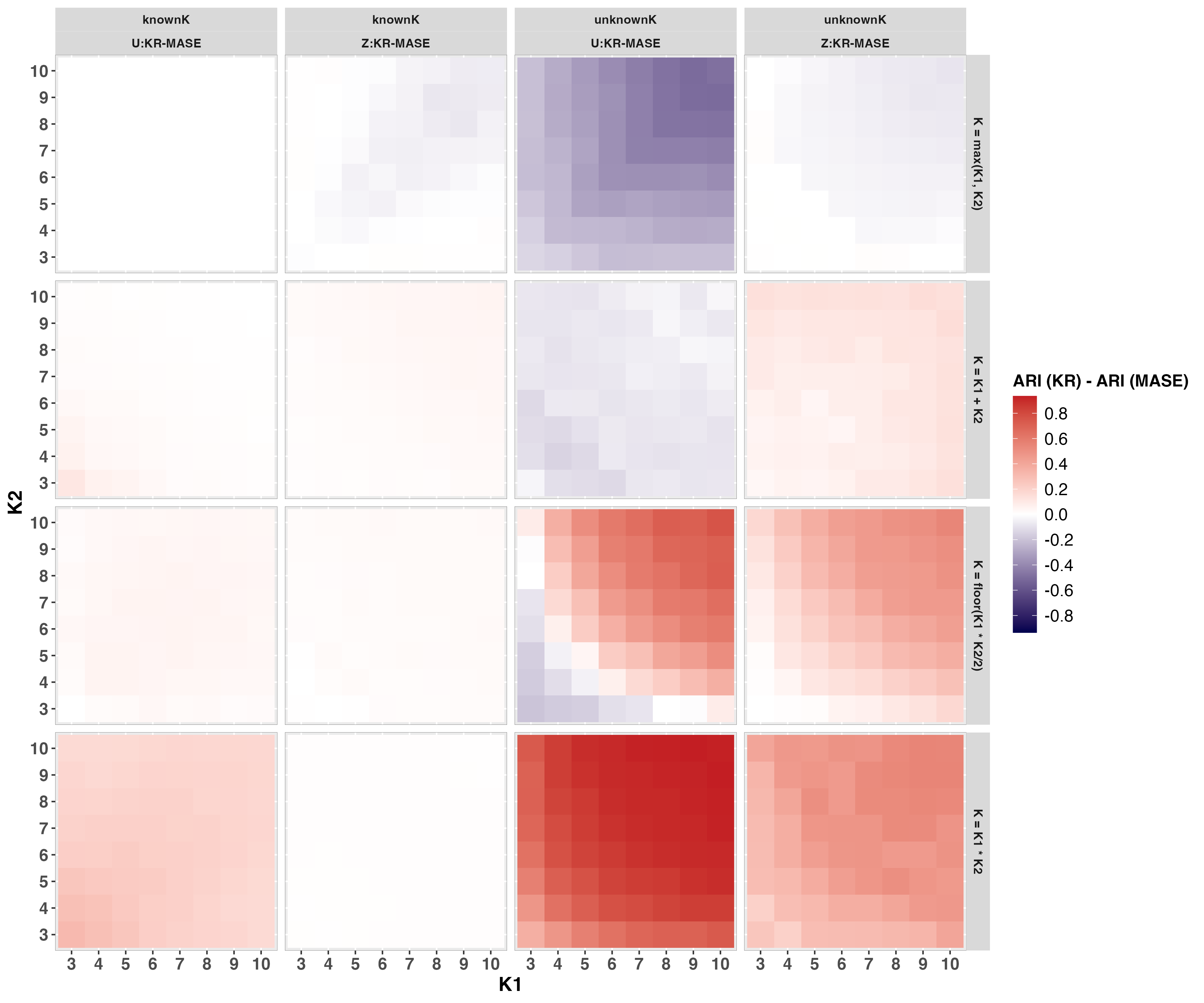}
            \caption{Comparison of joint clustering performance between KR and MASE using Z and U matrices across 4 different true $K$ settings, appearing in each row. The left 2 columns of the heatmap show the difference in mean ARI (averaged over 100 repetitions) between KR and MASE when the number of joint cluster \textbf{$K$ is known}, using either the Z or U matrix. The right 2 columns show the difference in mean ARI between KR and MASE when \textbf{$K$ is unknown}. \textbf{Cool colors $\to$ MASE better; warm colors $\to$ KR better.} $\sigma^2 = 0.1$, $p=101$.
}
            \label{fig:k1k2_ARI_method}
        \end{minipage}
\end{figure}

Between the choices of using $\hatZ^{(v)}$ matrices and $\hatU^{(v)}$ matrices, the second right column in Figure~\ref{fig:k1k2_ARI_matrix} confirms that when $K$ is small, KRAFTY with $\hatZ^{(v)}$ matrices yields better joint clustering performance, while when $K$ is large, KRAFTY with $\hatU^{(v)}$ matrices performs better, as indicated by cool colors in the upper panels and warm colors in the lower panels.

\begin{figure}
    \centering
        \begin{minipage}{\linewidth}
            \centering
            \includegraphics[width=\linewidth]{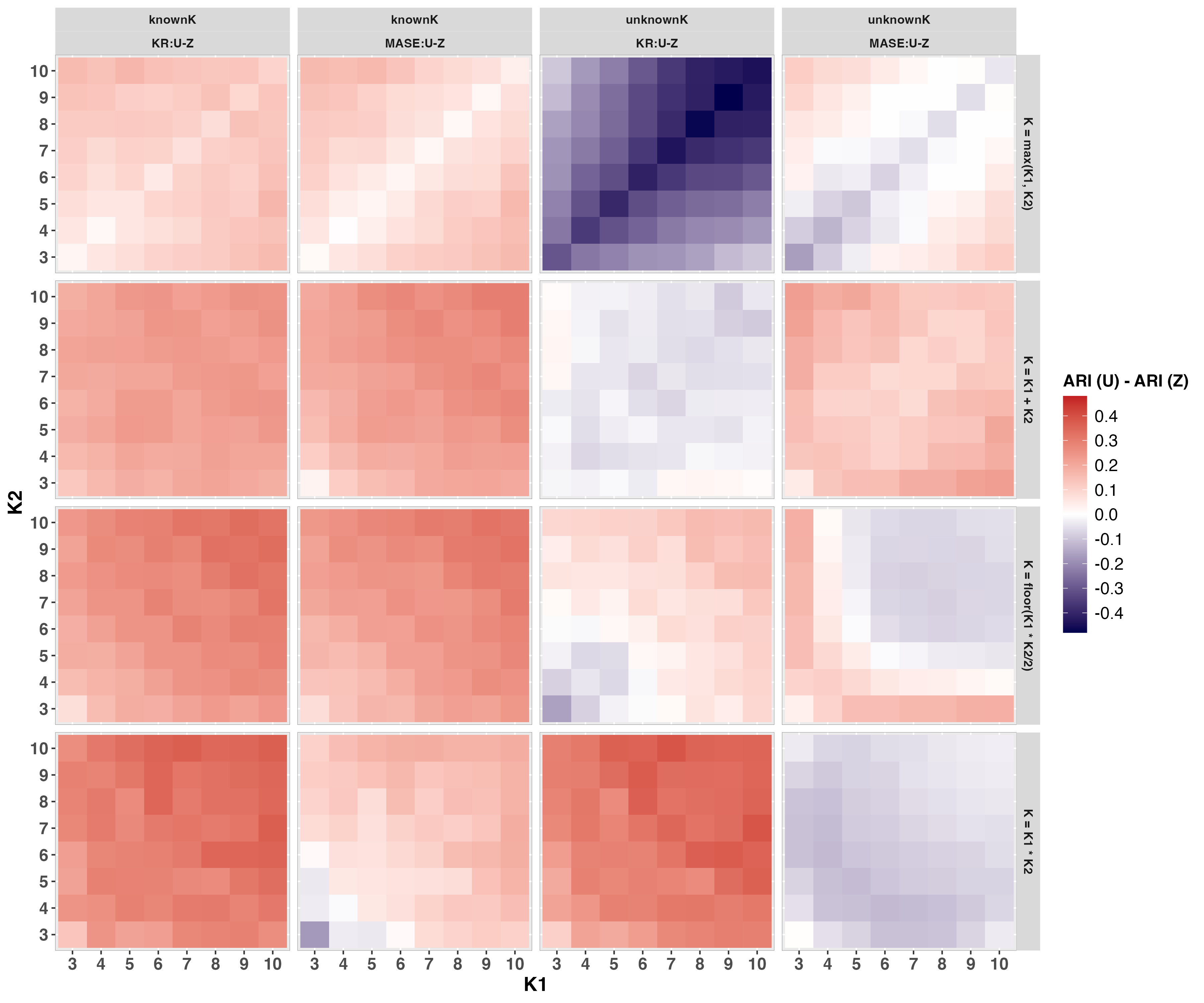}
            \caption{Comparison of joint clustering performance between Z and U using KR or MASE across 4 different true $K$ settings, appearing in each row. The left 2 columns of the heatmap show the difference in mean ARI (averaged over 100 repetitions) between U and Z when the number of joint cluster \textbf{$K$ is known}, using either KR or MASE. The right 2 columns show the difference in mean ARI between U and Z when \textbf{$K$ is unknown}. \textbf{Cool colors $\to$ Z better; warm colors $\to$ U better.} $\sigma = 0.1$, $p=101$.
}
            \label{fig:k1k2_ARI_matrix}
        \end{minipage}
\end{figure}

Figure~\ref{fig:k1k2_K_method} illustrates the performance of estimating $K$ under varying numbers of single-view clusters. Similar to the previous two figures, $K_1$ and $K_2$ range from 3 to 10, the dimension is fixed at 101, the noise level at $\sigma^2 = 0.1$, and the true $K$ is defined according to: $\max(K_1, K_2)$, $K_1+K_2$, $\lfloor K_1K_2/2 \rfloor$, and $K_1K_2$. The heatmap color represents the difference in Mean Absolute Error ($\text{MAE}_{\text{MASE}} - \text{MAE}_{\text{KR}}$); hence, cool colors indicate that MASE is better and warm colors mean that KR is better. Under scenarios $K=\max(K_1, K_2)$ and $K=K_1+K_2$, MASE generally outperforms KR since $K \leq K_1+K_2$. However, in some cases when using $\hatZ^{(v)}$ as inputs, KR can achieve comparable performance. When $K > K_1+K_2$, the advantage of KR becomes more pronounced as $K$ increasingly exceeds $K_1+K_2$.

\begin{figure}
    \centering
        \begin{minipage}{\linewidth}
            \centering
            \includegraphics[width=\linewidth]{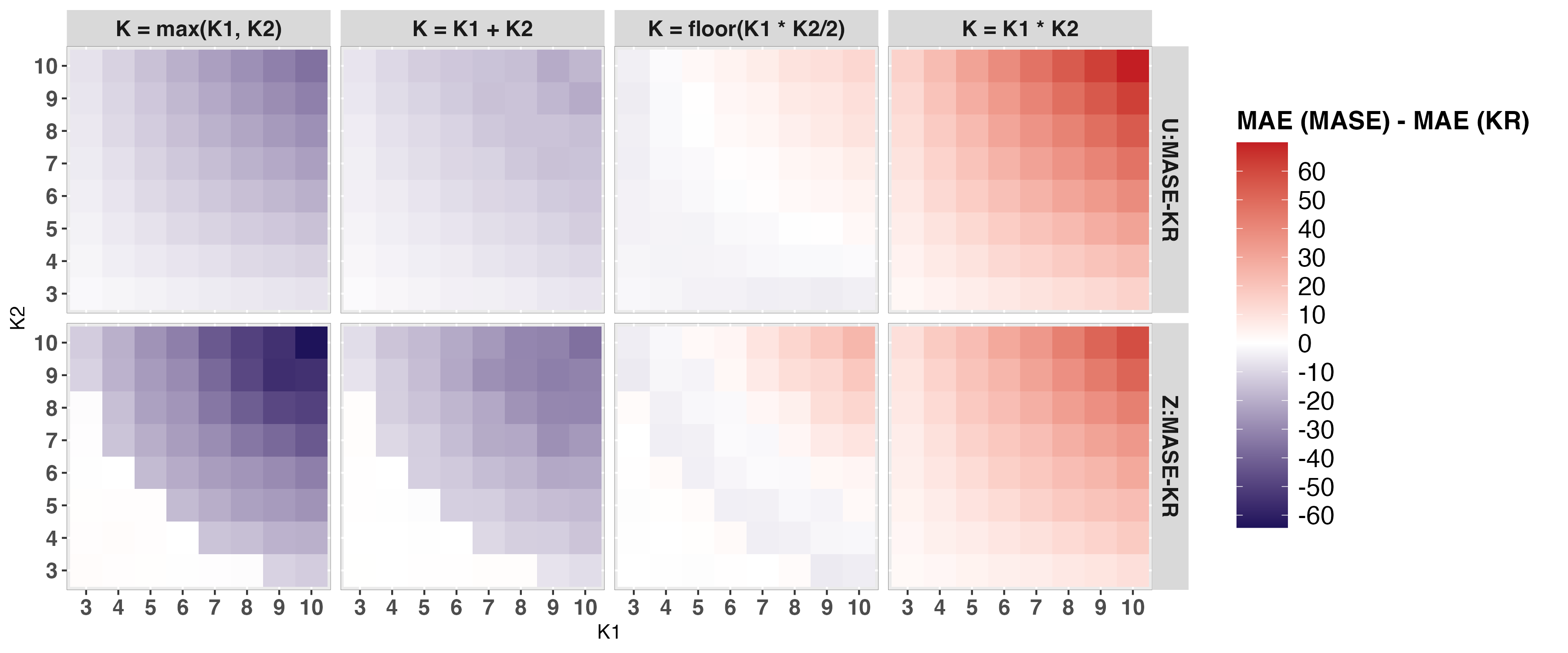}
            \caption{Comparison of joint cluster number estimation performance between KR and MASE using Z or U matrices across 4 true $K$ settings. The top row of the heatmap shows the difference in mean absolute error (averaged over 100 repetitions) between MASE and KR using U matrix. The bottom row shows the difference in mean absolute error between KR and MASE using Z matrix. \textbf{Cool colors $\to$ MASE better; warm colors $\to$ KR better.} $\sigma^2 = 0.1$, $p=101$.
}
            \label{fig:k1k2_K_method}
        \end{minipage}
\end{figure}

Figure~\ref{fig:k1k2_matrix} illustrates the difference in performance of estimating $K$ using $\hat{Z}$ or $\hat{U}$ as inputs ($\text{MAE}_Z - \text{MAE}_U$), across four true $K$ settings. The top row (KR) and bottom row (MASE) show different behaviors. For MASE, the choice of input has negligible impact. For KR, however, patterns vary by setting: for $K \leq K_1+K_2$, $\hatZ^{(v)}$ is preferred for small $K$, but $\hatU^{(v)}$ becomes better as $K$ increases; for $K = K_1 K_2$, $U$ is the better input. Conversely, in the $K = \lfloor K_1 K_2 / 2 \rfloor$ setting, $\hatZ^{(v)}$ yields lower errors specifically when $K < K_1+K_2$.

\begin{figure}
    \centering
        \begin{minipage}{\linewidth}
            \centering
            \includegraphics[width=\linewidth]{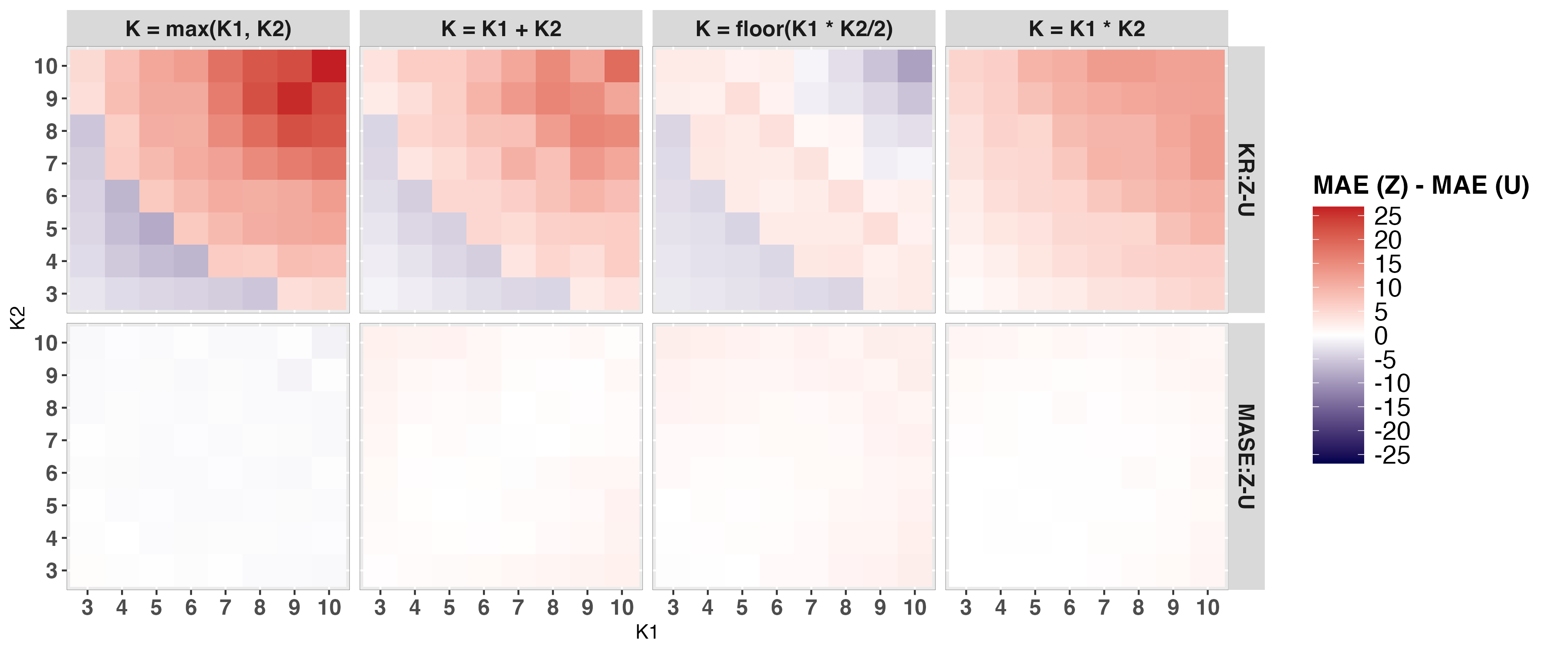}
            \caption{Comparison of joint cluster number estimation performance between Z and U using KR or MASE across 4 different true $K$ settings. The top row of the heatmap shows the difference in mean absolute error (averaged over 100 repetitions) when estimating $K$ using Z versus U using KR. \textbf{Cool colors $\to$ Z better; warm colors $\to$ U better.} The bottom row shows the difference in mean absolute error between Z and U using MASE. $\sigma=0.1, p = 101.$
}
            \label{fig:k1k2_matrix}
        \end{minipage}
\end{figure}

\subsection{Comparing X and U}

This section demonstrates the performance similarity between $\hat{U}^{(v)}$ and its scaled counterpart $\hat{X}^{(v)} = \hat{U}^{(v)} (D^{(v)})^{1/2}$. Consistent with our previous analysis, we examine the same three experimental settings.

\textbf{Varying Number of Joint Clusters}

Figure~\ref{fig:K_compareXU_K} illustrates the performance as $K$ varies at two noise levels. The results indicate that the choice between $\hatX^{(v)}$ and $\hatU^{(v)}$ has a negligible effect on joint clustering for both KR and MASE.
\begin{figure}
    \centering
        \begin{minipage}{\linewidth}
            \centering
            \includegraphics[width=\linewidth]{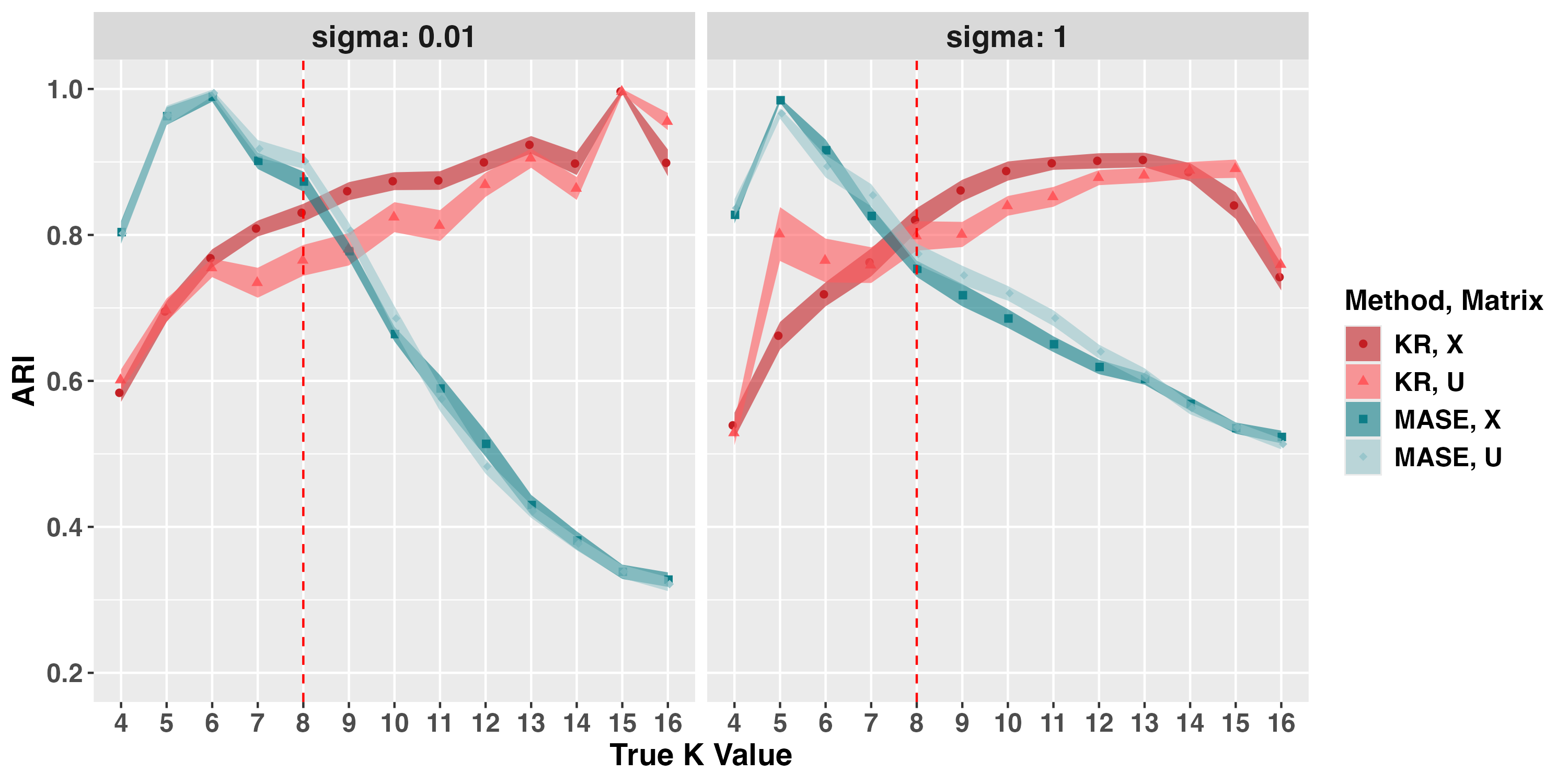}
            \caption{Joint clustering performance of the X and U matrices using KR and MASE under \textbf{unknown $K$} setting. The noise level $\sigma$ is set to either 0.01 or 1. ARI compares estimated joint clusters with ground truth. Here, $K_1 = K_2 = 4, p=20.$}
            \label{fig:K_compareXU_K}
        \end{minipage}  
\end{figure}

Figure~\ref{fig:K_compareXU_sigma} illustrates the performance as $\sigma^2$ varies for three different $K$ scenarios. We observe that the lines representing $\hatX^{(v)}$ and $\hatU^{(v)}$ overlap. It indicates that the choice between these inputs has a negligible effect on the joint clustering performance of both KR and MASE.
\begin{figure}
    \centering
        \begin{minipage}{\linewidth}
            \centering
            \includegraphics[width=\linewidth]{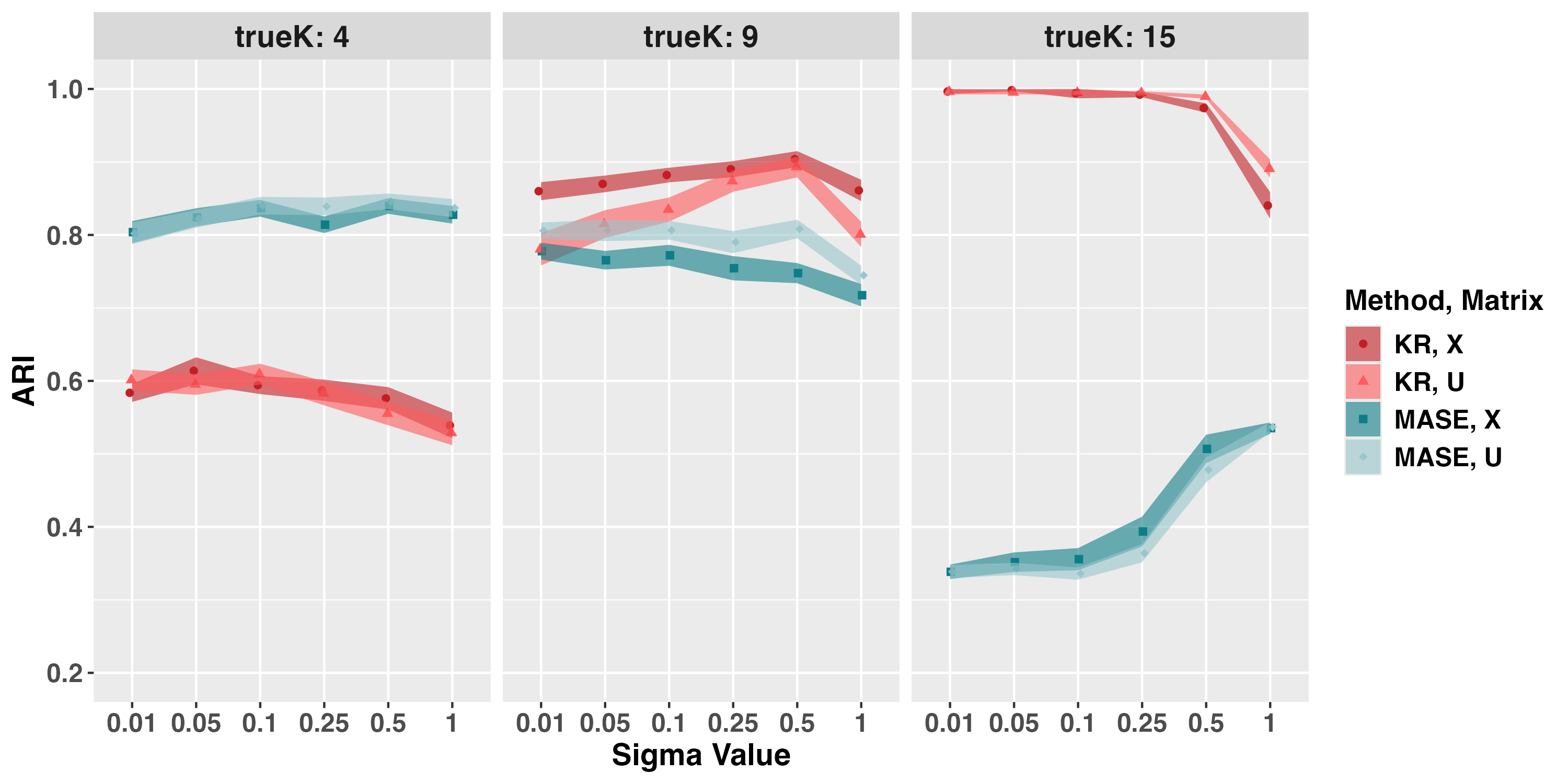}
            \caption{Joint clustering performance of the X and U matrices using KR and MASE under \textbf{unknown $K$} setting. The true number of joint clusters $K$ is  held constant at 4, 9, and 15 across experiments, with $\sigma$ varying. ARI compares estimated joint clusters with ground truth. Here, $K_1 = K_2 = 4, p=20.$}
            \label{fig:K_compareXU_sigma}
        \end{minipage}  
\end{figure}

Figure~\ref{fig:K_compareXU_estK} displays the heatmap comparing the performance of inputs $\hatX^{(v)}$ and $\hatU^{(v)}$ in estimating $K$. The observed difference in error is negligible. 
\begin{figure}
    \centering
        \begin{minipage}{\linewidth}
            \centering
            \includegraphics[width=\linewidth]{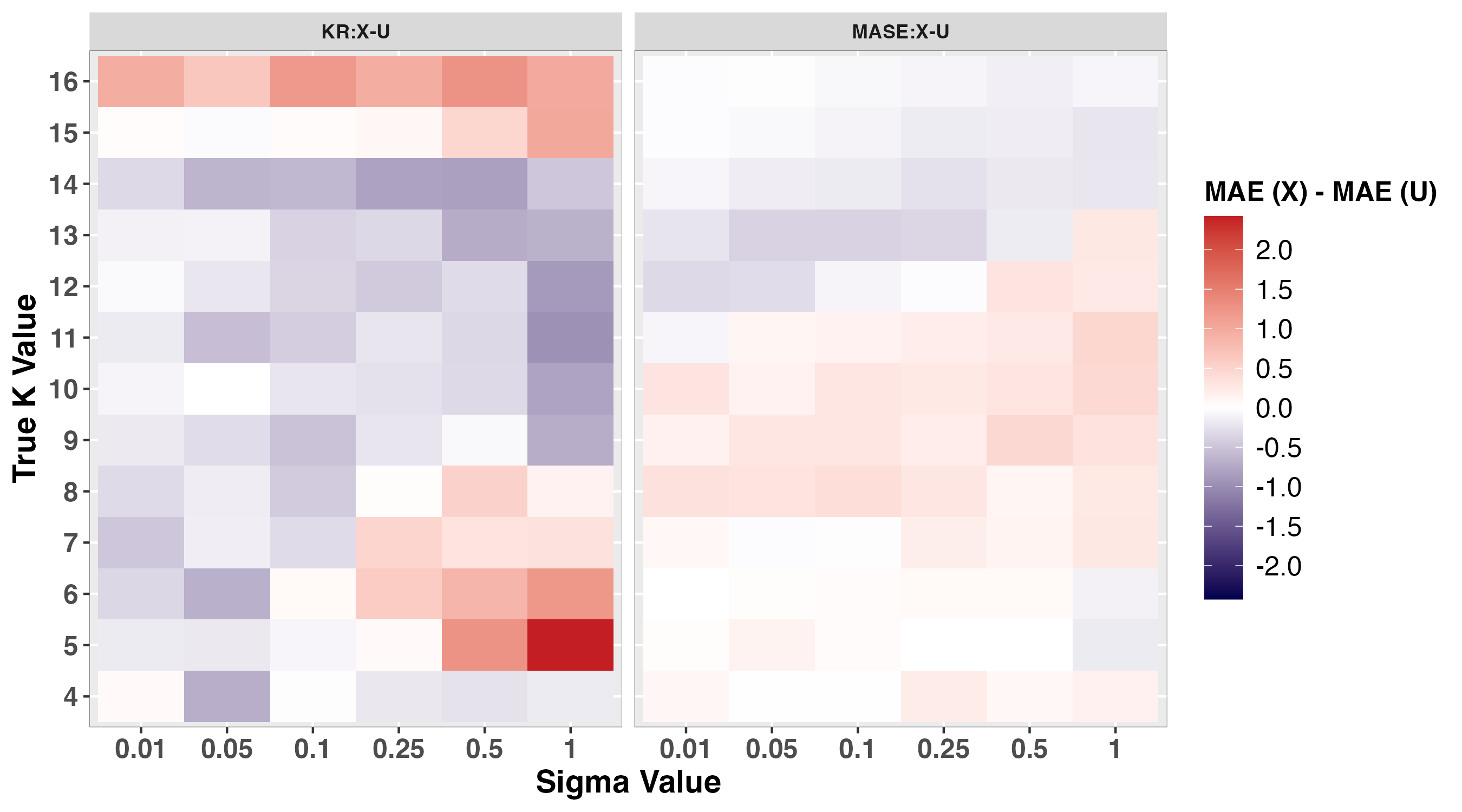}
            \caption{Comparison of joint cluster number estimation performance between X and U across KR and MASE methods. Heatmap colors represent the difference in mean absolute error (averaged over 100 repetitions) when estimating $K$ using X versus U. \textbf{Cool colors $\to$ X better; warm colors $\to$ U better.} $K_1 = K_2 = 4, p=20.$}
            \label{fig:K_compareXU_estK}
        \end{minipage}  
\end{figure}


\textbf{Varying Signal-to-Noise Ratio}
We compare the performance of inputs $\hatU^{(v)}$ and $\hatX^{(v)}$ in joint clustering (Figure~\ref{fig:m_diffARI_XU}) and in estimating $K$ (Figure~\ref{fig:m_K_XU_method}) when varying dimensions and noise levels. Consistent with previous findings, the performance discrepancies between the two inputs remain minimal.


\begin{figure}
    \centering
        \begin{minipage}{\linewidth}
            \centering
            \includegraphics[width=\linewidth]{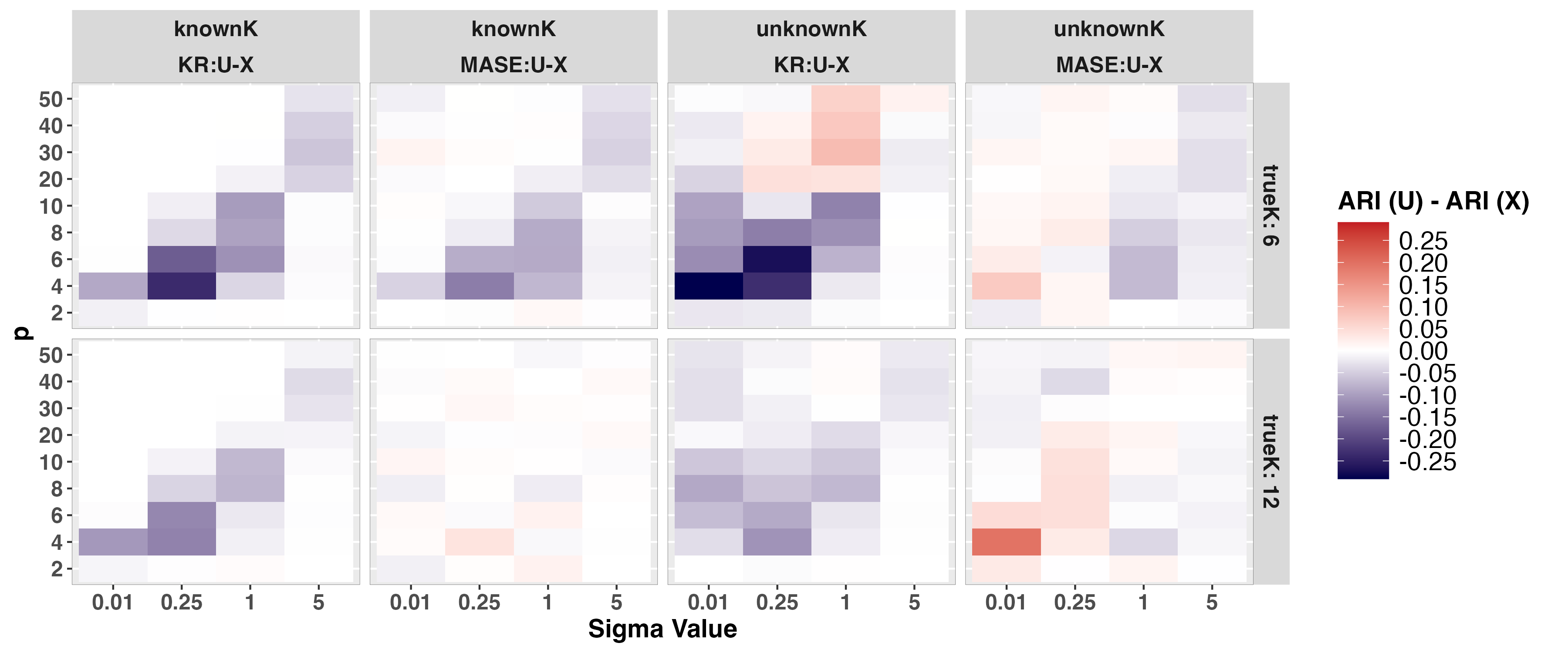}
            \caption{Comparison of joint clustering performance between X and U across KR and MASE methods. The left 2 columns of the heatmap show the difference in mean ARI (averaged over 100 repetitions) between X and U when the number of joint clusters $K$ is \textbf{known}, using either KR or MASE method. The right 2 columns show the difference in mean ARI between X and U when $K$ is \textbf{unknown}. \textbf{Cool colors $\to$ X better; warm colors $\to$ U better.} $K_1 = K_2 = 4$.
}
            \label{fig:m_diffARI_XU}
        \end{minipage}
\end{figure}

\begin{figure}
    \centering
        \begin{minipage}{\linewidth}
            \centering
            \includegraphics[width=\linewidth]{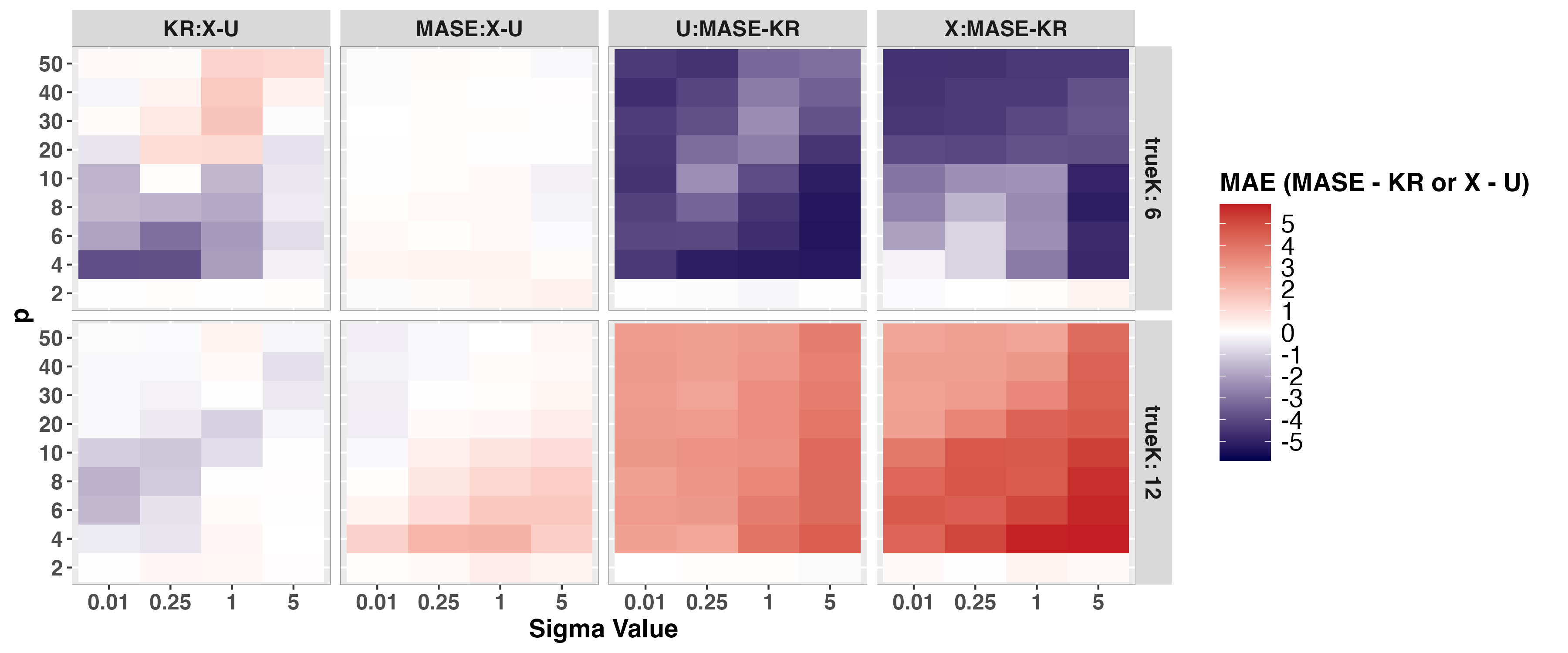}
            \caption{Comparison of joint cluster number estimation performance between KR and MASE across X and U matrices. The left two columns of the heatmap show the difference in mean absolute error (averaged over 100 repetitions) when estimating $K$ using X versus U. \textbf{Cool colors $\to$ X better; warm colors $\to$ U better.} The right 2 columns show the difference in mean absolute error between KR and MASE methods. \textbf{Cool colors $\to$ MASE better; warm colors $\to$ KR better.} Here, $K_1 = K_2 = 4$.}
            \label{fig:m_K_XU_method}
        \end{minipage}
\end{figure}

\textbf{Varying Number of Clusters in Single View}
We compare the performance of inputs $\hatU^{(v)}$ and $\hatX^{(v)}$ in joint clustering (Figure~\ref{fig:k1k2_XU_ARI_matrix}) and in estimating $K$ (Figure~\ref{fig:k1k2_XU_matrix}), across varying single-view cluster numbers and four different $K$ settings. Consistent with previous findings, the performance discrepancies between the two inputs remain negligible.

\begin{figure}
    \centering
        \begin{minipage}{\linewidth}
            \centering
            \includegraphics[width=\linewidth]{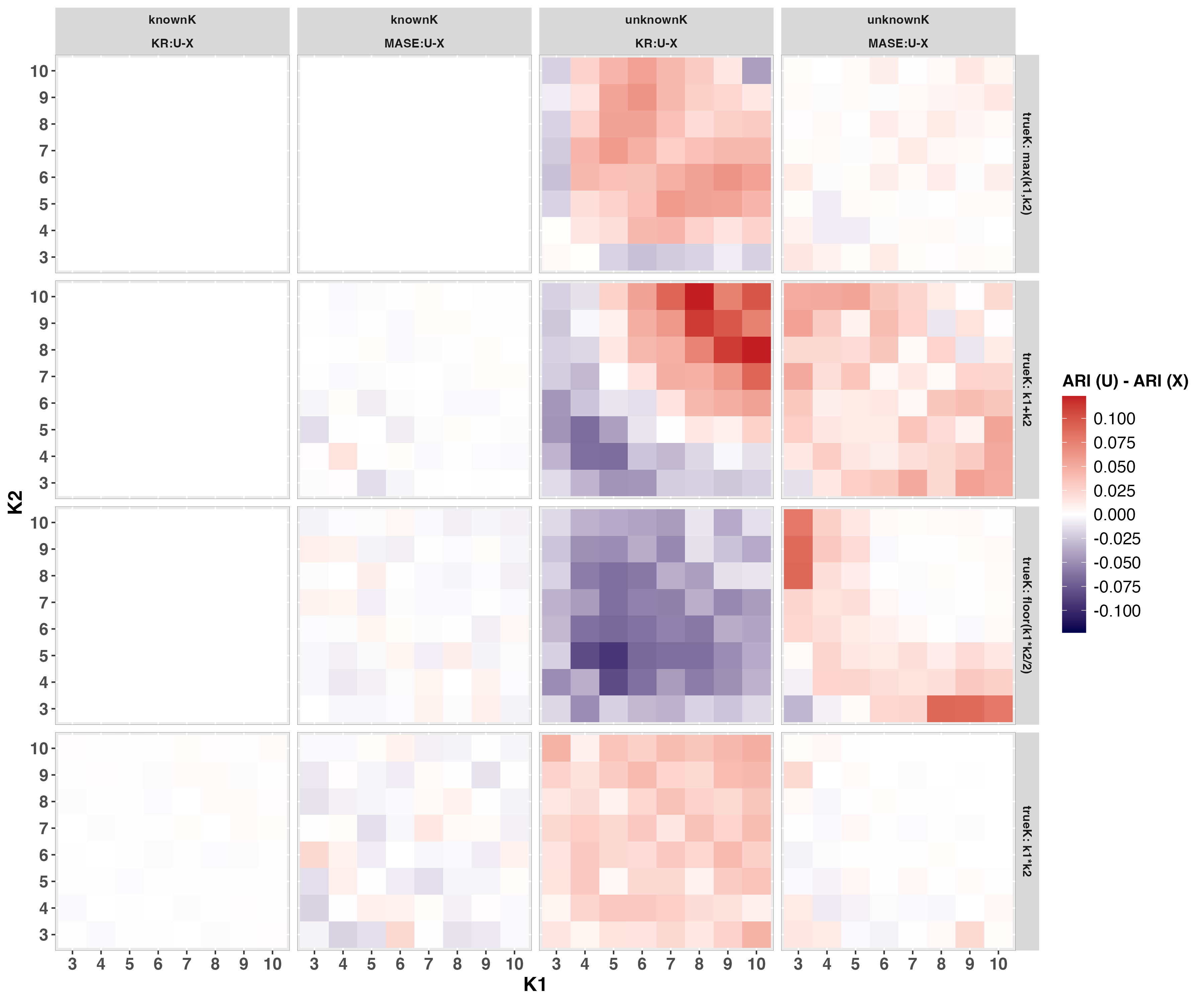}
            \caption{Comparison of joint clustering performance between X and U using KR or MASE across 4 different true $K$ settings. The left 2 columns of the heatmap show the difference in mean ARI (averaged over 100 repetitions) between U and X when the number of joint cluster \textbf{$K$ is known}, using either KR or MASE. The right 2 columns show the difference in mean ARI between U and X when \textbf{$K$ is unknown}. \textbf{Cool colors $\to$ X better; warm colors $\to$ U better.} $\sigma = 0.1$, $p=101$.
}
            \label{fig:k1k2_XU_ARI_matrix}
        \end{minipage}
\end{figure}

\clearpage
\section{Real Data Analysis Supplementary}
In this section, we provide supplementary figures for the real data analysis, along with the two-year analysis.
\subsection{Single-year Results on Global Trade Data Analysis}
Figure~\ref{fig:adjacency_1yr} presents the adjacency matrix for the raw chicken meat trade network. The matrix is symmetrized by summing the export and import volumes for each pair of countries. Rows and columns are ordered according to the clusters estimated by KRAFTY, arranged from Cluster 1 to Cluster 5 (consistent with the sequence in Figure~\ref{fig:map_1yr}). Notably, we observe that trading volume within Cluster 1 and Cluster 4 is small.
\begin{figure}
    \centering
        \begin{minipage}{\linewidth}
            \centering
            \includegraphics[width=\linewidth]{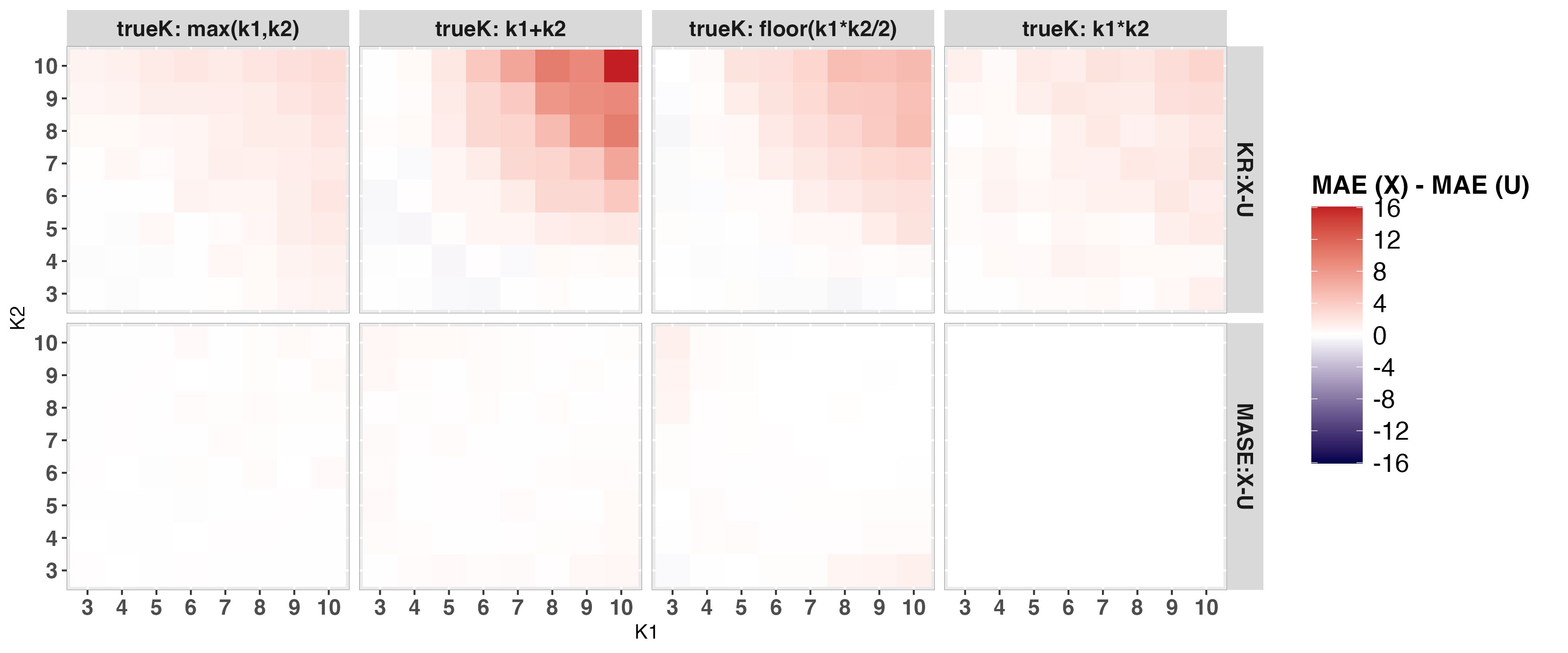}
            \caption{Comparison of joint cluster number estimation performance between X and U using KR or MASE across 4 different true $K$ settings. The top row of the heatmap shows the difference in mean absolute error (averaged over 100 repetitions) when estimating $K$ using X versus U using KR. \textbf{Cool colors $\to$ X better; warm colors $\to$ U better.} The bottom row shows the difference in mean absolute error between X and U using MASE. $\sigma=0.1, p= 101.$
}
            \label{fig:k1k2_XU_matrix}
        \end{minipage}
\end{figure}

\subsection{Two-year Results on Global Trade Data Analysis}

\begin{figure}[ht]
    \centering
    \includegraphics[width=0.6\linewidth]{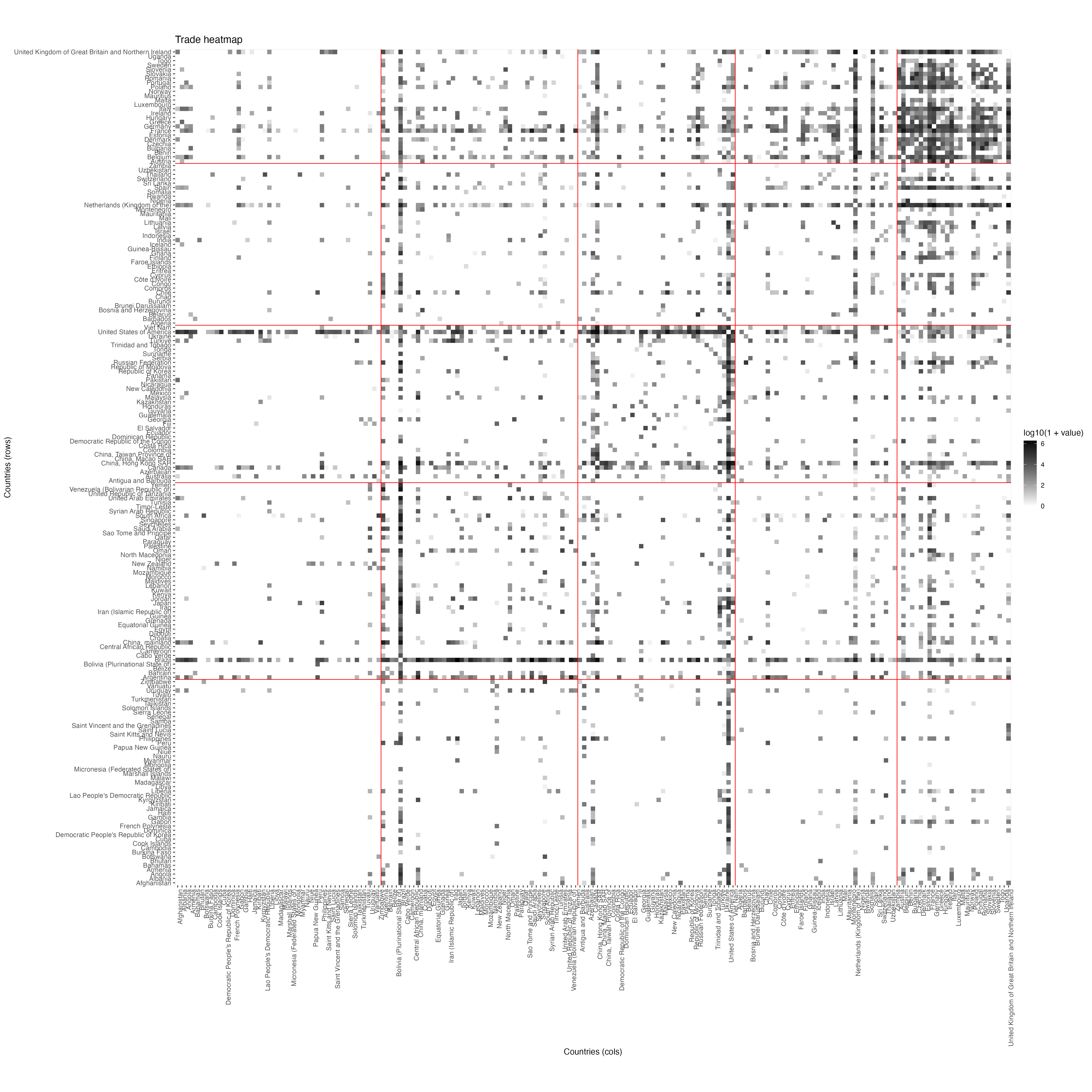}
    \caption{Global trade adjacency heatmap; countries are sorted in KRAFTY joint clusters and boundaries are indicated by red lines.}
    \label{fig:adjacency_1yr}
\end{figure}

We aim to identify time-invariant latent structure, i.e., the ‘unchanging’ clusters, among countries acting as exporters or importers. We apply KRAFTY and MASE to $\hat Z^{(v)}$ with $v\in\{(2010,\mathrm{imp}),(2023,\mathrm{exp})\}$ 
and $v\in\{(2010,\mathrm{exp}),(2023,\mathrm{exp})\}$. The results are shown in Figs.~\ref{fig:scree_exp}–\ref{fig:map_2yr_imp}. The scree plots display a clear elbow under KRAFTY, which are consistent with our theory; moreover, KRAFTY and MASE yield highly consistent joint clusters (Fig.~\ref{fig:alluvial_2yr_exp_kr_mase}, ~\ref{fig:alluvial_2yr_imp_kr_mase}), which aligns with our simulations. Comparing KRAFTY’s joint clusters to combinations of single-view clusters, we can see how the joint clusters are derived from the single-view partitions (Figs.~\ref{fig:alluvial_exp}, ~\ref{fig:alluvial_imp}). One interpretation is that country-level export/import behavior shifts across years, leading to view-specific differences, while a stable joint structure persists.

\begin{figure}
    \centering   
    \begin{minipage}[b]{0.48\linewidth}
        \centering
        \includegraphics[width=\linewidth]{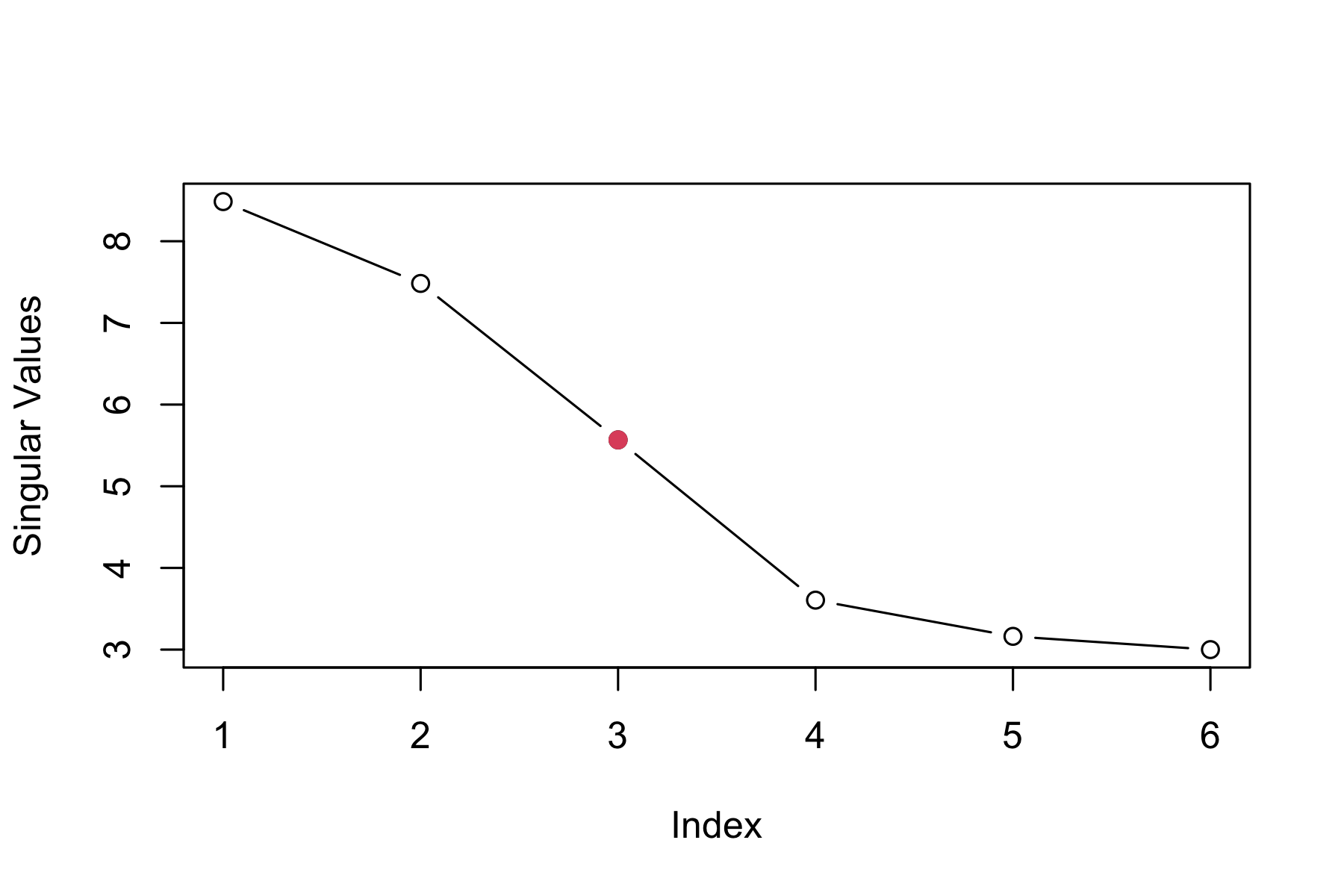}
        \vspace{3pt}
        \centerline{\small (a) KRAFTY}
        \label{fig:scree_kr_exp}
    \end{minipage}
    \hfill 
    \begin{minipage}[b]{0.48\linewidth}
        \centering
        \includegraphics[width=\linewidth]{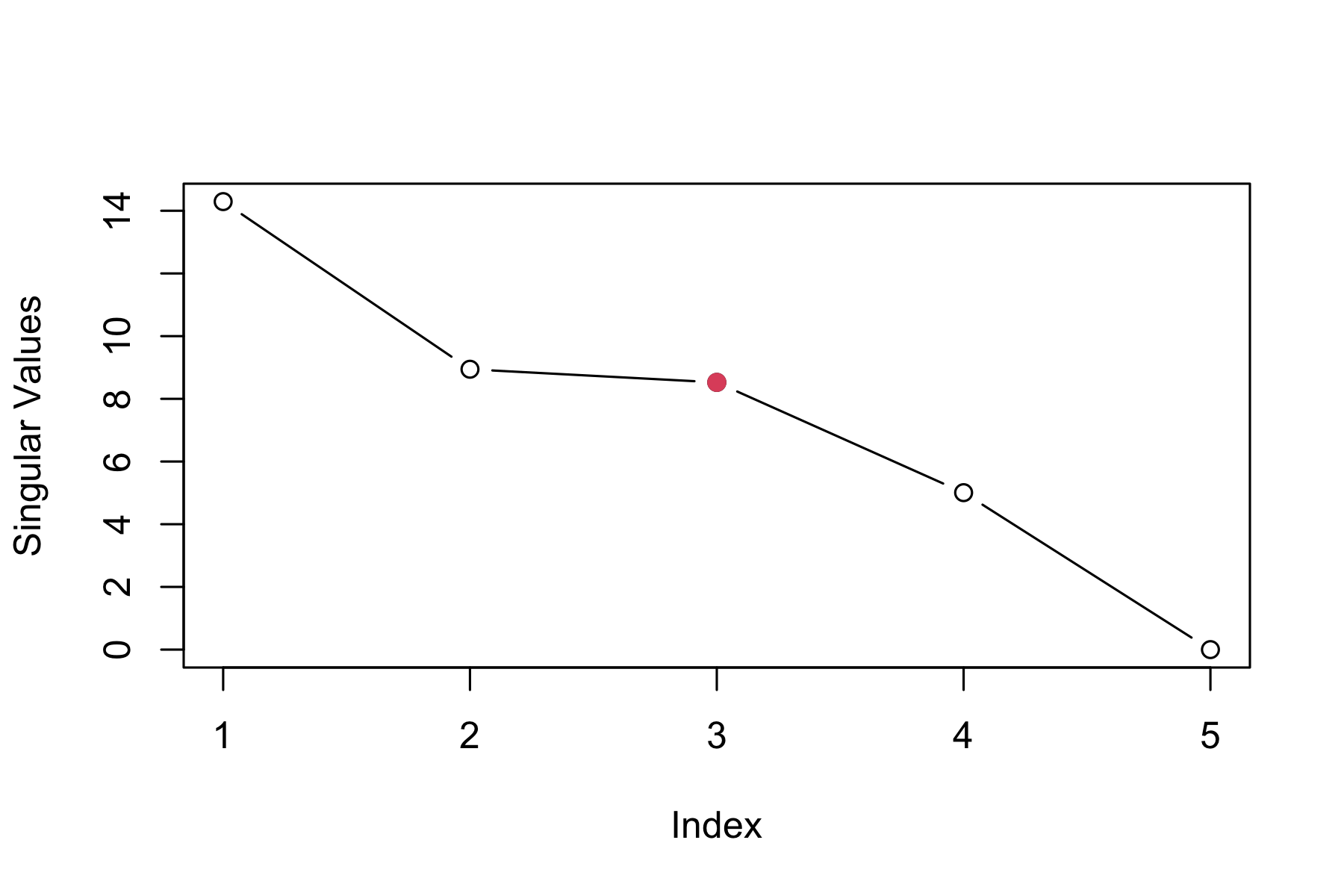}
        \vspace{3pt}
        \centerline{\small (b) MASE}
        \label{fig:scree_mase_exp}
    \end{minipage}
    
    \caption{Eigenvalue scree plot for two-year exporters only; the red dot indicates the selected dimension.}
    \label{fig:scree_exp}
\end{figure}

\begin{figure}
    \centering
    
    \begin{minipage}[b]{0.48\linewidth}
        \centering
        \includegraphics[width=\linewidth]{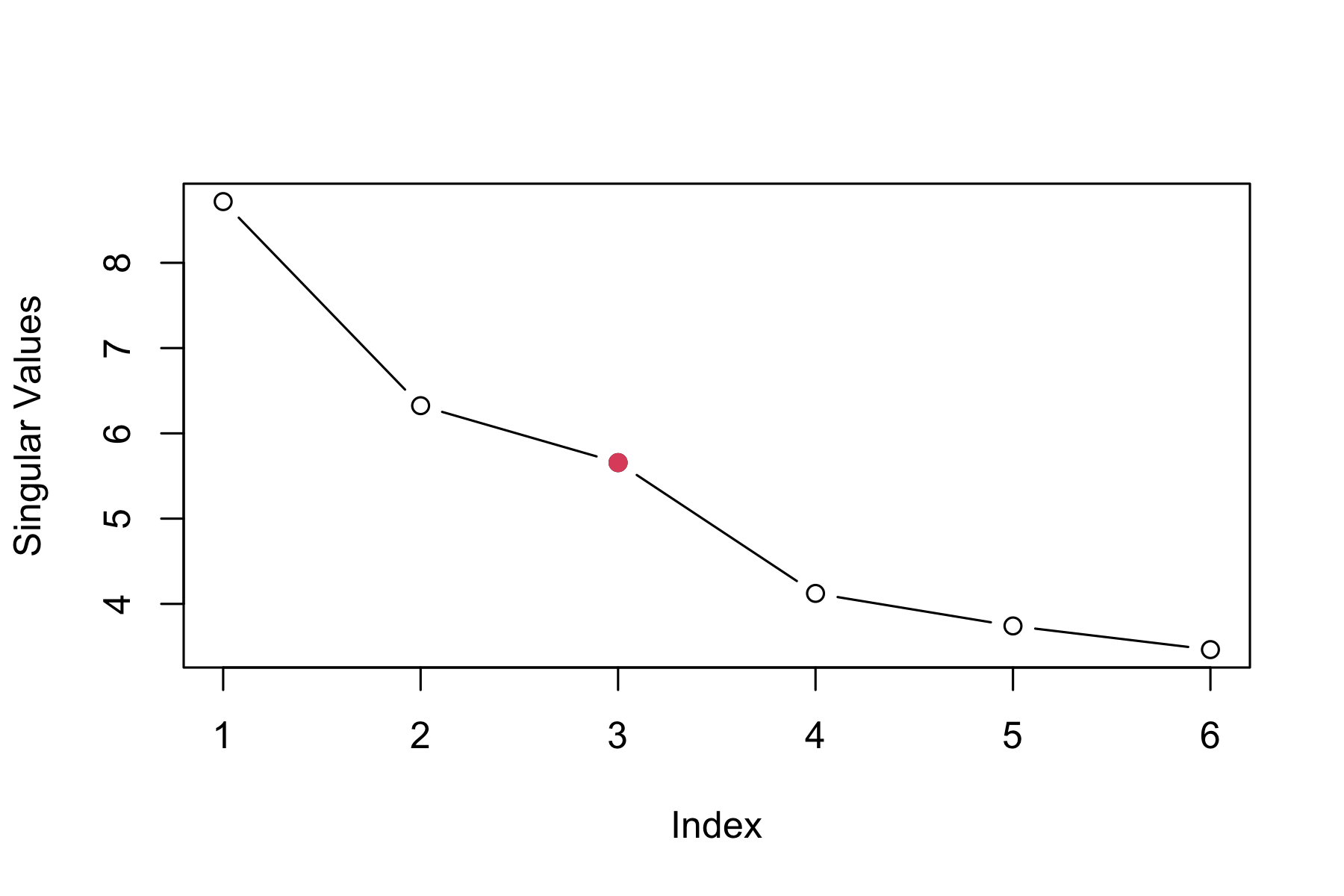}
        \vspace{3pt}
        \centerline{\small (a) KRAFTY}
        \label{fig:scree_kr_imp}
    \end{minipage}
    \hfill 
    \begin{minipage}[b]{0.48\linewidth}
        \centering
        \includegraphics[width=\linewidth]{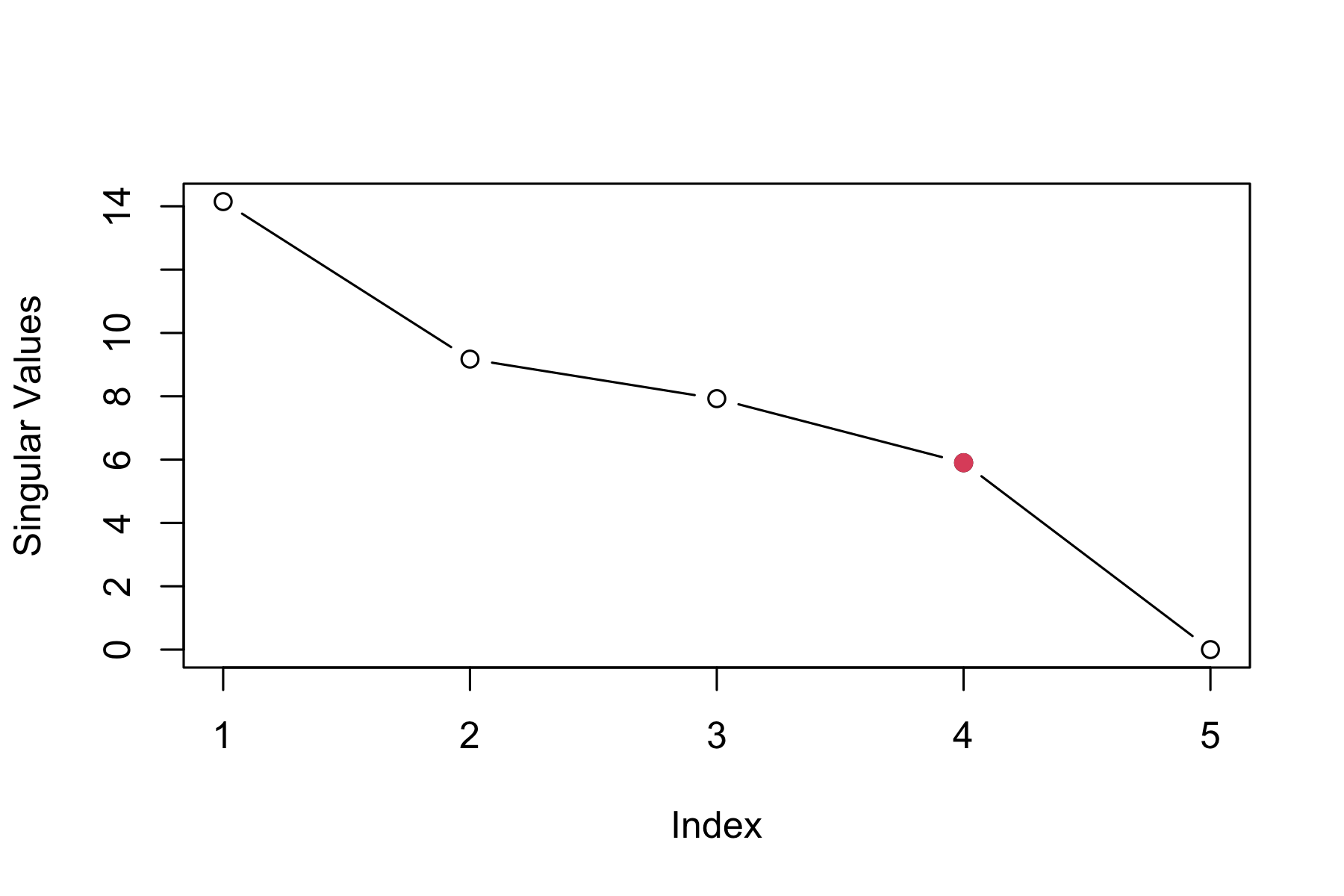}
        \vspace{3pt}
        \centerline{\small (b) MASE}
        \label{fig:scree_mase_imp}
    \end{minipage}
    
    \caption{Eigenvalue scree plot for two-year importers only; the red dot indicates the selected dimension.}
    \label{fig:scree_imp}
\end{figure}

\begin{figure}
    \centering
    \includegraphics[width=0.6\linewidth]{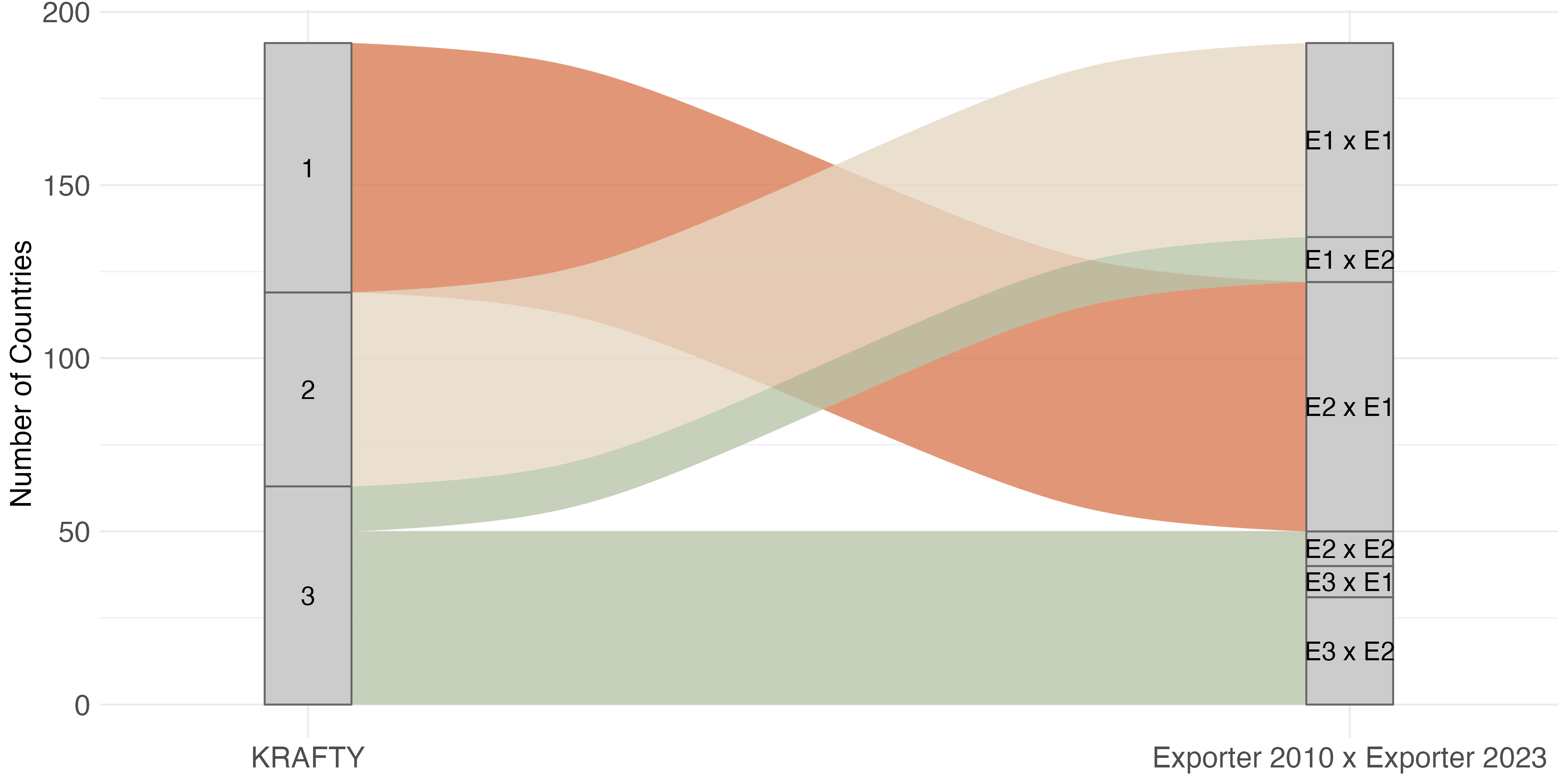}
    \caption{Alluvial plot of joint clusters by KRAFTY vs. single-sourced clusters (exporters); identical colors indicate the same set of entities, and block height reflects the number of countries in each cluster.}
    \label{fig:alluvial_exp}
\end{figure}

\begin{figure}
    \centering
    \includegraphics[width=0.6\linewidth]{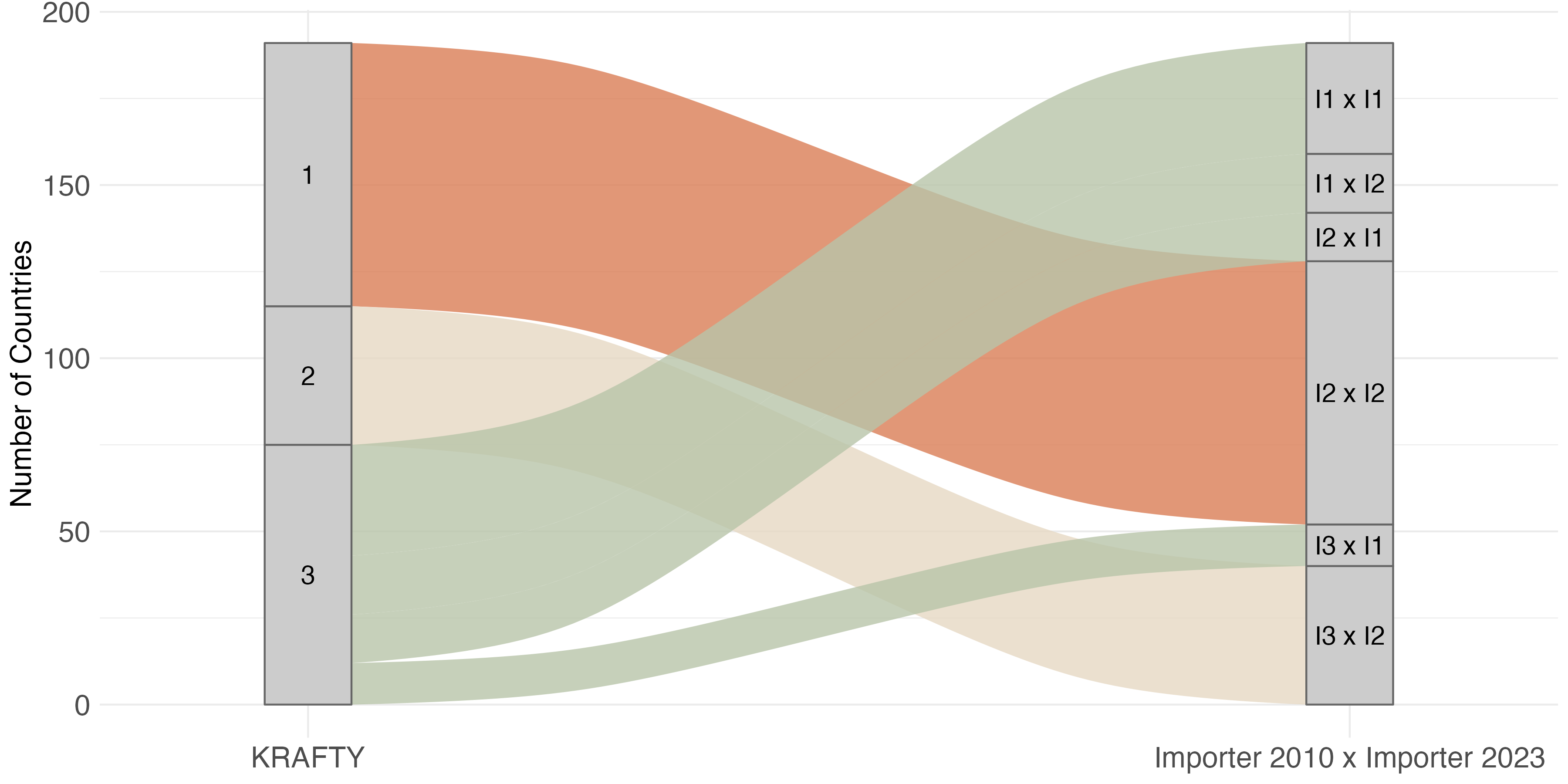}
    \caption{Alluvial plot of joint clusters by KRAFTY vs single-sourced clusters (importers); identical colors indicate the same set of entities, and block height reflects the number of countries in each cluster.}
    \label{fig:alluvial_imp}
\end{figure}

\begin{figure}
    \centering
    \includegraphics[width=0.6\linewidth]{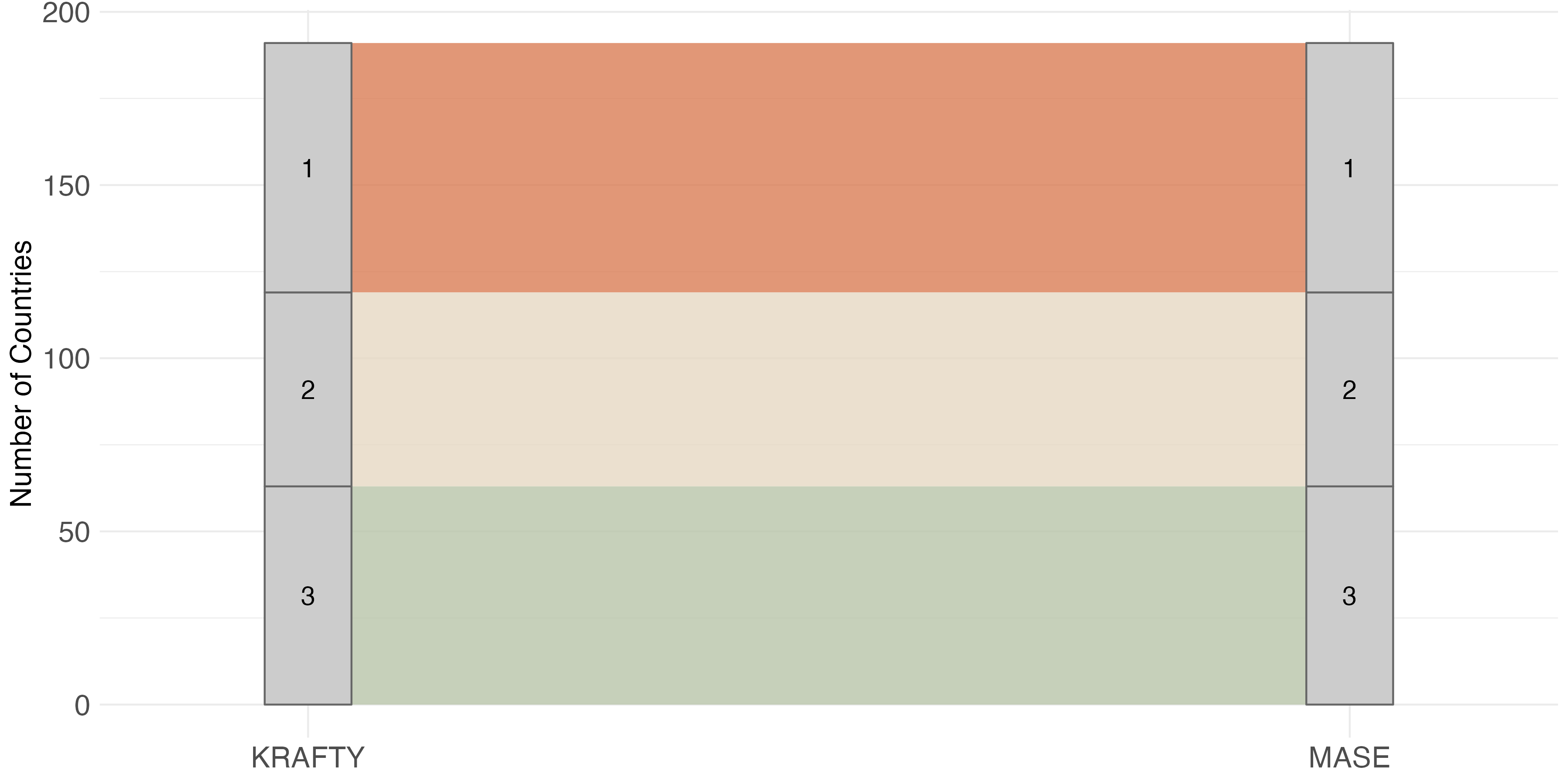}
    \caption{Alluvial plot of Joint Clusters by KRAFTY vs by MASE (Exporters); identical colors indicate the same set of entities, and block height reflects the number of countries in each cluster.}
    \label{fig:alluvial_2yr_exp_kr_mase}
\end{figure}

\begin{figure}
    \centering
    \includegraphics[width=0.6\linewidth]{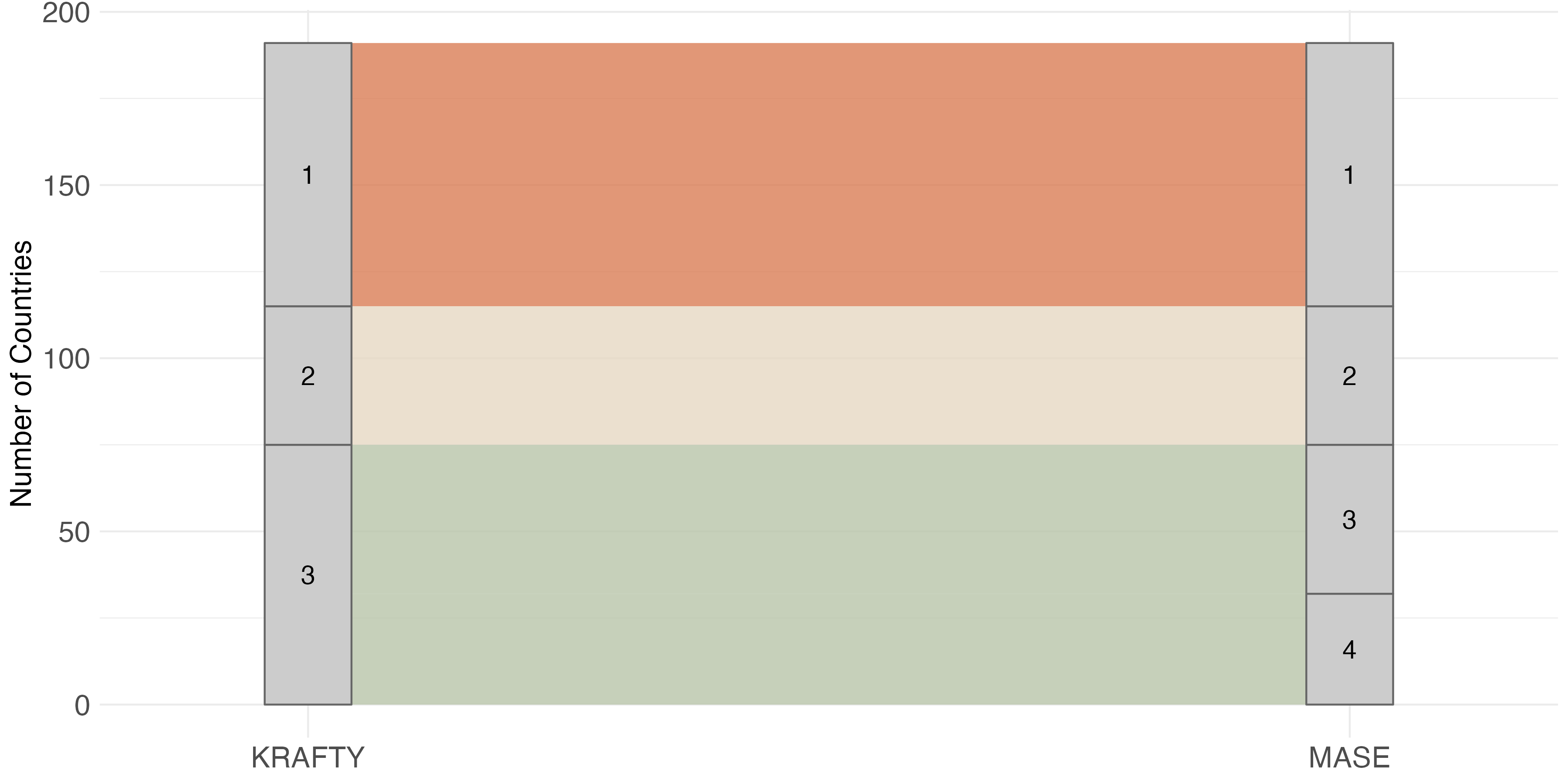}
    \caption{Alluvial plot of joint clusters by KRAFTY vs by MASE (importers); identical colors indicate the same set of entities, and block height reflects the number of countries in each cluster.}
    \label{fig:alluvial_2yr_imp_kr_mase}
\end{figure}

\begin{figure}
    \centering
    \includegraphics[width=0.8\linewidth]{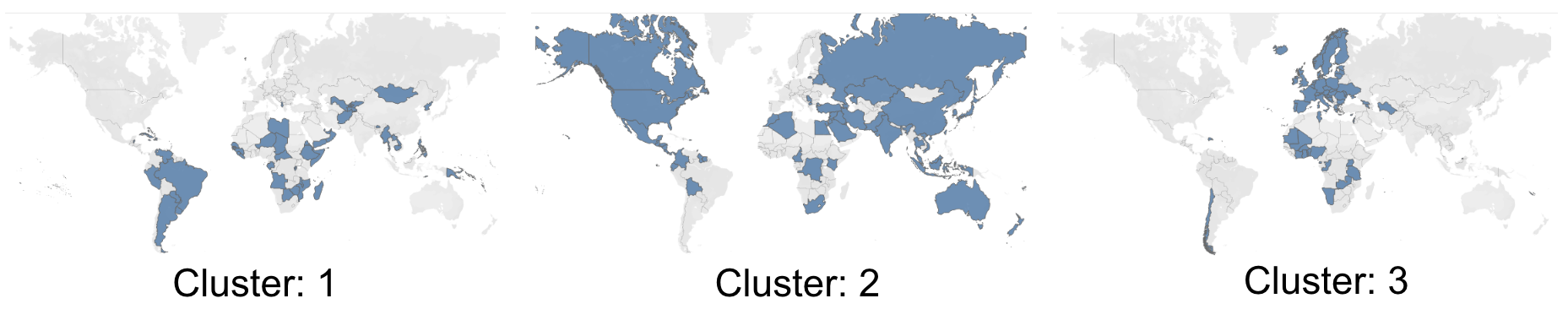}
    \caption{World map colored by KRAFTY joint clusters for two-year exporters; blue color denotes cluster membership.}
    \label{fig:map_2yr_exp}
\end{figure}

\begin{figure}
    \centering
    \includegraphics[width=0.8\linewidth]{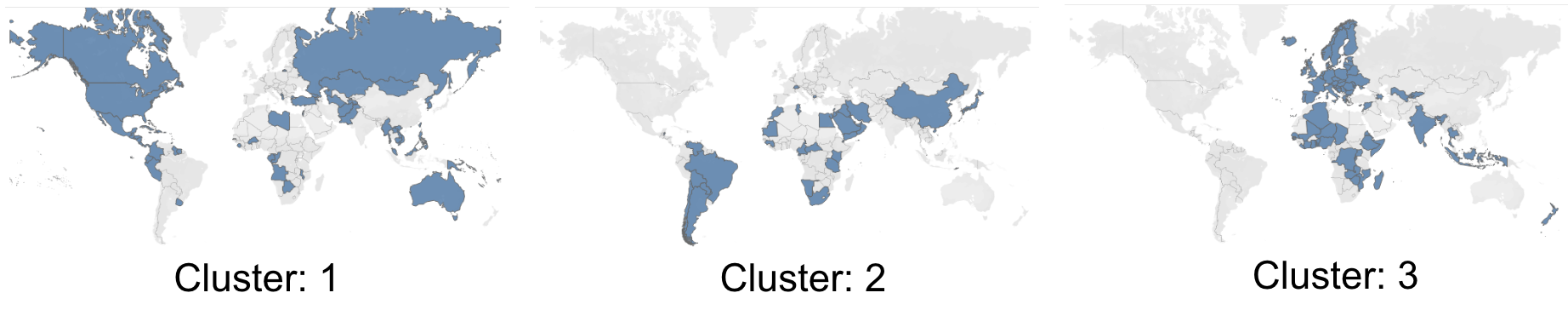}
    \caption{World map colored by KRAFTY joint clusters for two-year importers; blue color denotes cluster membership.}
    \label{fig:map_2yr_imp}
\end{figure}



\end{document}